\documentclass[11pt]{article}

\usepackage[sort,numbers]{natbib}
\usepackage[usenames]{color}
\usepackage[bookmarks=true,bookmarksnumbered=true,colorlinks=true,allcolors=Blue]{hyperref}

\usepackage{amsmath,amssymb,amsfonts,amsthm}
\usepackage{lscape}
\usepackage{arydshln}
\renewcommand*{\baselinestretch}{1.25}

\usepackage{geometry}
\usepackage{indentfirst}

\usepackage{enumitem}

\newtheorem{theorem}{Theorem}[section]
\newtheorem{lemma}{Lemma}[section]
\newtheorem{proposition}{Proposition}[section]
\newtheorem{corollary}{Corollary}[section]
\theoremstyle{definition}
\newtheorem{definition}{Definition}[section]
\newtheorem*{rmk*}{Remark}
\newtheorem{rmk}{Remark}[section]

\renewcommand*\proofname{\upshape{\bfseries{Proof}}}

\DeclareMathOperator{\variance}{Var}
\DeclareMathOperator{\covariance}{Cov}

\DeclareMathOperator{\trace}{tr}

\DeclareMathOperator{\opL}{L}
\DeclareMathOperator{\kernel}{Ker}
\DeclareMathOperator{\id}{Id}
\DeclareMathOperator{\influence}{Inf}
\DeclareMathOperator{\expectation}{E}

\newcommand{\wtimes}{\mathbin{\widetilde{\otimes}}}
\newcommand{\labs}{\left|}
\newcommand{\rabs}{\right|}
\newcommand{\lpa}{\left(}
\newcommand{\rpa}{\right)}

\newcommand{\bs}[1]{\boldsymbol{#1}}
\newcommand{\ol}[1]{\overline{#1}}
\newcommand{\ul}[1]{\underline{#1}}
\newcommand{\wh}[1]{\widehat{#1}}
\newcommand{\wt}[1]{\widetilde{#1}}
\newcommand{\wstar}[1]{\mathbin{\widehat{\star_{#1}^0}}}
\newcommand{\ex}[1]{\expectation\left[#1\right]}
\newcommand{\expe}[1]{\mathrm{E}[#1]}
\newcommand{\mf}[1]{\mathfrak{#1}}
\newcommand{\mc}[1]{\mathcal{#1}}

\allowdisplaybreaks

\numberwithin{equation}{section}

\makeatletter
    \renewcommand*{\section}{\@startsection{section}{1}{\z@}%
    {6pt}{3pt}{\reset@font\normalsize\bfseries}}
\makeatother
    
\makeatletter
    \renewcommand*{\subsection}{\@startsection{subsection}{2}{\z@}%
    {3pt}{3pt}{\reset@font\normalsize\mdseries\itshape}}
\makeatother

\makeatletter
    \renewcommand*{\subsubsection}{\@startsection{subsubsection}{3}{\z@}%
    {3pt}{3pt}{\reset@font\normalsize\mdseries\itshape}}
\makeatother

\makeatletter
\def\@listi{\leftmargin\leftmargini
  \topsep=.5\baselineskip 
  \partopsep=0pt \parsep=0pt \itemsep=0pt}
\let\@listI\@listi
\@listi
\def\@listii{\leftmargin\leftmarginii
  \labelwidth\leftmarginii \advance\labelwidth-\labelsep
  \topsep=0pt \partopsep=0pt \parsep=0pt \itemsep=0pt}
\def\@listiii{\leftmargin\leftmarginiii
  \labelwidth\leftmarginiii \advance\labelwidth-\labelsep
  \topsep=0pt \partopsep=0pt \parsep=0pt \itemsep=0pt}
\def\@listiv{\leftmargin\leftmarginiv
  \labelwidth\leftmarginiv \advance\labelwidth-\labelsep
  \topsep=0pt \partopsep=0pt \parsep=0pt \itemsep=0pt}
\makeatother

\makeatletter
\newcommand{\opnorm}{\@ifstar\@opnorms\@opnorm}
\newcommand{\@opnorms}[1]{%
  \left|\mkern-1.5mu\left|\mkern-1.5mu\left|
   #1
  \right|\mkern-1.5mu\right|\mkern-1.5mu\right|
}
\newcommand{\@opnorm}[2][]{%
  \mathopen{#1|\mkern-1.5mu#1|\mkern-1.5mu#1|}
  #2
  \mathclose{#1|\mkern-1.5mu#1|\mkern-1.5mu#1|}
}
\makeatother

\makeatletter
\renewenvironment{proof}[1][\proofname]{\par
  \pushQED{\qed}%
  \normalfont \topsep6\p@\@plus6\p@\relax
  \trivlist
  \item[\hskip\labelsep
        \bfseries
    #1\@addpunct{.}]\ignorespaces
}{%
  \popQED\endtrivlist\@endpefalse
}
\makeatother

\title{High-dimensional central limit theorems for homogeneous sums}
\author{Yuta Koike\\
\textit{University of Tokyo and CREST JST}
}

\begin{document}

\maketitle

\begin{abstract}

This paper develops a quantitative version of de Jong's central limit theorem for homogeneous sums in a high-dimensional setting. 
More precisely, under appropriate moment assumptions, we establish an upper bound for the Kolmogorov distance between a multi-dimensional vector of homogeneous sums and a Gaussian vector so that the bound depends polynomially on the logarithm of the dimension and is governed by the fourth cumulants and the maximal influences of the components.  
As a corollary, we obtain high-dimensional versions of fourth moment theorems, universality results and Peccati-Tudor type theorems for homogeneous sums. 
We also sharpen some existing (quantitative) central limit theorems by applications of our result. 
\vspace{3mm}


\end{abstract}

\section{Introduction}

Let $\bs{X}=(X_i)_{i=1}^\infty$ be a sequence of independent centered random variables with unit variance. 
A \textit{homogeneous sum} is a random variable of the form
\[
Q(f;\bs{X})=\sum_{i_1,\dots,i_p=1}^Nf(i_1,\dots,i_q)X_{i_1}\cdots X_{i_q},
\]
where $N,q\in\mathbb{N}$, $[N]:=\{1,\dots,N\}$ and $f:[N]^q\to\mathbb{R}$ is a symmetric function vanishing on diagonals, i.e.~$f(i_1,\dots,i_q)=0$ unless $i_1,\dots,i_q$ are mutually different. 
Studies of limit theorems for a sequence of homogeneous sums have some history in probability theory. 
\citet{Rotar1975,Rotar1979} investigated invariance principles for $Q(f;\bs X)$ regarding the law of $\bs X$. 
In the notable work of \citet{deJong1990}, the following striking result has been established: 
For every $n\in\mathbb{N}$, let $f_n:[N_n]^q\to\mathbb{R}$ be a symmetric function vanishing on diagonals with $q$ being fixed and $N_n\uparrow\infty$ as $n\to\infty$. Assume that $\expe{X_i^4}<\infty$ for all $i$ and $\expe{Q(f_n;\bs{X})^2}=1$ for all $n$. Then, $Q(f_n;\bs{X})$ converges in law to the standard normal distribution, provided that the following two conditions hold true:
\begin{enumerate}[label=(\roman*)]

\item\label{cond:fourth} $\expe{Q(f_n;\bs{X})^4}\to3$ as $n\to\infty$.

\item\label{cond:influence} $\max_{1\leq i\leq N_n}\influence_i(f_n)\to0$ as $n\to\infty$, where the quantity $\influence_i(f_n)$ is defined by
\begin{equation}\label{def:influence}
\influence_i(f_n):=\sum_{i_2,\dots,i_q=1}^{N_n}f_n(i,i_2,\dots,i_q)^2
\end{equation}
and called the \textit{influence} of the $i$th variable. 

\end{enumerate}
When $q=1$, condition \ref{cond:influence} says that $\max_{1\leq i\leq N_n}f_n(i)^2\to0$ as $n\to\infty$, which is equivalent to the celebrated \textit{Lindeberg condition}. In this case condition \ref{cond:fourth} is always implied by \ref{cond:influence}, and thus it is an extra one. 
In contrast, when $q\geq2$, condition \ref{cond:influence} is no longer sufficient for the asymptotic normality of the sequence $(Q(f_n;\bs{X}))_{n=1}^\infty$, so one needs an additional condition. 
The motivation of introducing condition \ref{cond:fourth} in \cite{deJong1990} was that one can easily check condition \ref{cond:fourth} is equivalent to the asymptotic normality of $(Q(f_n;\bs{X}))_{n=1}^\infty$ when $q=2$ and $\bs{X}$ is Gaussian (see also \cite{deJong1987}). 
Later on, this observation was significantly improved in the influential paper by \citet{NP2005}: For any $q$, the asymptotic normality of $(Q(f_n;\bs{X}))_{n=1}^\infty$ is implied just by condition \ref{cond:fourth} as long as $\bs{X}$ is Gaussian. Results of this type are nowadays called \textit{fourth moment theorems} and have been extensively studied in the past decade. 
In particular, further investigation of the fourth moment theorem in \cite{NP2005} has led to the introduction of the so-called \textit{Malliavin-Stein method} by \citet{NP2009PTRF}, which have produced one of the most active research areas in the recent probabilistic literature. 
We refer the reader to the monograph \cite{NP2012} for an introduction to this subject and the survey \cite{AP2018} for recent developments. 

Implication of the Malliavin-Stein method to de Jong's central limit theorem (CLT) for homogeneous sums has been investigated in the seminal work of \citet*{NPR2010aop}, where several important extensions of de Jong's result have been developed.  
The following three results are particularly relevant to our work:
\begin{enumerate}[label=(\Roman*)]

\item\label{intro:multi} First, they have established a multi-dimensional extension of de Jong's CLT which shows  multi-dimensional vectors of homogeneous sums enjoy a CLT if de Jong's criterion is satisfied component-wise. 
More precisely, let $d\in\mathbb{N}$ and, for every $j=1,\dots,d$, let $q_j\in\mathbb{N}$ and $f_{n,j}:[N_n]^{q_j}\to\mathbb{R}$ be a symmetric function vanishing on diagonals. Also, let $\mf{C}=(\mf{C}_{jk})_{1\leq j,k\leq d}$ be a $d\times d$ positive semidefinite symmetric matrix and suppose that $\max_{1\leq j,k\leq d}|\expe{Q(f_{n,j};\bs{X})Q(f_{n,k};\bs{X})}-\mf{C}_{jk}|\to0$ as $n\to\infty$. Then, the $d$-dimensional random vector $\bs{Q}^{(n)}(\bs{X}):=(Q(f_{n,1};\bs{X}),\dots,Q(f_{n,d};\bs{X}))$ converges in law to the $d$-dimensional normal distribution $\mathcal{N}_d(0,\mf{C})$ with mean 0 and covariance matrix $\mf{C}$ as $n\to\infty$ if $\expe{Q(f_{n,j};\bs{X})^4}-3\expe{Q(f_{n,j};\bs{X})^2}^2\to0$ and $\max_{1\leq i\leq N_n}\influence_i(f_{n,j})\to0$ as $n\to\infty$ for every $j=1,\dots,d$. 

\item\label{intro:univ} Second, they have found the following \textit{universality} of Gaussian variables in the context of homogeneous sums (\cite[Theorem 1.2]{NPR2010aop}): Assume $\sup_n\sum_{i_1,\dots,i_q=1}^{N_n}f_{n,j}(i_1,i_2,\dots,i_{q_j})^2<\infty$ and $\mf{C}_{jj}>0$ for every $j$. 
Then, if $\bs{Q}^{(n)}(\bs{G})$ converges in law to $\mathcal{N}_d(0,\mf{C})$ as $n\to\infty$ for a sequence of standard Gaussian variables $\bs{G}=(G_i)_{i=1}^\infty$, then $\bs{Q}^{(n)}(\bs{X})$ converges in law to $\mathcal{N}_d(0,\mf{C})$ as $n\to\infty$ for any sequence $\bs{X}=(X_i)_{i=1}^\infty$ of independent centered random variables with unit variance and such that $\sup_i\expe{|X_i|^3}<\infty$.  

\item\label{intro:quant} Third, they have established some quantitative versions of de Jong's CLT for homogeneous sums; see Proposition 5.4 and Corollary 7.3 in \cite{NPR2010aop} for details (see also Section \ref{sec:comp-npr}). 

\end{enumerate}
We remark that these results have been generalized in various directions by subsequent studies. 
For example, the universality results analogous to \ref{intro:univ} have also been established for Poisson variables in \citet{PZ2014} and i.i.d.~variables with zero skewness and non-negative excess kurtosis in \citet{NPPS2016,NPPS2016esaim}, respectively. 
Also, the recent work of \citet{DP2017ejp} has extended \ref{intro:multi} and \ref{intro:quant} to more general degenerate $U$-statistics which were originally treated in \cite{deJong1990}. 

As the title of the paper suggests, the aim of this paper is to extend the above results to a high-dimensional setting where the dimension $d$ depends on $n$ and $d=d_n\to\infty$ as $n\to\infty$. 
Of course, in such a setting, the ``asymptotic distribution'' $\mc{N}_d(0,\mf{C})$ also depends on $n$ and, even worse, it is typically no longer tight. Therefore, we need to properly reformulate the above statements in this setting. In this paper we adopt the so-called \textit{metric approach} to accomplish this purpose: We try to establish the convergence of some metric between the laws of $\bs{Q}^{(n)}(\bs{X})$ and $\mc{N}_d(0,\mf{C})$. Specifically, we take the \textit{Kolmogorov distance} as the metric between the probability laws. Namely, letting $Z^{(n)}$ be a $d_n$-dimensional centered Gaussian vector with covariance matrix $\mf{C}_n$ for each $n$, we aim at proving the following convergence:
\[
\sup_{x\in\mathbb{R}^{d_n}}|P(\bs{Q}^{(n)}(\bs{X})\leq x)-P(Z^{(n)}\leq x)|\to0\quad\text{as }n\to\infty.
\]
Here, for vectors $x=(x_1,\dots,x_{d_n})\in\mathbb{R}^{d_n}$ and $y=(y_1,\dots,y_{d_n})\in\mathbb{R}^{d_n}$, we write $x\leq y$ to express $x_j\leq y_j$ for every $j=1,\dots,d_n$. 
In addition, we are particularly interested in a situation where the dimension $d=d_n$ increases extremely faster than the ``standard'' convergence rate of Gaussian approximation for a sequence of univariate homogeneous sums. 
Given that both $\sqrt{|\expe{Q(f_n;\bs{X})^4}-3\expe{Q(f_{n};\bs{X})^2}^2|}$ and $\max_{1\leq i\leq N_n}\sqrt{\influence_i(f_n)}$ can be the optimal convergence rates of the Gaussian approximation of $Q(f_{n};\bs{X})$ in the Kolmogorov distance (see \cite[Proposition 3.8]{NP2009aop} for the former and \cite[Remark 1]{GT1999} for the latter), we might consider the quantity
\[
\delta_n:=\max_{1\leq j\leq d_n}\sqrt{|\expe{Q(f_{n,j};\bs{X})^4}-3\expe{Q(f_{n,j};\bs{X})^2}^2|+\max_{1\leq i\leq N_n}\influence_i(f_{n,j})}
\]
as an appropriate definition of the ``standard'' convergence rate. 
Then, we aim at proving
\begin{equation}\label{aim}
\sup_{x\in\mathbb{R}^{d_n}}|P(\bs{Q}^{(n)}(\bs{X})\leq x)-P(Z^{(n)}\leq x)|\leq C(\log d_n)^a\delta_n^b
\end{equation}
for all $n\in\mathbb{N}$, where $a,b,C>0$ are constants which do not depend on $n$ (here and below we assume $d_n\geq2$). 
As a byproduct, results of this type enable us to extend fourth moment theorems and universality results for homogeneous sums to a high-dimensional setting (see Theorem \ref{thm:universality} for the precise statement). 

Our formulation of a high-dimensional extension of CLTs for homogeneous sums is motivated by the recent path-breaking work of \citet*{CCK2013,CCK2017}, where results analogous to \eqref{aim} have been established for sums of independent random vectors. 
More formally, let $(\xi_{n,i})_{i=1}^n$ be a sequence of independent centered $d_n$-dimensional random vectors. 
Set $S_n:=n^{-1/2}\sum_{i=1}^n\xi_{n,i}$ and assume $\mf{C}_n=\expe{S_nS_n^\top}$ ($\top$ denotes the transpose of a matrix). 
Then, under an appropriate assumption on moments, we have
\begin{equation}\label{eq:cck}
\sup_{x\in\mathbb{R}^{d_n}}|P(S_n\leq x)-P(Z^{(n)}\leq x)|\leq C'\left(\frac{\log^7(d_nn)}{n}\right)^{1/6},
\end{equation}
where $C'>0$ is a constant which does not depend on $n$ (see Proposition \ref{prop:cck} for the precise statement). 
Here, we shall remark that the bound in \eqref{eq:cck} depends on $n$ through $n^{-1/6}$, which is suboptimal when the dimension $d_n$ is fixed. However, in \cite[Remark 2.1(ii)]{CCK2017} it is conjectured that the rate $n^{-1/6}$ is nearly optimal in a minimax sense when $d_n$ is extremely larger than $n$ (see also \cite[Remark 1]{Chen2017}). 
This conjecture is motivated by the fact that the rate $n^{-1/6}$ is minimax optimal in CLTs for sums of independent random variables taking values in an infinite-dimensional Banach space (see e.g.~\cite[Theorem 2.6]{BGPR2000}). 
Given that high-dimensional CLTs of type \eqref{eq:cck} are closely related to Gaussian approximation of the suprema of empirical processes (see e.g.~\cite{CCK2014,CCK2016}), it would be worth mentioning that a duality argument enables us to translate the minimax rate for CLTs in a Banach space to the one for Gaussian approximation of the suprema of empirical processes with a specific class of functions in the Kolmogorov distance; see \cite{Paulauskas1992} for details. 
For this reason we also conjecture that $b=1/3$ would give an optimal dependence on $\delta_n$ of the bound in \eqref{aim} (note that the rate $n^{-1/2}$ is the standard convergence rate of CLTs for sums of independent one-dimensional random variables). In this paper we indeed establish that the bound of type \eqref{aim} holds true with $b=1/3$ under a moment assumption on $\bs{X}$ when $q_j$'s do not depend on $j$ (see Theorem \ref{thm:main} and Remark \ref{rmk:main}). 

We remark that there are a number of articles which extend the scope of the Chernozhukov-Chetverikov-Kato theory (CCK theory for short) in various directions. 
We refer the reader to the survey \cite{BCCHK2018} for recent developments. 
Nevertheless, most studies focus on \textit{linear} statistics (i.e.~sums of random variables) and there are only a few articles concerned with \textit{non-linear} statistics. 
Two exceptions are $U$-statistics developed in \cite{Chen2017,CK2017,CK2017rand,SCK2019} and Wiener functionals developed in \cite{Koike2017stein,Koike2018sk}. 
On the one hand, however, the former are mainly concerned with non-degenerate $U$-statistics which are approximately linear statistics via Hoeffding decomposition (\citet{CK2017rand} also handle degenerate $U$-statistics, but they focus on the randomized incomplete versions that are still approximately linear statistics). 
On the other hand, although the latter deal with essentially non-linear statistics, they must be functionals of a (possibly infinite-dimensional) Gaussian process, except for \cite[Theorem 3.2]{Koike2017stein} that is a version of our result with $q_j\equiv2$ (see Section \ref{sec:comp-stein} for more details). 
In this sense, our result would be the first extension of CCK type results to essentially non-linear statistics based on possibly non-Gaussian variables. 

Finally, we remark that the main results of this paper have potential applications to statistics. In fact, the original motivation of this paper is to improve the Gaussian approximation result for maxima of high-dimensional vectors of random quadratic forms given by \cite[Theorem 3.2]{Koike2017stein}, which is used to ensure the validity of the bootstrap testing procedure proposed in \cite[Section 4.1]{Koike2017stein} (see Section \ref{sec:hry}). 
Another potential application might be specification test for parametric form in nonparametric regression. In this area, to derive the null distributions of test statistics, one sometimes needs to approximate the maximum of (essentially degenerate) quadratic forms; see \cite{HS2001,DH2007,LSC2017} for instance. 

This paper is organized as follows. 
Section \ref{sec:main} presents the main results obtained in the paper, while Sections \ref{sec:cck}--\ref{sec:proof} are devoted to the proof of the main results: Section \ref{sec:cck} demonstrates a basic scheme of the CCK theory to prove high-dimensional CLTs. 
Subsequently, Section \ref{sec:stein} presents a connection of this scheme to Stein's method. 
Based on this observation, Section \ref{sec:normal-gamma} develops a high-dimensional CLT of the form \eqref{aim} for homogeneous sums based on normal and gamma variables. 
Then, Section \ref{sec:lindeberg} establishes a kind of invariance principle for high-dimensional homogeneous sums using a randomized version of the Lindeberg method. 
Finally, Section \ref{sec:proof} completes the proof of the main results. 

\section*{Notation}

$\mathbb{Z}_+$ denotes the set of all non-negative integers. 
For $x=(x_1,\dots,x_d)\in\mathbb{R}^d$, we define $\|x\|_{\ell_\infty}:=\max_{1\leq j\leq d}|x_j|$. 
For $N\in\mathbb{N}$, we set $[N]:=\{1,\dots,N\}$. 
We set $\sum_{i=p}^q\equiv0$ if $p>q$ by convention. 
For $q\in\mathbb{N}$, we denote by $\mf{S}_q$ the set of all permutations of $[q]$, i.e.~the symmetric group of degree $q$. 
For a function $f:[N]^q\to\mathbb{R}$, we set $\mathcal{M}(f):=\max_{1\leq i\leq N}\influence_i(f)$ (recall that $\influence_i(f)$ is defined according to \eqref{def:influence}). We also set
\[
\|f\|_{\ell_2}:=\sqrt{\sum_{i_1,\dots,i_q=1}^Nf(i_1,\dots,i_q)^2}.
\]
For a function $h:\mathbb{R}^d\to\mathbb{R}$, we set $\|h\|_\infty:=\sup_{x\in\mathbb{R}^d}|h(x)|$. 
We write $C^m_b(\mathbb{R}^d)$ for the set of all real-valued $C^m$ functions on $\mathbb{R}^d$ all of whose partial derivatives are bounded. 
We write $\partial_{j_1\dots j_m}=\frac{\partial^m}{\partial x_{j_1}\cdots\partial x_{j_m}}$ for short. 
Throughout the paper, $Z=(Z_1,\dots,Z_d)$ denotes a $d$-dimensional centered Gaussian random vector with  covariance matrix $\mathfrak{C}=(\mathfrak{C}_{ij})_{1\leq i,j\leq d}$ (note that we \textit{do  not} assume that $\mf{C}$ is positive definite in general). 
Also, $(q_j)_{j=1}^\infty$ stands for a sequence of positive integers. Throughout the paper, we will regard $(q_j)_{j=1}^\infty$ as fixed, i.e.~it does not vary when we consider asymptotic results. 
Given a probability distribution $\mu$, we write $X\sim\mu$ to express that $X$ is a random variable with distribution $\mu$. 
For $\nu>0$, we write $\gamma(\nu)$ for the gamma distribution with shape $\nu$ and rate 1. 
If $S$ is a topological space, $\mc{B}(S)$ denotes the Borel $\sigma$-field of $S$. 

Given a random variable $X$, we set $\|X\|_p:=\{\expe{|X|^p}\}^{1/p}$ for every $p>0$. 
When $X$ satisfies $\expe{X^4}<\infty$, we denote the fourth cumulant of $X$ by $\kappa_4(X)$. Note that $\kappa_4(X)=\expe{X^4}-3\expe{X^2}^2$ if $X$ is centered. 
For $\alpha>0$, we define the $\psi_\alpha$-norm of $X$ by
\begin{equation*}
\|X\|_{\psi_\alpha}:=\inf\{C>0:\expe{\psi_\alpha(|X|/C)}\leq1\},
\end{equation*}
where $\psi_\alpha(x):=\exp(x^\alpha)-1$. Note that $\|\cdot\|_{\psi_\alpha}$ is indeed a norm (on a suitable space) if and only if $\alpha\geq1$. Some useful properties of the $\psi_\alpha$-norm are collected in Appendix \ref{sec:psi}. 

\section{Main results}\label{sec:main}

Our first main result is a high-dimensional version of de Jong's CLT for homogeneous sums:
\begin{theorem}\label{thm:main}
Let $\bs{X}=(X_i)_{i=1}^N$ be a sequence of independent centered random variables with unit variance. Set $w=\frac{1}{2}$ if $\expe{X_i^3}=0$ for every $i\in[N]$ and $w=1$ otherwise. 
For every $j\in[d]$, let $f_j:[N]^{q_j}\to\mathbb{R}$ be a symmetric function vanishing on diagonals, and set $\bs{Q}(\bs{X}):=(Q(f_1;\bs{X}),\dots,Q(f_d;\bs{X}))$. 
Suppose that $d\geq2$, $\ul{\sigma}:=\min_{1\leq j\leq d}\|Z_j\|_2>0$ and $\ol{B}_N:=\max_{1\leq i\leq N}\|X_i\|_{\psi_\alpha}<\infty$ for some $\alpha\in(0,w^{-1}]$. Then we have
\begin{align}
&\sup_{x\in\mathbb{R}^d}\left|P(\bs{Q}(\bs{X})\leq x)-P(Z\leq x)\right|\nonumber\\
&\leq C(1+\ul{\sigma}^{-1})\left\{(\log d)^{\frac{2}{3}}\delta_0[\bs{Q}(\bs{X})]^{\frac{1}{3}}
+(\log d)^{\mu+\frac{1}{2}}\delta_1[\bs{Q}(\bs{X})]^{\frac{1}{3}}
+(\log d)^{\frac{2\ol{q}_d-1}{\alpha}+\frac{3}{2}}\max_{1\leq k\leq d}\ol{B}_N^{q_k}\sqrt{\mathcal{M}(f_k)}\right\},\label{eq:main}
\end{align}
where $\ol{q}_d:=\max_{1\leq j\leq d}q_j$, $\mu:=\max\{\frac{2}{3}w\ol{q}_d-\frac{1}{6},\frac{2(\ol{q}_d-1)}{3\alpha}+\frac{1}{3}\}$, $C>0$ depends only on $\alpha,\ol{q}_d$ and 
\begin{align*}
\delta_0[\bs{Q}(\bs{X})]&:=\max_{1\leq j,k\leq d}\left|\expe{Q(f_j;\bs{X})Q(f_k;\bs{X})}-\mathfrak{C}_{jk}\right|,\\
\delta_1[\bs{Q}(\bs{X})]&:= \ol{A}_N^{2w\ol{q}_d-1}
\max_{1\leq j,k\leq d}\left\{1_{\{q_j=q_k\}}\sqrt{|\kappa_4(Q(f_k;\bs{X}))|+\ol{A}_N^{4q_k}\sum_{i=1}^N\influence_i(f_k)^2}\right.\\
&\hphantom{:= \ol{A}_N^{2\ol{q}_d-1}
\max_{1\leq j,k\leq d}}\left.+1_{\{q_j< q_k\}}\ol{A}_N^{q_j}\|f_j\|_{\ell_2}\left(|\kappa_4(Q(f_k;\bs{X}))|+\ol{A}_N^{4q_k}\sum_{i=1}^N\influence_i(f_k)^2\right)^{1/4}\right\}
\end{align*} 
with $\ol{A}_N:=\max_{1\leq i\leq N}(|\expe{X_i^3}|\vee\|X_i\|_4)$.
\end{theorem}

\begin{rmk}\label{rmk:main}
(a) Since $\sum_{i=1}^N\influence_i(f_k)^2\leq\|f_k\|_{\ell_2}^2\mc{M}(f_k)$, Theorem \ref{thm:main} gives the bound of the form \eqref{aim} under reasonable assumptions when $q_1=\cdots=q_d$. For example, this is the case when $\expe{Q(f_j;\bs{X})Q(f_k;\bs{X})}=\mathfrak{C}_{jk}$ for all $j,k\in[d]$, $\sup_i\|X_i\|_{\psi_\alpha}<\infty$ and $\sup_j\|f_j\|_{\ell_2}<\infty$. 
Here, we keep the quantity $\sum_{i=1}^N\influence_i(f_k)^2$ rather than $\|f_k\|_{\ell_2}^2\mc{M}(f_k)$ for the convenience of the latter application (see Section \ref{proof:comp-cck})

(b) When $q_j<q_k$ for some $j,k\in[d]$, the exponents of $|\kappa_4(Q(f_k;\bs{X}))|$ and $\mc{M}(f_k)$ appearing in the bound of \eqref{eq:main} are $1/12$, which are halves of those for the case $q_j=q_k$. This phenomenon is not specific to the high-dimensional setting but common in fourth moment type theorems. See Remark 1.9(a) in \cite{DVZ2018} for more details. 

(c) In Section \ref{sec:comparison} we compare Theorem \ref{thm:main} to some existing results in some detail. The results therein show the dependence of the bound in \eqref{eq:main} on the dimension $d$ is as sharp as (and often sharper than) the previous results. 

\end{rmk}

We can easily extend Theorem \ref{thm:main} to a high-dimensional CLT for homogeneous sums in hyperrectangles as follows. Let $\mathcal{A}^\mathrm{re}(d)$ be the set of all hyperrectangles in $\mathbb{R}^d$, i.e.~$\mathcal{A}^\mathrm{re}(d)$ consists of all sets $A$ of the form
\[
A=\{(x_1,\dots,x_d)\in\mathbb{R}^d:a_j\leq x_j\leq b_j\text{ for all }j=1,\dots,d\}
\]
for some $-\infty\leq a_j\leq b_j\leq\infty$, $j=1,\dots,d$. 
\begin{corollary}\label{coro:rect}
Under the assumptions of Theorem \ref{thm:main}, we have
\begin{align*}
&\sup_{A\in\mathcal{A}^\mathrm{re}(d)}\left|P(\bs{Q}(\bs{X})\in A)-P(Z\in A)\right|\\
&\leq C'(1+\ul{\sigma}^{-1})\left\{(\log d)^{\frac{2}{3}}\delta_0[\bs{Q}(\bs{X})]^{\frac{1}{3}}
+(\log d)^{\mu+\frac{1}{2}}\delta_1[\bs{Q}(\bs{X})]^{\frac{1}{3}}
+(\log d)^{\frac{2\ol{q}_d-1}{\alpha}+\frac{3}{2}}\max_{1\leq k\leq d}\ol{B}_N^{q_k}\sqrt{\mathcal{M}(f_k)}\right\},
\end{align*}
where $C'>0$ depends only on $\alpha,\ol{q}_d$. 
\end{corollary}

For application, it is often useful to restate Theorem \ref{thm:main} in an asymptotic form as follows. 
\begin{corollary}\label{coro:main}
Let $\bs{X}=(X_i)_{i=1}^\infty$ be a sequence of independent centered random variables with unit variance. 
Set $w=\frac{1}{2}$ if $\expe{X_i^3}=0$ for every $i\in\mathbb{N}$ and $w=1$ otherwise. 
For every $n\in\mathbb{N}$, let $N_n,d_n\in\mathbb{N}\setminus\{1\}$ and $f_{n,k}:[N_n]^{q_k}\to\mathbb{R}$ ($k=1,\dots,d_n$) be symmetric functions vanishing on diagonals, and set $\bs{Q}^{(n)}(\bs{X}):=(Q(f_{n,1};\bs{X}),\dots,Q(f_{n,d_n};\bs{X}))$. 
Moreover, for every $n\in\mathbb{N}$, let $Z^{(n)}=(Z_{n,1},\dots,Z_{n,d_n})$ be a $d_n$-dimensional centered Gaussian vector with covariance matrix $\mf{C}_n=(\mf{C}_{n,kl})_{1\leq k,l\leq d_n}$. 
Suppose that $\ol{q}_\infty:=\sup_{j\in\mathbb{N}}q_j<\infty$, $\inf_{n\in\mathbb{N}}\min_{1\leq k\leq d_n}\|Z_{n,k}\|_2>0$, $\sup_{i\in\mathbb{N}}\|X_i\|_{\psi_\alpha}<\infty$ for some $\alpha\in(0,w^{-1}]$ and
\begin{equation}\label{asymp:covariance}
(\log d_n)^2\max_{1\leq k,l\leq d_n}|\expe{Q(f_{n,k};\bs{X})Q(f_{n,l};\bs{X})}-\mf{C}_{n,kl}|\to0 
\end{equation}
as $n\to\infty$. 
Moreover, setting $a_1:=(4w\ol{q}_\infty-2)\vee(4\alpha^{-1}(\ol{q}_\infty-1)+5)$ and $a_2:=2\alpha^{-1}(2\ol{q}_\infty-1)+3$, we suppose that either one of the following conditions is satisfied:
\begin{enumerate}[label=(\roman*)]

\item $(\log d_n)^{2a_1}\max_{1\leq j\leq d_n}|\kappa_4(Q(f_{n,j};\bs{X}))|\to0$ and $(\log d_n)^{2a_1\vee a_2}\max_{1\leq j\leq d_n}\mathcal{M}(f_{n,j})\to0$ as $n\to\infty$.

\item $(\log d_n)^{a_1}\max_{1\leq j\leq d_n}|\kappa_4(Q(f_{n,j};\bs{X}))|\to0$ and $(\log d_n)^{a_1\vee a_2}\max_{1\leq j\leq d_n}\mathcal{M}(f_{n,j})\to0$ as $n\to\infty$ and $q_1=q_2=\cdots$. 

\end{enumerate}
Then we have $\sup_{A\in\mathcal{A}^\mathrm{re}(d_n)}|P(\bs{Q}^{(n)}(\bs{X})\in A)-P(Z^{(n)}\in A)|\to0$ as $n\to\infty$. 
\end{corollary}

Our second main result gives high-dimensional versions of fourth moment theorems, universality results and Peccati-Tudor type theorems for homogeneous sums:
\begin{theorem}\label{thm:universality}
Let us keep the same notation as in Corollary \ref{coro:main}. 
Suppose that one of the following conditions is satisfied:
\begin{enumerate}[label={(\Alph*)}]

\item\label{univ:kurtosis} $\bs{X}$ is a sequence of independent copies of a random variable $X$ such that $\|X\|_{\psi_\alpha}<\infty$ for some $\alpha>0$ and $\expe{X^3}=0$ and $\expe{X^4}\geq 3$. 

\item\label{univ:poisson} For every $i$, $X_i$ is a standardized Poisson random variable with intensity $\lambda_i>0$, i.e.~$\lambda_i+\sqrt{\lambda_i}X_i$ is a Poisson random variable with intensity $\lambda_i$. Moreover, $\inf_{i\in\mathbb{N}}\lambda_i>0$.  

\item\label{univ:gamma} For every $i$, $X_i$ is a standardized gamma random variable with shape $\nu_i>0$ and unit rate, i.e.~$\nu_i+\sqrt{\nu_i}X_i\sim\gamma(\nu_i)$. Moreover, $\inf_{i\in\mathbb{N}}\nu_i>0$.  

\end{enumerate}
Suppose also that $2\leq\inf_{j\in\mathbb{N}}q_j\leq\sup_{j\in\mathbb{N}}q_j<\infty$, $0<\inf_{n\in\mathbb{N}}\min_{1\leq j\leq d_n}\mf{C}^{(n)}_{jj}\leq\sup_{n\in\mathbb{N}}\max_{1\leq j\leq d_n}\mf{C}^{(n)}_{jj}<\infty$ and 
\[
(\log d_n)^a\max_{1\leq j,k\leq d_n}|\expe{Q(f_{n,j};\bs{X})Q(f_{n,k};\bs{X})}-\mf{C}_{n,jk}|\to0
\] 
as $n\to\infty$ for every $a>0$. 
Then we have $\kappa_4(Q(f;\bs{X}))\geq0$ for any symmetric function $f:[N]^q\to\mathbb{R}$ vanishing on diagonals. Moreover, the following conditions are equivalent:
\begin{enumerate}[label={(\roman*)}]

\item\label{univ:fourth} $(\log d_n)^a\max_{1\leq j\leq d_n}\kappa_4(Q(f_{n,j};\bs{X}))\to0$ as $n\to\infty$ for every $a>0$.

\item\label{univ:marginal} $(\log d_n)^a\max_{1\leq j\leq d_n}\sup_{x\in\mathbb{R}}|P(Q(f_{n,j};\bs{X})\leq x)-P(Z_{n,j}\leq x)|\to0$ as $n\to\infty$ for every $a>0$.

\item\label{univ:joint} $(\log d_n)^a\sup_{x\in\mathbb{R}^{d_n}}|P(\bs{Q}^{(n)}(\bs{X})\leq x)-P(Z^{(n)}\leq x)|\to0$ as $n\to\infty$ for every $a>0$.

\item\label{univ:general} $(\log d_n)^a\sup_{x\in\mathbb{R}^{d_n}}|P(\bs{Q}^{(n)}(\bs{Y})\leq x)-P(Z^{(n)}\leq x)|\to0$ as $n\to\infty$ for any $a>0$ and sequence $\bs{Y}=(Y_i)_{i\in\mathbb{N}}$ of centered independent random variables with unit variance such that $\sup_{i\in\mathbb{N}}\|Y_i\|_{\psi_\alpha}<\infty$ for some $\alpha>0$.

\end{enumerate}
\end{theorem}

\begin{rmk}
(a) The implications \ref{univ:fourth} $\Rightarrow$ \ref{univ:joint}, \ref{univ:joint} $\Rightarrow$ \ref{univ:general} and \ref{univ:marginal} $\Rightarrow$ \ref{univ:joint} can be viewed as high-dimensional versions of fourth moment theorems, universality results and Peccati-Tudor type theorems for homogeneous sums, respectively. 
Here, \textit{Peccati-Tudor type theorems} refer to statements such that a joint CLT is implied by component-wise CLTs (\citet{PT2005} have established such a result for multiple Wiener-It\^o integrals with respect to an isonormal Gaussian process). 

(b) The proof of Theorem \ref{thm:universality} relies on the fact that condition \ref{univ:fourth} \textit{automatically} yields $(\log d_n)^a\max_{j}\mc{M}(f_{n,j})\to0$ as $n\to\infty$ for every $a>0$. 
On the one hand, this fact has already been established in the previous work for cases \ref{univ:kurtosis} and \ref{univ:poisson} (see the proof of Lemma \ref{lemma:transfer}). 
On the other hand, for case \ref{univ:gamma}, this fact seems not to have appeared in the literature so far. Indeed, for case \ref{univ:gamma} we obtain it as a byproduct of the proof of Proposition \ref{gamma-bound} (see Lemma \ref{lemma:zheng}). As a consequence, Theorem \ref{thm:universality} seems new for case \ref{univ:gamma} even in the fixed dimensional case. We remark that the fourth moment theorem for case \ref{univ:gamma} has been established by \cite{ACP2014} in the univariate case, which inspired our discussions in Section \ref{sec:normal-gamma} (see also \cite{CNPP2016}). 
\end{rmk}

\subsection{Comparison of Theorem \ref{thm:main} to some existing results}\label{sec:comparison}

\subsubsection{Comparison to Corollary 7.3 in \citet*{NPR2010aop}}\label{sec:comp-npr}

First we compare our result to the quantitative multi-dimensional CLT for homogeneous sums obtained in \citet{NPR2010aop}. To state their result, we need to introduce the notion of \textit{contraction}, which will also play an important role in Section \ref{sec:gamma-bound}. 
For two symmetric functions $f:[N]^p\to\mathbb{R},g:[N]^q\to\mathbb{R}$ and $r\in\{0,1\dots,p\wedge q\}$, we define the contraction $f\star_rg:[N]^{p+q-2r}\to\mathbb{R}$ by 
\begin{equation}\label{def:contraction}
f\star_rg(i_1,\dots,i_{p+q-2r})
=\sum_{k_1,\dots,k_r=1}^Nf(i_1,\dots,i_{p-r},k_1,\dots,k_{r})g(i_{p-r+1},\dots,i_{p+q-2r},k_1,\dots,k_r).
\end{equation}
In particular, we have
\[
f\star_0g(i_1,\dots,i_{p+q})=f\otimes g(i_1,\dots,i_{p+q})=f(i_1,\dots,i_p)g(i_{p+1},\dots,i_{p+q}).
\]

Now we are ready to state the result of \cite{NPR2010aop}. To simplify the notation, we focus only on the identity covariance matrix case and do not keep the explicit dependence of constants on $q_j$'s. 
\begin{proposition}[\citet{NPR2010aop}, Corollary 7.3]\label{prop:npr}
Let us keep the same notation as in Theorem \ref{thm:main}. 
Suppose that $\mf{C}_{jk}=\expe{Q(f_j;\bs{X})Q(f_k;\bs{X})}$ for all $j,k\in[d]$ and $\mf{C}$ is the identity matrix of size $d$. 
Suppose also that $\beta:=\max_{1\leq i\leq N}\expe{|X_i|^3}<\infty$ and $q_d\geq\cdots\geq q_1\geq2$. 
Then we have
\begin{equation}\label{eq:npr}
\sup_{A\in\mathcal{C}(\mathbb{R}^{d})}|P(\bs{Q}(\bs{X})\in A)-P(Z\in A)|
\leq Kd^{3/8}\left\{\ol{\Delta}+\mathsf{C}(\beta+1)\left(\sum_{j=1}^d\beta^{(q_j-1)/3}\right)^3\sqrt{\max_{1\leq j\leq d}\mc{M}(f_j)}\right\}^{1/4},
\end{equation}
where $\mathcal{C}(\mathbb{R}^{d})$ is the set of all convex Borel subsets of $\mathbb{R}^d$, $K>0$ is a constant depending only on $\ol{q}_d$, $\mathsf{C}:=\sum_{i=1}^N\max_{1\leq j\leq d}\influence_i(f_j)$ and
\[
\ol{\Delta}:=\sum_{1\leq j\leq k\leq d}\left(\sum_{r=1}^{q_j-1}(\|f_j\star_{q_j-r}f_j\|_{\ell_2}+\|f_k\star_{q_k-r}f_k\|_{\ell_2})+1_{\{q_j<q_k\}}\sqrt{\|f_k\star_{q_k-q_j}f_k\|_{\ell_2}}\right).
\]
\end{proposition}

To compare Proposition \ref{prop:npr} to our result, we need to bound the quantity $\ol{\Delta}$ by $|\kappa_4(Q(f_j;\bs{X}))|$ and $\mc{M}(f_j)$, $j\in[d]$. This can be carried out by the following lemma (proved in Section \ref{sec:proof-contraction}):
\begin{lemma}\label{lemma:contraction}
Let $\bs{X}=(X_i)_{i=1}^N$ be a sequence of independent centered random variables with unit variance and such that $M:=1+\max_{1\leq i\leq N}\expe{X_i^4}<\infty$. 
Also, let $q\geq2$ be an integer and $f:[N]^q\to\mathbb{R}$ be a symmetric function vanishing on diagonals. 
Then we have
\[
\max_{1\leq r\leq q-1}\|f\star_r f\|_{\ell_2}\leq \sqrt{|\kappa_4(Q(f;\bs{X}))|+CM\|f\|_{\ell_2}^2\mc{M}(f)},
\]
where $C>0$ depends only on $q$. 
\end{lemma}

\begin{rmk}
The bound in Lemma \ref{lemma:contraction} is generally sharp. 
In fact, it is well-known that $\sqrt{|\kappa_4(Q(f;\bs{X}))|}$ has the same order as $\max_{1\leq r\leq q-1}\|f\star_r f\|_{\ell_2}$ if $\bs{X}$ is Gaussian (see e.g.~Eq.(5.2.6) in \cite{NP2012}). Moreover, if $q=2$ and $f(i,j)=N^{-1/2}1_{\{|i-j|=1\}}$, then both $\|f\star_1f\|_{\ell_2}$ and $\|f\|_{\ell_2}\sqrt{\mc{M}(f)}$ are of order $N^{-1/2}$. 
\end{rmk}
With the help of Lemma \ref{lemma:contraction}, we observe that the bound in \eqref{eq:npr} typically has the same order as
\[
d^{3/8}\left\{d^2\max_{1\leq j,k\leq d}\wh{\Delta}_{jk}+d^3\mathsf{C}\max_{1\leq j\leq d}\sqrt{\mc{M}(f_j)}\right\}^{1/4},
\]
where
\[
\wh{\Delta}_{jk}:=1_{\{q_j=q_k\}}\sqrt{|\kappa_4(Q(f_j;\bs{X}))|+\mc{M}(f_j)}+1_{\{q_j<q_k\}}\left\{|\kappa_4(Q(f_k;\bs{X}))|+\mc{M}(f_k)\right\}^{1/4}.
\] 
Thus, in the bound of \eqref{eq:npr}, the dimension appears as a power of $d$, while the exponent of the ``standard'' convergence rate $\delta:=\max_{1\leq j\leq d}\sqrt{|\kappa_4(Q(f_j;\bs{X}))|+\mc{M}(f_j)}$ is $1/4$. 
These are much improved in our result because the former appears as a power of $\log d$ and the latter is $1/3$. 
Nevertheless, we should note that the bound in \eqref{eq:npr} is given for the much stronger metric than the Kolmogorov distance. In fact, to the best of the author's knowledge, all the known bounds for this metric depend polynomially on the dimension even for sums of independent random variables; see \cite[Section 1.1]{Zhai2018} and references therein. 

\begin{rmk}\label{rmk:npr}
(a) Roughly speaking, the exponent of $\delta$ is $1/4$ in the bound of \eqref{eq:npr} is because this bound is transferred from an analogous quantitative CLT for the Gaussian counterpart by the Lindeberg method with matching moments \textit{up to the second order}. To overcome this issue, we need to match moments \textit{up to the third order} and thus we can no longer rely on the result analogous to Theorem \ref{thm:main} for the Gaussian counterpart, which is obtained in \cite{Koike2017stein}. For this reason we will develop a high-dimensional CLT for homogeneous sums based on normal and gamma variables in Section \ref{sec:normal-gamma}. 

(b) It is worth noting that the quantity $\mathsf{C}=\sum_{i=1}^N\max_{1\leq j\leq d}\influence_i(f_j)$ in the bound of \eqref{eq:npr} can be much larger than $\max_{1\leq j\leq d}\sum_{i=1}^N\influence_i(f_j)=\max_{1\leq j\leq d}\|f_j\|_{\ell_2}^2$ in high-dimensional situations (see Remark \ref{rmk:qf} for a concrete example). 
Indeed, na\"ive application of the Lindeberg method produces a quantity like $\mathsf{C}$, which prevents us from using the Lindeberg method in its pure form (this is why \citet{CCK2013,CCK2017} rely on Stein's method to prove their high-dimensional CLTs; see \cite[Appendix L]{CCK2013supp} for a detailed discussion). 
In Section \ref{sec:lindeberg}, we will resolve this issue by \textit{randomizing} the Lindeberg method as \citet{DZ2017} have recently done in the context of sums of independent random variables.    
\end{rmk}

\subsubsection{Comparison to Proposition 2.1 in \citet*{CCK2017}}\label{comp-cck}

Let us recall the precise statement of \cite[Proposition 2.1]{CCK2017}:
\begin{proposition}[\citet{CCK2017}, Proposition 2.1]\label{prop:cck}
Let $\xi_i=(\xi_{i1},\dots,\xi_{id})$ $(i=1,\dots,n)$ be independent centered $d$-dimensional random vectors and set $S_n=(S_{n,1},\dots,S_{n,d}):=n^{-1/2}\sum_{i=1}^n\xi_i$. 
Assume $n\geq4$, $d\geq3$, $\mf{C}=\expe{S_nS_n^\top}$ and $\ul{\sigma}^2:=\min_{1\leq j\leq d}S_{n,j}^2>0$.  
Also, suppose that there is a constant $B_n\geq1$ such that $n^{-1}\sum_{i=1}^n\expe{|\xi_{n,ij}|^{2+k}}\leq B_n^k$~~for all $j\in[d]$ and $k=1,2$. Then the following statements hold true.
\begin{enumerate}[label=(\alph*)]

\item\label{cck:psi} Assume $\max_{1\leq i\leq n}\max_{1\leq j\leq d}\|\xi_{ij}\|_{\psi_1}\leq B_n$. Then we have
\begin{equation*}
\sup_{A\in\mathcal{A}^{\mathrm{re}}(d)}|P(S_n\in A)-P(Z\in A)|\leq C\left(\frac{B_n^2\log^7(dn)}{n}\right)^{1/6},
\end{equation*}
where $C>0$ depends only on $\ul{\sigma}$. 

\item\label{cck:mom} Assume $\max_{1\leq i\leq n}\ex{(\max_{1\leq j\leq d}|\xi_{ij}|/B_n)^q}\leq 2$ for some $q>0$. Then we have
\begin{equation*}
\sup_{A\in\mathcal{A}^{\mathrm{re}}(d)}|P(S_n\in A)-P(Z\in A)|\leq C'\left\{\left(\frac{B_n^2\log^7(dn)}{n}\right)^{1/6}+\left(\frac{B_n^2\log^3(dn)}{n^{1-2/q}}\right)^{1/3}\right\},
\end{equation*}
where $C'>0$ depends only on $\ul{\sigma}$. 

\end{enumerate}
\end{proposition}

If we restrict our attention to Gaussian approximation for the maximum statistic $\max_{1\leq j\leq d}S_{n,j}$, we can slightly improve the bound in Proposition \ref{prop:cck}\ref{cck:psi} in terms of $d$ by combining Theorem \ref{thm:main} with \cite[Theorem 2]{DZ2017} as follows:
\begin{proposition}\label{prop:cck-main}
Under the assumptions of Proposition \ref{prop:cck}\ref{cck:psi}, we have
\begin{equation}\label{eq:cck-main-psi}
\sup_{t\in\mathbb{R}}\labs P\lpa\max_{1\leq j\leq d}S_{n,j}\leq t\rpa-P\lpa\max_{1\leq j\leq d}Z_j\leq t\rpa\rabs\leq C\left\{\left(\frac{B_n^2\log^6(dn)}{n}\right)^{1/6}
+\sqrt{\frac{B_n^2(\log d)^5(\log n)^2}{n}}\right\},
\end{equation}
where $C>0$ depends only on $\ul{\sigma}$. 
\end{proposition}

\begin{rmk}
(a) Proposition \ref{prop:cck-main} is concerned only with approximation of $\max_{1\leq j\leq d}S_{n,j}$ because \cite[Theorem 2]{DZ2017} is. Since it is presumably possible to extend \cite[Theorem 2]{DZ2017} to (bootstrap) approximation in hyperrectangles, Proposition \ref{prop:cck-main} could hopefully be extended to a high-dimensional CLT in hyperrectangles. 

(b) The proof of Proposition \ref{prop:cck-main} suggests that we could drop the second term in the right side of \eqref{eq:cck-main-psi} when $\max_{i,j}\|\xi_{ij}\|_{\psi_2}\leq B_n$.
\end{rmk}

\subsubsection{Comparison to Theorem 3.2 in \citet{Koike2017stein}}\label{sec:comp-stein}

As the last illustration, we compare our result to the Gaussian approximation result for maxima of quadratic forms obtained in \cite[Theorem 3.2]{Koike2017stein}. 
Here, for an explicit comparison, we state this result with applying \cite[Corollary 3.1]{Koike2017stein}. 
For a function $f:[N]^2\to\mathbb{R}$, we denote the $N\times N$ matrix $(f(i,j))_{1\leq i,j\leq N}$ by $[f]$.
\begin{proposition}[\citet{Koike2017stein}, Theorem 3.2 and Corollary 3.1]\label{prop:koike}
Let us keep the same notation as in Corollary \ref{coro:main}. 
Suppose that $q_1=q_2=\cdots=2$, $\inf_{n\in\mathbb{N}}\min_{1\leq k\leq d_n}\|Z_{n,k}\|_2>0$, $\sup_{i\in\mathbb{N}}\|X_i\|_{\psi_2}<\infty$ and \eqref{asymp:covariance} holds true as $n\to\infty$. 
Suppose also that
\begin{align}
(\log d_n)^3\max_{1\leq k\leq d_n}\sqrt{\trace\left([f_{n,k}]^4\right)}+(\log d_n)^5\left(\max_{1\leq k\leq d_n}\sqrt{\mc{M}(f_{n,k})}\right)\sum_{i=1}^{N_n}\max_{1\leq k\leq d_n}\influence_i(f_{n,k})\to0\label{qf:deJong}
\end{align}
as $n\to\infty$. 
Then we have
\begin{equation}\label{qf:aim}
\sup_{t\in\mathbb{R}}\labs P\lpa\max_{1\leq k\leq d_n}|Q(f_{n,k};\bs{X})|\leq t\rpa-P\lpa\max_{1\leq k\leq d_n}|Z_{n,k}|\leq t\rpa\rabs\to0
\end{equation}
as $n\to\infty$. 
\end{proposition}

When we apply our result to quadratic forms as above, we obtain the following result. 
\begin{proposition}\label{prop:koike-main}
Let us keep the same notation as in Corollary \ref{coro:main}. 
Set
\[
\Delta_n:=\max_{1\leq k\leq d_n}\sqrt{\trace\left([f_{n,k}]^4\right)}+\max_{1\leq k\leq d_n}\sqrt{\mc{M}(f_{n,k})}\|f_{n,k}\|_{\ell_2}
\]
for every $n$. 
Assume $q_1=q_2=\cdots=2$, $\inf_{n\in\mathbb{N}}\min_{1\leq k\leq d_n}\|Z_{n,k}\|_2>0$ and \eqref{asymp:covariance}. Assume also that either one of the following conditions is satisfied:
\begin{enumerate}[label=(\roman*)]

\item\label{qf-cond1} $\sup_{i\in\mathbb{N}}\|X_i\|_{\psi_1}<\infty$ and $(\log d_n)^5\Delta_n\to0$ as $n\to\infty$.

\item\label{qf-cond2} $\sup_{i\in\mathbb{N}}\|X_i\|_{\psi_2}<\infty$, $\expe{X_i^3}=0$ for all $i$ and $(\log d_n)^3\Delta_n\to0$ as $n\to\infty$. 

\end{enumerate}
Then we have $\sup_{A\in\mathcal{A}^\mathrm{re}(d_n)}|P(\bs{Q}^{(n)}(\bs{X})\in A)-P(Z^{(n)}\in A)|\to0$ as $n\to\infty$. 
\end{proposition}

\begin{rmk}
\label{rmk:qf}
(a) Regarding the convergence rate of $\max_{1\leq k\leq d_n}\trace\left([f_{n,k}]^4\right)$, condition \ref{qf-cond1} in Proposition \ref{prop:koike-main} is stronger than the one in Proposition \ref{prop:koike}. However, the former imposes a weaker moment condition on $\bs{X}$ than the latter. More importantly, the second term of $\Delta_n$ is always smaller than or equal to the second term in \eqref{qf:deJong}, and the latter can be much larger than the former. 
For example, let us assume $N_n=d_n=n$ and consider the functions $f_{n,k}$ defined as follows:
\[
f_{n,k}(i,j)=\left\{
\begin{array}{ll}
n^{-1/2}  & \text{if }|i-j|=1,i\neq k,j\neq k,\\
n^{-1/4}  & \text{if }|i-j|=1,i=k\text{ or }j=k,\\
0 & \text{otherwise}.        
\end{array}
\right.
\]
Then, we have $\influence_i(f_{n,k})=(1+1_{\{1<i<n\}})n^{-1/2}$ if $i\in\{k,k\pm1\}$ and $\influence_i(f_{n,k})=(1+1_{\{1<i<n\}})n^{-1}$ otherwise. Therefore, on the one hand $\left(\max_{1\leq k\leq d_n}\sqrt{\mc{M}(f_{n,k})}\right)\sum_{i=1}^{N_n}\max_{1\leq k\leq d_n}\influence_i(f_{n,k})$ does not converge to 0 as $n\to\infty$, but on the other hand $\max_{1\leq k\leq d_n}\sqrt{\mc{M}(f_{n,k})}\|f_{n,k}\|_{\ell_2}=O(n^{-1/4})$ as $n\to\infty$. 
Note that in this case we have $\|f_{n,k}\|_{\ell_2}\to1$ and $\max_{1\leq k\leq d_n}\sqrt{\trace\left([f_{n,k}]^4\right)}=O(n^{-1/4})$ as $n\to\infty$, so \eqref{qf:aim} holds true due to Proposition \ref{prop:koike-main}. 

(b) Condition \ref{qf-cond2} in Proposition \ref{prop:koike-main} requires the additional zero skewness assumption, but it always improves the assumption on the functions $f_{n,k}$ than the one in Proposition \ref{prop:koike}. 

(c) Let $\|\cdot\|_{\mathrm{sp}}$ denote the spectral norm of matrices. Then we have $\Delta_n\leq2\max_{1\leq k\leq d_n}\|[f_{n,k}]\|_{\mathrm{sp}}\|f_{n,k}\|_{\ell_2}$. Therefore, $(\log d_n)^a\Delta_n\to 0$ for some $a>0$ is implied by $(\log d_n)^a\max_{1\leq k\leq d_n}\|[f_{n,k}]\|_{\mathrm{sp}}\|f_{n,k}\|_{\ell_2}\to0$. 
\end{rmk}

\subsection{Statistical application: Bootstrap test for the absence of lead-lag relationship}\label{sec:hry}

Let $W_t=(W_t^1,W_t^2)$ $(t\in\mathbb R)$ be a two-sided bivariate standard Wiener process. 
Also let $\rho\in(-1,1)$ and $\vartheta\in\mathbb R$ be two (unknown) parameters. 
We define the bivariate process $B_t=(B_t^1,B_t^2)$ $(t\in\mathbb R)$ as $B_t^1=W^1_t$ and $B_t^2=\rho W^1_{t-\vartheta}+\sqrt{1-\rho^2}W^2_t$. 
For each $\nu=1,2$, we consider the process $X^\nu=(X^\nu_t)_{t\geq0}$ given by
\begin{equation}\label{model}
X^\nu_t=X^\nu_0+\int_0^t\sigma_\nu(s)dB^\nu_s,\qquad t\geq0,
\end{equation}
where $\sigma_\nu\in L^2(0,\infty)$ is nonnegative-valued and deterministic. 
If $\rho\neq0$, there is a correlation between $X^1$ and $X^2$ with a time lag of $\vartheta$. 
We aim to test for whether such a correlation really exists or not, given (possibly asynchronous) high-frequency observations of $X^1$ and $X^2$. 
Specifically, for each $\nu=1,2$, we observe the process $X^\nu$ on the interval $[0,T]$ at the deterministic sampling times $0\leq t^\nu_0<t^\nu_1<\cdots<t^\nu_{n_\nu}\leq T$, which implicitly depend on the parameter $n\in\mathbb{N}$ such that
\[
r_n:=\max_{\nu=1,2}\max_{i=0,1,\dots,n_\nu+1}(t^\nu_i-t^\nu_{i-1})\to0
\]
as $n\to\infty$, where we set $t^\nu_{-1}:=0$ and $t^\nu_{n_\nu+1}:=T$ for each $\nu=1,2$. To test for the null hypothesis $H_0:\rho=0$ against the alternative $H_1:\rho\neq0$, \citet{Koike2017stein} proposed the test statistic given by $T_n=\sqrt{n}\max_{\theta\in\mathcal{G}_n}|U_n(\theta)|$, where $\mathcal{G}_n$ is a finite subset of $\mathbb{R}$ and
\[
U_n(\theta)=\sum_{i=1}^{n_1}\sum_{j=1}^{n_2}\Delta^n_iX^1\Delta^n_jX^2K^{ij}_\theta\qquad
\text{with }\Delta^n_iX^\nu=X^\nu_{t^\nu_i}-X^\nu_{t^\nu_{i-1}}\text{ and }K^{ij}_\theta=1_{\{(t^1_{i-1},t^1_i]\cap(t^2_{j-1}-\theta,t^2_j-\theta]\neq\emptyset\}}.
\] 
The null distribution of $T_n$ can be approximated by its Gaussian analog as follows:
\begin{proposition}[\cite{Koike2017stein}, Proposition 4.1]\label{hry}
For each $n\in\mathbb{N}$, let $(Z_n(\theta))_{\theta\in\mathcal{G}_n}$ be a family of centered Gaussian variables such that $\expe{Z_n(\theta)Z_n(\theta')}=n\covariance[U_n(\theta),U_n(\theta')]$ for all $\theta,\theta'\in\mathcal{G}_n$. 
Suppose that $\sup_{t\in[0,T]}(\sigma_1(t)+\sigma_2(t))<\infty$ and there are positive constants $\underline{v},\overline{v}$ such that 
\[
\underline{v}\leq n\sum_{i=1}^{n_1}\sum_{j=1}^{n_2}\left(\int_{t_{i-1}^1}^{t_i^1}\sigma_1(t)^2dt\right)\left(\int_{t_{j-1}^2}^{t_j^2}\sigma_2(t)^2dt\right)K^{ij}_\theta\leq\overline{v}
\]
for all $n\in\mathbb{N}$ and $\theta\in\mathcal{G}_n$. Then, under the null hypothesis $\rho=0$, we have
\[
\sup_{x\in\mathbb{R}}\left|P\left(T_n\leq x\right)-P\left(\max_{\theta\in\mathcal{G}_n}|Z_n(\theta)|\leq x\right)\right|\to0
\]
as $n\to\infty$, provided that $nr_n^2\log^6(\#\mathcal{G}_n)\to0$.
\end{proposition}
Since the distribution of $\max_{\theta\in\mathcal{G}_n}|Z_n(\theta)|$ is analytically intractable, \citet{Koike2017stein} proposed a wild bootstrap procedure to approximate it. 
Formally, let $(w^1_i)_{i=1}^\infty$ and $(w^2_j)_{j=1}^\infty$ be mutually independent sequences of i.i.d.~random variables independent of $X^1$ and $X^2$. 
Assume that $\expe{w^1_1}=\expe{w^2_1}=0$, $\variance[w^1_1]=\variance[w^2_1]=1$ and $\|w^1_1\|_{\psi_2}\vee\|w^2_1\|_{\psi_2}<\infty$. 
Define the bootstrapped test statistic as $T_n^*=\sqrt{n}\max_{\theta\in\mathcal{G}_n}|U_n^*(\theta)|$ where
\[
U_n^*(\theta)=
\sum_{i=1}^{n_1}\sum_{j=1}^{n_2}\left(w^1_i\Delta^n_iX^1\right)\left(w^2_j\Delta^n_jX^2\right)K^{ij}_\theta.
\]
In \cite[Proposition B.8]{Koike2017stein}, it is shown that
\begin{equation}\label{eq:boot}
\sup_{x\in\mathbb{R}}\left|P\left(T_n^*\leq x\mid X\right)-P\left(\max_{\theta\in\mathcal{G}_n}|Z_n(\theta)|\leq x\right)\right|\to^p0
\end{equation}
as $n\to\infty$, provided that $r_n=O(n^{-3/4-\eta})$ and $\#\mathcal G_n=O(n^\gamma)$ for some $\eta,\gamma>0$ in addition to the assumptions of Proposition \ref{hry}. 
Our result allows us to relax the condition on $r_n$ as follows:
\begin{proposition}\label{boot-dist}
Under the assumptions of Proposition \ref{hry}, we have \eqref{eq:boot} as $n\to\infty$, provided that $r_n=O(n^{-1/2-\eta})$ and $\#\mathcal G_n=O(n^\gamma)$ for some $\eta,\gamma>0$.
\end{proposition}

\section{Chernozhukov-Chetverikov-Kato theory}\label{sec:cck}

In this section we demonstrate a basic scheme of the CCK theory to establish high-dimensional CLTs. 
One main ingredient of the CCK theory is the following smooth approximation of the maximum function: For each $\beta>0$, we define the function $\Phi_\beta:\mathbb{R}^d\to\mathbb{R}$ by 
\[
\Phi_\beta(x)=\beta^{-1}\log\left(\sum_{j=1}^de^{\beta x_j}\right),\qquad x=(x_1,\dots,x_d)\in\mathbb{R}^d.
\]
Eq.(1) in \cite{CCK2015} states that
\begin{equation}\label{max-smooth}
0\leq \Phi_\beta(x)-\max_{1\leq j\leq d}x_j\leq \beta^{-1}\log d
\end{equation}
for any $x\in\mathbb{R}^d$. Therefore, the larger $\beta$ is, the better $\Phi_\beta$ approximates the maximum function. 
The next lemma, which is a summary of \cite[Lemmas 5--6]{DZ2017}, highlights the key properties of this smooth max function: 
\begin{lemma}\label{cck-derivative}
For any $\beta>0$, $m\in\mathbb{N}$ and $C^m$ function $h:\mathbb{R}\to\mathbb{R}$, there is an $\mathbb{R}^{\otimes m}$-valued function $\Upsilon_{\beta}(x)=(\Upsilon^{j_1,\dots,j_m}_\beta(x))_{1\leq j_1,\dots,j_m\leq d}$ on $\mathbb{R}^d$ satisfying the following conditions:
\begin{enumerate}[label=(\roman*)]

\item For any $x\in\mathbb{R}^d$ and $j_1,\dots,j_m\in[d]$, we have 
$
|\partial_{j_1\dots j_m}(h\circ\Phi_\beta)(x)|\leq \Upsilon_\beta^{j_1,\dots,j_m}(x).
$

\item For every $x\in\mathbb{R}^d$, we have
\begin{align*}
\sum_{j_1,\dots,j_m=1}^d\Upsilon_\beta^{j_1,\dots, j_m}(x)
\leq c_{m}\max_{1\leq k\leq m}\beta^{m-k}\|h^{(k)}\|_\infty,
\end{align*}
where $c_{m}>0$ depends only on $m$. 

\item For any $x,t\in\mathbb{R}^d$ and $j_1,\dots,j_m\in[d]$, we have
\[
e^{-8\|t\|_{\ell_\infty}\beta}\Upsilon_\beta^{j_1,\dots,j_m}(x+t)\leq \Upsilon_\beta^{j_1,\dots,j_m}(x)\leq e^{8\|t\|_{\ell_\infty}\beta}\Upsilon_\beta^{j_1,\dots,j_m}(x+t).
\]

\end{enumerate}
\end{lemma}

\begin{rmk}
An explicit expression of the constant $c_m$ in Lemma \ref{cck-derivative} can be derived from \cite[Lemma 5]{DZ2017}. In particular, we have $c_1=1,c_2=3,c_3=13$ and $c_4=75$.
\end{rmk}

Another important ingredient of the CCK theory is the so-called \textit{anti-concentration inequality}. 
For our purpose, the following one is particularly useful (see \cite{CCK2017nazarov} for the proof): 
\begin{lemma}[Nazarov's inequality]\label{lemma:nazarov}
If $\ul{\sigma}:=\min_{1\leq j\leq d}\|Z_j\|_2>0$, for any $x\in\mathbb{R}^d$ and $\varepsilon>0$ we have
\[
P(Z\leq x+\varepsilon)-P(Z\leq x)\leq\frac{\varepsilon}{\ul{\sigma}}\left(\sqrt{2\log d}+2\right).
\]
\end{lemma}
%

With the help of these tools, we establish the following form of smoothing inequality:
\begin{proposition}\label{smoothing}
Let $g_0:\mathbb{R}\to[0,1]$ be a measurable function such that $g_0(t)=1$ for $t\leq0$ and $g_0(t)=0$ for $t\geq1$. 
Also, let $\varepsilon>0$ and set $\beta:=\varepsilon^{-1}\log d$. Suppose that $\ul{\sigma}:=\min_{1\leq j\leq d}\|Z_j\|_2>0$. Then, for any $d$-dimensional random vector $F$, we have
\begin{equation}\label{eq:smoothing}
\sup_{x\in\mathbb{R}^d}\left|P(F\leq x)-P(Z\leq x)\right|\leq \Delta_\varepsilon(F,Z)+\frac{2\varepsilon}{\ul{\sigma}}\left(\sqrt{2\log d}+2\right),
\end{equation}
where
\begin{equation}\label{eq:delta}
\Delta_\varepsilon(F,Z):=\sup_{y\in\mathbb{R}^d}\left|\expe{g_0(\varepsilon^{-1}\Phi_\beta(F-y))}-\expe{g_0(\varepsilon^{-1}\Phi_\beta(Z-y))}\right|.
\end{equation}
\end{proposition}

\begin{proof}
This result has been essentially shown in Step 2 in the proof of \cite[Lemma 5.1]{CCK2017}, so our argument is almost the same as theirs. 
Take $x\in\mathbb{R}^d$ arbitrarily. Using \eqref{max-smooth} and the assumptions on $g_0$, we obtain
\begin{align*}
P(F\leq x)&\leq P(\Phi_\beta(F-x-\varepsilon)\leq0)\\
&\leq \expe{g_0(\varepsilon^{-1}\Phi_\beta(F-x-\varepsilon))}
\leq \expe{g_0(\varepsilon^{-1}\Phi_\beta(Z-x-\varepsilon))}+\Delta_\varepsilon(F,Z)\\
&\leq P(\Phi_\beta(Z-x-\varepsilon)<\varepsilon)+\Delta_\varepsilon(F,Z)
\leq P(Z\leq x+2\varepsilon)+\Delta_\varepsilon(F,Z).
\end{align*}
Therefore, Lemma \ref{lemma:nazarov} yields
\[
P(F\leq x)\leq P(Z\leq x)+\Delta_\varepsilon(F,Z)+\frac{2\varepsilon}{\ul{\sigma}}\left(\sqrt{2\log d}+2\right).
\]
Similarly, we can prove
\[
P(F\leq x)\geq P(Z\leq x)-\Delta_\varepsilon(F,Z)-\frac{2\varepsilon}{\ul{\sigma}}\left(\sqrt{2\log d}+2\right).
\]
Combining these estimates, we obtain the desired result. 
\end{proof}

\begin{rmk}
Proposition \ref{smoothing} can be seen as a special version of more general smoothing inequalities such as \cite[Lemma 2.1]{Bentkus2003}. An important feature of bound \eqref{eq:smoothing} is that the quantity $\Delta_\varepsilon(F,Z)$ contains only test functions of the form $x\mapsto g_0(\Phi_\beta(x-y))$ for some $y\in\mathbb{R}^d$. If $g_0$ is sufficiently smooth, derivatives of such a test function admit good estimates with respect to the dimension $d$, as seen from Lemma \ref{cck-derivative}. 
\end{rmk}

It might be worth mentioning that we can use Proposition \ref{smoothing} to derive a bound for the Kolmogorov distance by the Wasserstein distance. 
Let us recall the definition of the Wasserstein distance. 
\begin{definition}[Wasserstein distance]
For $d$-dimensional random vectors $F,G$ with integrable components, the \textit{Wasserstein distance} between the laws of $F$ and $G$ is defined by
\[
\mathcal{W}_1(F,G):=\sup_{h\in\mc{H}}|\expe{h(F)}-\expe{h(G)}|,
\] 
where $\mc{H}$ denotes the set of all functions $h:\mathbb{R}^d\to\mathbb{R}$ such that
\[
\|h\|_{\mathrm{Lip}}:=\sup_{x,y\in\mathbb{R}^d:x\neq y}\frac{|h(x)-h(y)|}{\|x-y\|}\leq1.
\]
Here, $\|\cdot\|$ is the usual Euclidian norm on $\mathbb{R}^d$. 
\end{definition}

\begin{corollary}\label{coro:wass}
Under the assumptions of Proposition \ref{smoothing}, we have
\[
\sup_{x\in\mathbb{R}^d}\left|P(F\leq x)-P(Z\leq x)\right|\leq\sqrt{\frac{2\left(\sqrt{2\log d}+2\right)}{\ul{\sigma}}\mathcal{W}_1(F,Z)}.
\]
\end{corollary}

\begin{proof}
It suffices to consider the case $\mathcal{W}_1(F,Z)>0$. 
Let us define the function $g_0:\mathbb{R}\to[0,1]$ by $g_0(x)=\min\{1,\max\{1-x,0\}\}$, $x\in\mathbb{R}$. Then, for any $x,x',y\in\mathbb{R}^d$ and $\varepsilon>0$, we have $|g_0(\varepsilon^{-1}\Phi_\beta(x-y)-g_0(\varepsilon^{-1}\Phi_\beta(x'-y)|\leq\varepsilon^{-1}\|x-x'\|_{\ell_\infty}$ by \cite[Lemma A.3]{CCK2013}, so we obtain $\Delta_\varepsilon(F,Z)\leq\varepsilon^{-1}\mathcal{W}_1(F,Z)$. Now, setting $\varepsilon=\sqrt{\ul{\sigma}\mathcal{W}_1(F,Z)/(2\sqrt{2\log d}+4)}$, we infer the desired result from Proposition \ref{smoothing}. 
\end{proof}
When $d=1$, Corollary \ref{coro:wass} recovers the standard estimate (cf.~Eq.(C.2.6) in \cite{NP2012}). 
We should remark that a bound similar to the above (with a slightly different constant) has already appeared in \cite[Theorem 3.1]{APP2016}. 
\begin{rmk}
It is generally impossible to derive \eqref{aim} type bounds from the corresponding ones for the Wasserstein distance. To see this, let $F=(F_1,\dots,F_d)$ be a $d$-dimensional random vector such that the laws of $F_1,\dots,F_d$ are identical (and integrable). Also, let $G=(G_1,\dots,G_d)$ be another $d$-dimensional random vector satisfying the same condition. Then we can easily prove $\mathcal{W}_1(F,G)\geq\sqrt{d}\mathcal{W}_1(F_1,G_1)$ by definition. 
\end{rmk}

\section{Stein kernels and high-dimensional CLTs}\label{sec:stein}

In the rest of the paper, we fix a $C^\infty$ function $g_0:\mathbb{R}\to[0,1]$ such that $g_0(t)=1$ for $t\leq0$ and $g_0(t)=0$ for $t\geq1$: For example, we can take it as $g_0(t)=f_0(1-t)/\{f_0(t)+f_0(1-t)\}$, where the function $f_0:\mathbb{R}\to\mathbb{R}$ is defined by
\[
f_0(t)=
\left\{
\begin{array}{cl}
e^{-1/t}  & \text{if }t>0,\\
0  & \text{if }t\leq0.
\end{array}
\right.
\]
%

To make Proposition \ref{smoothing} useful, we need to obtain a ``good'' upper bound for the quantity $\Delta_\varepsilon(F,Z)$. 
As briefly mentioned in Remark \ref{rmk:npr}, \citet{CCK2013} have pointed out that Stein's method effectively solves this task. 
Moreover, discussions in \cite{CCK2015,Koike2017stein} implicitly suggest that the CCK theory would have a nice connection to Stein kernels. In this section we illustrate this idea. 
\begin{definition}[Stein kernel]
Let $F=(F_1,\dots,F_d)$ be a centered $d$-dimensional random variable. 
A $d\times d$ matrix-valued measurable function $\tau_F=(\tau_F^{ij})_{1\leq i,j\leq d}$ on $\mathbb{R}^d$ is called a \textit{Stein kernel} for (the law of) $F$ if $\max_{1\leq i,j\leq d}\expe{|\tau_F^{ij}(F)|}<\infty$ and 
\begin{equation}\label{eq:stein}
\sum_{j=1}^d\expe{\partial_j\varphi(F)F_j}=\sum_{i,j=1}^d\expe{\partial_{ij}\varphi(F)\tau_F^{ij}(F)}
\end{equation}
for any $\varphi\in C^\infty_b(\mathbb{R}^d)$.
\end{definition}

\begin{rmk}
In this paper we adopt $C^\infty_b(\mathbb{R}^d)$ as the class of test functions for which identity \eqref{eq:stein} holds true because of convenience, but other classes of test functions are also used in the literature; see \cite{CFP2017} for instance.  
\end{rmk}

\begin{lemma}\label{stein}
Let $F=(F_1,\dots,F_d)$ be a centered $d$-dimensional random vector. Also, let $\tau_F=(\tau_F^{ij})_{1\leq i,j\leq d}$ be a Stein kernel for $F$. Then we have
\[
\sup_{y\in\mathbb{R}^d}\left|\ex{h\left(\Phi_\beta(F-y)\right)}-\ex{h\left(\Phi_\beta(Z-y)\right)}\right|
\leq \frac{3}{2}\max\{\|h''\|_\infty,\beta\|h'\|_\infty\}\Delta
\]
for any $\beta>0$ and $h\in C^\infty_b(\mathbb{R})$, where 
\begin{equation*}
\Delta:=\ex{\max_{1\leq i,j\leq d}|\tau_F^{ij}(F)-\mathfrak{C}_{ij}|}.
\end{equation*}
\end{lemma}

\begin{proof}
Without loss of generality, we may assume that $F$ and $Z$ are independent. Take a vector $y\in\mathbb{R}^d$ arbitrarily. Then, we define the functions $\varphi:\mathbb{R}^d\to\mathbb{R}$ and $\Psi:[0,1]\to\mathbb{R}$ by $\varphi(x)=h(\Phi_\beta(x-y))$ for $x\in\mathbb{R}$ and $\Psi(t)=\expe{\varphi(\sqrt{t}F+\sqrt{1-t}Z)}$ for $t\in[0,1]$ respectively. It is easy to check that $\Psi$ is continuous on $[0,1]$ and differentiable on $(0,1)$ and 
\[
\Psi'(t)=\frac{1}{2}\sum_{j=1}^d\ex{\partial_j\varphi(\sqrt{t}F+\sqrt{1-t}Z)\left(\frac{F_j}{\sqrt{t}}-\frac{Z_j}{\sqrt{1-t}}\right)}
\]
for every $t\in(0,1)$. Now, by Lemma \ref{cck-derivative} we have 
$
\sum_{i,j=1}^d\left|\partial_{ij}\varphi(x)\right|\leq 3\max\{\|h''\|_\infty,\beta\|h'\|_\infty\}
$ 
for any $x\in\mathbb{R}^d$. In particular, $\partial_{ij}\varphi$ is bounded for all $i,j\in[d]$. Therefore, noting that the independence between $Z$ and $F$, Stein's identity (e.g.~Lemma 2 in \cite{CCK2015}) implies that
\begin{align*}
\sum_{j=1}^d\ex{\partial_j\varphi(\sqrt{t}F+\sqrt{1-t}Z)\frac{Z_j}{\sqrt{1-t}}}
=\sum_{i,j=1}^d\ex{\partial_{ij}\varphi(\sqrt{t}F+\sqrt{1-t}Z)\mathfrak{C}_{ij}}.
\end{align*}
Moreover, since $\tau_F$ is a Stein kernel for $F$, we have
\begin{align*}
\sum_{j=1}^d\ex{\partial_j\varphi(\sqrt{t}F+\sqrt{1-t}Z)\frac{F_j}{\sqrt{t}}}
=\sum_{i,j=1}^d\ex{\partial_{ij}\varphi(\sqrt{t}F+\sqrt{1-t}Z)\tau_F^{ij}(F)}.
\end{align*}
Hence we conclude that
\begin{equation*}
\Psi'(t)=\frac{1}{2}\sum_{i,j=1}^d\ex{\partial_{ij}\varphi(\sqrt{t}F+\sqrt{1-t}Z)(\tau_F^{ij}(F)-\mathfrak{C}_{ij})}
\end{equation*}
for every $t\in(0,1)$. Consequently, we obtain
\[
\left|\ex{\varphi\left(F\right)\right]-E\left[\varphi\left(Z\right)}\right|
\leq\int_0^1|\Psi'(t)|dt
\leq\frac{3}{2}\max\{\|h''\|_\infty,\beta\|h'\|_\infty\}\Delta,
\]
which completes the proof.
\end{proof}

\begin{proposition}\label{kolmogorov}
Suppose that $d\geq2$ and $\ul{\sigma}:=\min_{1\leq j\leq d}\|Z_j\|_2>0$. 
Under the assumptions of Lemma \ref{stein}, there is a universal constant $C>0$ such that
\begin{equation}\label{eq:kolmogorov2}
\sup_{x\in\mathbb{R}^d}\left|P(F\leq x)-P(Z\leq x)\right|
\leq C(1+\ul{\sigma}^{-1})\Delta^{1/3}(\log d)^{2/3}.
\end{equation}
\end{proposition}

\begin{proof}
Thanks to \cite[Lemma 4.1.3]{NP2012}, it suffices to consider the case $\Delta>0$. 
By Lemma \ref{stein}, for any $\varepsilon>0$ we have $\Delta_\varepsilon(F,Z)\leq C'\varepsilon^{-2}(\log d)\Delta$, where $C'>0$ is a universal constant. Therefore, Proposition \ref{smoothing} yields
\[
\sup_{x\in\mathbb{R}^d}\left|P(F\leq x)-P(Z\leq x)\right|\leq C'\varepsilon^{-2}(\log d)\Delta+\frac{2\varepsilon}{\ul{\sigma}}\left(\sqrt{2\log d}+2\right).
\]
Now, setting $\varepsilon=\Delta^{1/3}(\log d)^{1/6}$, we obtain the desired result. 
\end{proof}

\begin{rmk}
Proposition \ref{kolmogorov} should be contrasted with Proposition 3.4 in \cite{LNP2015} that gives a bound for the Wasserstein distance in terms of Stein kernels. To be precise, under the assumptions of Lemma \ref{stein}, we have 
\begin{equation}\label{wass-stein}
\mc{W}_1(F,Z)\leq\sqrt{\frac{2}{\pi}}\ex{\sum_{i,j=1}^d|\tau^{ij}_F(F)-\mf{C}_{ij}|}
\end{equation}
if $\mf{C}$ is the identity matrix. 
Since the bound in \eqref{wass-stein} generally has the same order as $d\times\Delta$, it provides a better bound than Proposition \ref{kolmogorov} (see Corollary \ref{coro:wass}) when $d$ is fixed. However, this is not the case when $d$ increases (much) faster than the convergence rate of $\Delta$. 
\end{rmk}

\section{A high-dimensional CLT for normal-gamma homogeneous sums}\label{sec:normal-gamma}

In view of the results in Section \ref{sec:stein}, we naturally seek a situation where a vector of homogeneous sums has a Stein kernel. This is the case when all the components are eigenfunctions of a Markov diffusion operator (cf.~Proposition 5.1 in \cite{LNP2015}). 
Moreover, as clarified in \cite{Ledoux2012,ACP2014,CNPP2016}, only some spectral properties of the Markov diffusion operator are essential for deriving a fourth moment type bound for the variance of the corresponding Stein kernel. 
This spectral property is especially satisfied when each $X_i$ is either a Gaussian or (standardized) gamma variable, so this section focuses on such a situation and derive a high-dimensional CLT for this special case. 

For each $\nu>0$, we denote by $\gamma_\pm(\nu)$ the distribution of the random variable $\pm(X-\nu)/\sqrt{\nu}$ with $X\sim\gamma(\nu)$. 
Also, for every $q\in\mathbb{N}$ we set
\[
\mf{c}_q:=\sum_{r=1}^qr!\binom{q}{r}^2.
\]
\begin{proposition}\label{prop:normal-gamma}
Let us keep the same notation as in Theorem \ref{thm:main} and assume $d\geq2$. 
Let $\bs{Y}=(Y_i)_{i=1}^N$ be a sequence of independent random variables such that the law of $Y_i$ belongs to $\{\mathcal{N}(0,1)\}\cup\{\gamma_+(\nu):\nu>0\}\cup\{\gamma_-(\nu):\nu>0\}$ for all $i$. 
For every $i$, define the constants $v_i$ and $\eta_i$ by
\[
v_i:=
\left\{
\begin{array}{ll}
2 & \text{if }Y_i\sim \mathcal{N}(0,1), \\
2(1+\nu^{-1}) & \text{if }Y_i\sim \gamma_\pm(\nu),       
\end{array}
\right.
\qquad
\eta_i:=
\left\{
\begin{array}{ll}
1 & \text{if }Y_i\sim \mathcal{N}(0,1), \\
1\wedge\sqrt{\nu} & \text{if }Y_i\sim \gamma_\pm(\nu).       
\end{array}
\right.
\]
We also set $w_*=1/2$ if $Y_i\sim \mathcal{N}(0,1)$ for every $i$ and $w_*=1$ otherwise. 
Then, $\kappa_4(Q(f_j;\bs{Y}))\geq0$ for every $j$ and
\begin{equation}\label{bound:normal-gamma}
\sup_{y\in\mathbb{R}^d}\left|\expe{h\left(\Phi_\beta(\bs{Q}(\bs{Y})-y)\right)}-\expe{h\left(\Phi_\beta(Z-y)\right)}\right|
\leq \frac{3}{2}\max\{\|h''\|_\infty,\beta\|h'\|_\infty\}\left(\delta_0[\bs{Q}(\bs{Y})]+C\delta_2[\bs{Q}(\bs{Y})]\right)
\end{equation}
for any $\beta>0$ and $h\in C^\infty_b(\mathbb{R})$, where $C>0$ depends only on $\ol{q}_d$ and 
\begin{multline*}
\delta_2[\bs{Q}(\bs{Y})]:=\max_{1\leq j,k\leq d}\{\ul{\eta}_N^{-1}(\log d)\}^{w_*(q_j+q_k)-1}\left\{1_{\{q_j< q_k\}}\|Q(f_j;\bs{Y})\|_4\kappa_4(Q(f_k;\bs{Y}))^{1/4}\right.\\
\left.+1_{\{q_j=q_k\}}\sqrt{2\kappa_4(Q(f_j;\bs{Y}))+\left(2^{-q_j}\ol{v}_N^{q_j}-1\right)(2q_j)!\mf{c}_{q_j}
\sum_{i=1}^N\influence_i(f_j)^2}\right\}
\end{multline*}
with $\ol{v}_N:=\max_{1\leq i\leq N}v_i$ and $\ul{\eta}_N:=\min_{1\leq i\leq N}\eta_i$.
\end{proposition}

The rest of this section is devoted to the proof of Proposition \ref{prop:normal-gamma}. 
In the remainder of this section, we assume that the probability space $(\Omega,\mathcal{F},P)$ is given by the product probability space $(\prod_{i=1}^N\Omega_i,\bigotimes_{i=1}^N\mathcal{F}_i,\bigotimes_{i=1}^NP_i)$, where
\[
(\Omega_i,\mathcal{F}_i,P_i):=
\left\{
\begin{array}{ll}
(\mathbb{R},\mathcal{B}(\mathbb{R}),\mathcal{N}(0,1)) & \text{if }Y_i\sim \mathcal{N}(0,1), \\
((0,\infty),\mathcal{B}((0,\infty)),\gamma(\nu)) & \text{if }Y_i\sim \gamma_\pm(\nu).       
\end{array}
\right.
\]
Then we realize the variables $Y_1,\dots,Y_N$ as follows: For $\omega=(\omega_1,\dots,\omega_N)\in\Omega$, we define
\[
Y_i(\omega):=
\left\{
\begin{array}{ll}
\omega_i & \text{if }Y_i\sim \mathcal{N}(0,1), \\
\pm(\omega_i-\nu)/\sqrt{\nu} & \text{if }Y_i\sim \gamma_\pm(\nu).       
\end{array}
\right.
\]

\subsection{$\Gamma$-calculus}


Our first aim is to construct a suitable Markov diffusion operator whose eigenspaces contain all the components of $\bs{Q}(\bs{Y})$. 
In the following, for an open subset $U$ of $\mathbb{R}^m$, we write $C^\infty_p(U)$ for the set of all real-valued $C^\infty$ functions on $U$ all of whose partial derivatives have at most polynomial growth. 

First, we denote by $\opL_\text{OU}$ the Ornstein-Uhlenbeck operator on $\mathbb{R}$. 
$\opL_\text{OU}$ is densely defined symmetric operator in $L^2(\mathbb{R},\mc{B}(\mathbb{R}),\mathcal{N}(0,1))$ and given by 
\[
\opL_\text{OU}\phi(x)=\phi''(x)-x\phi'(x)
\]
for any $\phi\in C^\infty_p(\mathbb{R})$ . 
Next, for every $\nu>0$, we write $\opL_\nu$ for the Laguerre operator on $(0,\infty)$ with parameter $\nu$. 
$\opL_\nu$ is densely defined symmetric operator in $L^2((0,\infty),\mc{B}((0,\infty)),\gamma(\nu))$ and given by 
\[
\opL_\nu\phi(x)=x\phi''(x)+(\nu-x)\phi'(x)
\]
for any $\phi\in C^\infty_p((0,\infty))$. 
We then define the operators $\mathcal{L}_1,\dots,\mathcal{L}_N$ by
\[
\mathcal{L}_i:=
\left\{
\begin{array}{ll}
\opL_\text{OU} & \text{if }Y_i\sim \mathcal{N}(0,1), \\
\opL_\nu & \text{if }Y_i\sim \gamma_\pm(\nu).       
\end{array}
\right.
\]
Finally, we construct the densely defined symmetric operator $\opL$ in $L^2(P)$ by tensorization of $\mathcal{L}_1,\dots,\mathcal{L}_N$ (see Section 2.2 of \cite{AMMP2016} for details). 

Let us write $\mathcal{S}=C^\infty_p(\Omega)$. 
We define the \textit{carr\'e du champ operator} of $\opL$ by
\[
\Gamma(F,G)=\frac{1}{2}\left(\opL(FG)-F\opL G-G\opL F\right)
\]
for all $F,G\in\mathcal{S}$. 
Since $\opL$ is symmetric and $\opL1=0$, we have the following \textit{integration by parts formula}: For any $F,G\in\mathcal{S}$, it holds that
\begin{equation}\label{eq:IBP}
\expe{\Gamma(F,G)}=-\expe{F\opL G}=-\expe{G\opL F}.
\end{equation}
In addition to this formula, we will use the following properties of the operators $\opL$ and $\Gamma$:
\begin{enumerate}[label=(\alph*)]

\item\label{cond:diffusion} \textit{Diffusion}: For any $F_1,\dots,F_k,G\in\mathcal{S}$ and $\Psi\in C^\infty_b(\mathbb{R}^k)$, it holds that
\[
\Gamma(\Psi(F_1,\dots,F_k),G)=\sum_{j=1}^k\partial_j\Psi(F_1,\dots,F_k)\Gamma(F_j,G).
\]

\item\label{cond:decomp} \textit{Spectral decomposition}: The spectrum of the operator $-\opL$ is given by $\mathbb{Z}_+$ and we have
\[
L^2(P)=\bigoplus_{k=0}^\infty\kernel(\opL+k\id).
\]

\item\label{cond:stable} \textit{Spectral stability}: If $F$ and $G$ are eigenfunctions of $-\opL$ associated with eigenvalues $p$ and $q$ respectively, 
\[
FG\in\bigoplus_{k=0}^{p+q}\kernel(\opL+k\id).
\]

\end{enumerate}
To check that these three properties are indeed satisfied, it is enough to verify that every $\mathcal{L}_i$ satisfies analogous properties (cf.~Section 2.2 of \cite{AMMP2016}). 
The verification that both $\opL_\text{OU}$ and $\opL_\nu$ satisfies properties \ref{cond:diffusion}--\ref{cond:stable} is found in \cite{ACP2014}. In particular, the eigenspaces of $\opL_\text{OU}$ and $\opL_\nu$ associated with eigenvalue $k\in\mathbb{Z}_+$ are given by $\kernel(\opL_\text{OU}+k\id)=\{aH_k:a\in\mathbb{R}\}$ and $\kernel(\opL_\nu+k\id)=\{aL^{(\nu-1)}_k:a\in\mathbb{R}\}$ respectively. Here, $H_k$ denotes the Hermite polynomial with degree $k$:
\[
H_k(x)=(-1)^ke^{x^2/2}\frac{d^k}{dx^k}e^{-x^2/2}.
\] 
Also, $L_k^{(\alpha)}$ denotes the Laguerre polynomial with degree $k$ and parameter $\alpha>-1$:
\[
L_k^{(\alpha)}(x)=\frac{e^xx^{-\alpha}}{k!}\frac{d^k}{dx^k}(e^{-x}x^{k+\alpha}).
\]
Finally, we note that the eigenspace of $\opL$ associated with eigenvalue $k\in\mathbb{Z}_+$ is given by
\begin{equation}\label{eq:spectral}
\kernel(\opL+k\id)=\bigoplus_{\begin{subarray}{c}
k_1+\cdots+k_N=k\\
k_1,\dots,k_N\in\mathbb{Z}_+
\end{subarray}}
\kernel(\mathcal{L}_1+k_1\id)\otimes\cdots\otimes\kernel(\mathcal{L}_N+k_N\id).
\end{equation}
These properties are all we need to know about the operators $\opL$ and $\Gamma$ in the following discussions. 
We refer to \cite{BGL2014} for more details about the properties of these operators. 

Using integration by parts formula \eqref{eq:IBP} and diffusion property \ref{cond:diffusion}, we can construct a Stein kernel for $\bs{Q}(\bs{Y})$ as follows: 
\begin{lemma}\label{gamma-stein}
For every $(i,j)\in[d]^2$, define the function $\tau^{ij}:\mathbb{R}^d\to\mathbb{R}^d\otimes\mathbb{R}^d$ by
\[
\tau^{ij}(x)=\frac{1}{q_j}\expe{\Gamma(Q(f_i;\bs{Y}),Q(f_j;\bs{Y}))\mid \bs{Q}(\bs{Y})=x},\qquad x\in\mathbb{R}^d.
\]
Then $\tau=(\tau^{ij})_{1\leq i,j\leq d}$ is a Stein kernel for $\bs{Q}(\bs{Y})$. 
\end{lemma}

\begin{proof}
This result is shown in \cite[Proposition 5.1]{LNP2015} for a more general setting, but we give the proof for the sake of exposition. 
First, from the definition of $\Gamma$ it is evident that $\tau^{ij}(\bs{Q}(\bs{Y}))=\Gamma(Q(f_i;\bs{Y}),Q(f_j;\bs{Y}))\in L^1(P)$ for every $(i,j)\in[d]^2$. 
Next, from \eqref{eq:spectral} we have $\opL Q(f_j;\bs{Y})=-q_jQ(f_j;\bs{Y})$ for every $j\in[d]$. 
Therefore, for any $\varphi\in C^\infty_b(\mathbb{R}^d)$ we have
\begin{align*}
\sum_{j=1}^d\expe{\partial_j\varphi(\bs{Q}(\bs{Y}))Q(f_j;\bs{Y})}
&=-\sum_{j=1}^dq_j^{-1}\expe{\partial_j\varphi(\bs{Q}(\bs{Y}))\opL Q(f_j;\bs{Y})}\\
&=\sum_{j=1}^dq_j^{-1}\expe{\Gamma(\partial_j\varphi(\bs{Q}(\bs{Y})),Q(f_j;\bs{Y}))}~(\because\text{\eqref{eq:IBP}})\\
&=\sum_{i,j=1}^dq_j^{-1}\expe{\partial_j\varphi(\bs{Q}(\bs{Y}))\Gamma(Q(f_i;\bs{Y}),Q(f_j;\bs{Y}))}~(\because\text{Diffusion property})\\
&=\sum_{i,j=1}^d\expe{\partial_j\varphi(\bs{Q}(\bs{Y}))\tau^{ij}(\bs{Q}(\bs{Y}))}.
\end{align*}
This completes the proof. 
\end{proof}

\subsection{A bound for the variance of the carr\'e du champ operator}\label{sec:gamma-bound}

In view of Lemmas \ref{stein} and \ref{gamma-stein}, we obtain \eqref{bound:normal-gamma} once we show that
\begin{equation}\label{max-normal-gamma}
\ex{\max_{1\leq j,k\leq d}\labs\frac{1}{q_k}\Gamma\lpa Q(f_j;\bs{Y}),Q(f_k;\bs{Y})\rpa-\mf{C}_{jk}\rabs}\leq \delta_0[\bs{Q}(\bs{Y})]+C\delta_2[\bs{Q}(\bs{Y})],
\end{equation}
where $C>0$ depends only on $\ol{q}_d$. 
As a first step, we estimate $\variance[\Gamma\lpa Q(f_j;\bs{Y}),Q(f_k;\bs{Y})\rpa]$ for every $(j,k)\in[d]^2$. More precisely, the aim of this subsection is to prove the following result:
\begin{proposition}\label{gamma-bound}
Let $p\leq q$ be two positive integers. Let $f:[N]^p\to\mathbb{R}$ and $g:[N]^q\to\mathbb{R}$ be symmetric functions vanishing on diagonals and set $F:=Q(f;\bs{Y})$ and $G:=Q(g;\bs{Y})$. Then, $\kappa_4(F)\geq0$, $\kappa_4(G)\geq0$ and
\begin{multline}\label{eq:gamma-bound}
\variance\left[\frac{1}{q}\Gamma(F,G)\right]
\leq1_{\{p<q\}}\sqrt{\expe{F^4}}\sqrt{\kappa_4(G)}\\     
+1_{\{p=q\}}\left\{2\sqrt{\kappa_4(F)}\sqrt{\kappa_4(G)} 
+\left(2^{-p}\ol{v}_N^p-1\right)
(2p)!\mf{c}_p
\sqrt{\sum_{i=1}^N\influence_i(f)^2}
\sqrt{\sum_{i=1}^N\influence_i(g)^2}
\right\}.
\end{multline}
\end{proposition}

Before starting the proof, we remark how this result is related to the preceding studies. 
When $f=g$, \citet{ACP2014} have derived a better estimate than \eqref{eq:gamma-bound} in a more general setting. 
Their technique of the proof can also be applied to the case $f\neq g$, and this has been implemented in \citet{CNPP2016}. However, this leads to a bound containing the quantity $\covariance[F^2,G^2]-2\ex{FG}^2$, so we need an additional argument to estimate it. 
For this reason, we take an alternative route for the proof, which is inspired by the discussions in \citet{Zheng2017} as well as \cite[Proposition 3.6]{CNPP2016}. As a byproduct of this strategy, we obtain inequality \eqref{var-ineq} which leads to the universality of gamma variables. 

We begin by introducing some notation. We write $J_k$ for the orthogonal projection of $L^2(P)$ onto the eigenspace $\kernel(\opL+k\id)$. 
For every $i$ we define the random variable $\mathfrak{p}_2(Y_i)$ by
\[
\mathfrak{p}_2(Y_i):=
\left\{
\begin{array}{ll}
H_2(Y_i) & \text{if }Y_i\sim \mathcal{N}(0,1), \\
\pm\frac{2}{\nu}L_2^{(\nu-1)}(\pm\sqrt{\nu}(Y_i+1)) & \text{if }Y_i\sim \gamma_\pm(\nu).
\end{array}
\right.
\]
The following lemma is a direct consequence of a simple computation, so we omit the proof. 
\begin{lemma}\label{lemma:moment}
For every $i$, we have $\expe{\mathfrak{p}_2(Y_i)}=\expe{Y_i\mathfrak{p}_2(Y_i)}=0$ and $\expe{\mf{p}_2(Y_i)^2}=v_i$.
\end{lemma}
%
%
%
Next, for any function $h:[N]^r\to\mathbb{R}$, we define its symmetrization $\wt{h}:[N]^r\to\mathbb{R}$ by
\[
\wt{h}(i_1,\dots,i_r)=\frac{1}{r!}\sum_{\sigma\in\mathfrak{S}_r}h(i_{\sigma(1)},\dots,i_{\sigma(r)}).
\]
We write $f\wtimes g$ for the symmetrization of $f\otimes g$. 
Given another function $h':[N]^r\to\mathbb{R}$, we define
\[
\langle h,h'\rangle:=\sum_{i_1,\dots,i_r=1}^Nh(i_1,\dots,i_r)h'(i_1,\dots,i_r).
\]
Note that $\|h\|_{\ell_2}^2=\langle h,h\rangle$. 
For every $r\in\{0,1,\dots,p\wedge q\}$, we define the function $f\wstar{r}g:[N]^{p+q-r}\to\mathbb{R}$ by
\begin{multline*}
f\wstar{r}g(i_1,\dots,i_{p+q-2r},k_1,\dots,k_r)\\
:=\frac{1}{(p+q-2r)!}\sum_{\sigma\in\mathfrak{S}_{p+q-2r}}f(i_{\sigma(1)},\dots,i_{\sigma(p-r)},k_1,\dots,k_{r})g(i_{\sigma(p-r+1)},\dots,i_{\sigma(p+q-2r)},k_1,\dots,k_r).
\end{multline*}
Note that we have
\begin{equation}\label{eq:star}
\wt{f\star_rg}(i_1,\dots,i_{p+q-2r})=\sum_{(k_1,\dots,k_r)\in\Delta_r^N}f\wstar{r}g(i_1,\dots,i_{p+q-2r},k_1,\dots,k_r).
\end{equation}
Here, recall that the contraction $f\star_rg$ is defined by \eqref{def:contraction}. 
Finally, we set $\Delta_q^N:=\{(i_1,\dots,i_q)\in[N]^q:i_j\neq i_k\text{ if }j\neq k\}$. 

\begin{lemma}\label{lemma:multiple}
Let $f^*:[N]^p\to\mathbb{R}$ and $g^*:[N]^q\to\mathbb{R}$ be two symmetric functions vanishing on diagonals. Then 
\begin{multline*}
\left(\sum_{i_1,\dots,i_p=1}^Nf^*(i_1,\dots,i_p)\right)\left(\sum_{j_1,\dots,j_q=1}^Ng^*(j_1,\dots,j_q)\right)\\
=\sum_{r=0}^{p\wedge q}r!\binom{p}{r}\binom{q}{r}\sum_{(i_1,\dots,i_{p+q-2r})\in\Delta^N_{p+q-2r}}\wt{f^*\star_rg^*}(i_1,\dots,i_{p+q-2r}).
\end{multline*}
\end{lemma}

\begin{proof} 
Consider a probability space $(\Omega',\mathcal{F}',P')$ on which we have a sequence $\bs{\epsilon}=(\epsilon_1,\dots,\epsilon_N)$ of independent Rademacher variables such that $P'(\epsilon_i=1)=P'(\epsilon_i=-1)=1/2$ for every $i$. By \cite[Proposition 2.9]{NPR2010ejp} we have
\begin{align*}
Q(f^*;\bs{\epsilon})Q(g^*;\bs{\epsilon})=\sum_{r=0}^{p\wedge q}r!\binom{p}{r}\binom{q}{r}Q(\wt{f^*\star_rg^*}1_{\Delta^N_{p+q-2r}};\bs{\epsilon})
\end{align*}
with probability 1. Since $P'(\epsilon_1=\cdots=\epsilon_N=1)=2^{-N}>0$, we obtain the desired result. 
\end{proof}

%
%
%

\begin{lemma}\label{lemma:projection}
Under the assumptions of Proposition \ref{prop:normal-gamma}, we have
\begin{align*}
J_{p+q}(FG)=\sum_{r=0}^{p\wedge q}r!\binom{p}{r}\binom{q}{r}\sum_{(i_1,\dots,i_{p+q-r})\in\Delta^N_{p+q-r}}f\wh{\star_r^0}g(i_{1},\dots,i_{p+q-r})Y_{i_1}\cdots Y_{i_{p+q-2r}}\mathfrak{p}_2(Y_{i_{p+q-2r+1}})\cdots \mathfrak{p}_2(Y_{i_{p+q-r}}).
\end{align*}
\end{lemma}

\begin{proof}
Applying Lemma \ref{lemma:multiple} with $f^*(i_1,\dots,i_p)=f(i_1,\dots,i_p)Y_{i_1}\cdots Y_{i_p}$, $g^*(i_1,\dots,i_p)=g(i_1,\dots,i_p)Y_{i_1}\cdots Y_{i_p}$ and using \eqref{eq:star}, we obtain
\begin{align}
FG
&=\sum_{r=0}^{p\wedge q}r!\binom{p}{r}\binom{q}{r}\sum_{(i_1,\dots,i_{p+q-r})\in\Delta^N_{p+q-r}}f\wstar{r}g(i_{1},\dots,i_{p+q-r})Y_{i_1}\cdots Y_{i_{p+q-2r}}Y_{i_{p+q-2r+1}}^2\cdots Y_{i_{p+q-r}}^2.\label{FG-expr}
\end{align}
For $(i_1,\dots,i_{p+q-2r})\in\Delta^N_{p+q-2r}$ we evidently have
\[
J_{p+q}(Y_{i_1}\cdots Y_{i_{p+q-2r}}Y_{i_{p+q-2r+1}}^2\cdots Y_{i_{p+q-r}}^2)
=Y_{i_1}\cdots Y_{i_{p+q-2r}}\mathfrak{p}_2(Y_{i_{p+q-2r+1}})\cdots \mathfrak{p}_2(Y_{i_{p+q-r}}).
\]
This completes the proof of the lemma. 
\end{proof}

The following lemma is an immediate consequence of product formulae for multiple Wiener-It\^o integrals with respect to an isonormal Gaussian process (see e.g.~\cite[Theorem 2.7.10]{NP2012}), so we omit the proof. 
\begin{lemma}\label{lemma:wiener}
Under the assumptions of Proposition \ref{prop:normal-gamma}, if $Y_i\sim \mathcal{N}(0,1)$ for all $i$, we have $\expe{J_{p+q}(FG)^2}=(p+q)!\|f\wtimes g\|_{\ell_2}^2$ and $\expe{J_{p}(F^2)J_q(G^2)}=1_{\{p=q\}}p!^2\langle f\wtimes f, g\wtimes g\rangle$.
\end{lemma}
The following lemma can be shown in a similar manner to the proof of \cite[Lemma 3.3]{DK2019} (it is in fact derived from Eq.(50) in \cite{DK2019} and the H\"older inequality):
\begin{lemma}\label{lemma:dk}
Let $f:[N]^p\to\mathbb{R}$ be a symmetric function vanishing on diagonals. Then we have $\| f\wtimes f1_{(\Delta^N_{2p})^c}\|^2_{\ell_2}\leq\mf{c}_p\sum_{i=1}^N\influence_i(f)^2$. 
\end{lemma}

The next lemma is a key part in our proof.  
\begin{lemma}\label{lemma:key}
Under the assumptions of Proposition \ref{prop:normal-gamma}, we have
\begin{equation}\label{key:eq1}
\expe{J_{p+q}(FG)^2}\geq(p+q)!\|f\wtimes g\|_{\ell_2}^2.
\end{equation}
Moreover, if $p=q$, we have
\begin{equation}\label{key:eq2}
\left|\expe{J_{2p}(F^2)J_{2p}(G^2)}-(2p)!\langle f\wtimes f,g\wtimes g\rangle\right|
\leq\left(2^{-p}\ol{v}_N^p-1\right)
(2p)!\mf{c}_p
\sqrt{\sum_{i=1}^N\influence_i(f)^2}
\sqrt{\sum_{i=1}^N\influence_i(g)^2}.
\end{equation}
\end{lemma}

\begin{proof}
Let $(i_1,\dots,i_{p+q-r})\in\Delta^N_{p+q-r}$ and $(j_1,\dots,j_{p+q-l})\in\Delta^N_{p+q-l}$. Thanks to Lemma \ref{lemma:moment}, the quantity $\expe{Y_{i_1}\cdots Y_{i_{p+q-2r}}\mathfrak{p}_2(Y_{i_{p+q-2r+1}})\cdots \mathfrak{p}_2(Y_{i_{p+q-r}})Y_{j_1}\cdots Y_{j_{p+q-2l}}\mathfrak{p}_2(Y_{j_{p+q-2l+1}})\cdots \mathfrak{p}_2(Y_{j_{p+q-l}})}$ does not vanish if and only if the following condition is satisfied: 
\begin{enumerate}[label={($\star$)}]

\item \label{cond:perm} $(i_1,\dots,i_{p+q-2r})$ is a permutation of $(j_1,\dots,j_{p+q-2l})$ and $(i_{p+q-2r+1},\dots,i_{p+q-r})$ is a permutation of $(j_{p+q-2l+1},\dots,j_{p+q-l})$. 

\end{enumerate}
Note that the condition \ref{cond:perm} can hold true only if $r=l$. 
Moreover, if the condition \ref{cond:perm} is satisfied, we have
\begin{multline*}
\expe{Y_{i_1}\cdots Y_{i_{p+q-2r}}\mathfrak{p}_2(Y_{i_{p+q-2r+1}})\cdots \mathfrak{p}_2(Y_{i_{p+q-r}})Y_{j_1}\cdots Y_{j_{p+q-2l}}\mathfrak{p}_2(Y_{j_{p+q-2l+1}})\cdots \mathfrak{p}_2(Y_{j_{p+q-l}})}\\
=v_{i_{p+q-2r+1}}\cdots v_{i_{p+q-r}}
\end{multline*}
by Lemma \ref{lemma:moment}. Since there are totally $(p+q-2r)!$ permutations of $(i_1,\dots,i_{p+q-2r})$ and $r!$ permutations of $(i_{p+q-2r+1},\dots,i_{p+q-r})$ respectively, from Lemma \ref{lemma:projection} we infer that
\begin{multline}\label{eq:var-FG}
\expe{J_{p+q}(FG)^2}\\
=\sum_{r=0}^{p\wedge q}r!^2\binom{p}{r}^2\binom{q}{r}^2(p+q-2r)!r!\sum_{(i_1,\dots,i_{p+q-r})\in\Delta^N_{p+q-r}}f\wstar{r}g(i_{1},\dots,i_{p+q-r})^2v_{i_{p+q-2r+1}}\cdots v_{i_{p+q-r}}.
\end{multline}
Now, \eqref{eq:var-FG} is especially true when all the $Y_i$'s follow the standard normal distribution. 
Therefore, Lemma \ref{lemma:wiener} yields 
\[
(p+q)!\|f\wtimes g\|_{\ell_2}^2=\sum_{r=0}^{p\wedge q}r!^2\binom{p}{r}^2\binom{q}{r}^2(p+q-2r)!r!\sum_{(i_1,\dots,i_{p+q-r})\in\Delta^N_{p+q-r}}f\wstar{r}g(i_{1},\dots,i_{p+q-r})^2\cdot 2^r.
\]
Combining this formula with \eqref{eq:var-FG}, we obtain \eqref{key:eq1}. 

Next we prove \eqref{key:eq2}. An argument analogous to the proof of \eqref{eq:var-FG} yields
\begin{multline*}
\expe{J_{2p}(F^2)J_{2p}(G^2)}\\
=\sum_{r=0}^{p}r!^2\binom{p}{r}^4(2p-2r)!r!\sum_{(i_1,\dots,i_{2p-r})\in\Delta^N_{2p-r}}f\wstar{r}f(i_{1},\dots,i_{2p-r})g\wstar{r}g(i_{1},\dots,i_{2p-r})v_{i_{2p-2r+1}}\cdots v_{i_{2p-r}}.
\end{multline*}
Moreover, taking account of the case that all the $Y_i$'s follow the standard normal distribution as above, we obtain
\begin{align}
(2p)!\langle f\wtimes f,g\wtimes g\rangle
=\sum_{r=0}^{p}r!^2\binom{p}{r}^4(2p-2r)!r!\sum_{(i_1,\dots,i_{2p-r})\in\Delta^N_{2p-r}}f\wstar{r}f(i_{1},\dots,i_{2p-r})g\wstar{r}g(i_{1},\dots,i_{2p-r})\cdot 2^r.\label{eq:gaussian2}
\end{align}
Hence we have
\begin{multline*}
\expe{J_{2p}(F^2)J_{2p}(G^2)}-(2p)!\langle f\wtimes f,g\wtimes g\rangle\\
=\sum_{r=1}^{p}r!^2\binom{p}{r}^4(2p-2r)!r!\sum_{(i_1,\dots,i_{2p-r})\in\Delta^N_{2p-r}}f\wstar{r}f(i_{1},\dots,i_{2p-r})g\wstar{r}g(i_{1},\dots,i_{2p-r})\left(v_{i_{2p-2r+1}}\cdots v_{i_{2p-r}}-2^r\right).
\end{multline*}
Therefore, we deduce that
\begin{align*}
&\left|\expe{J_{2p}(F^2)J_{2p}(G^2)}-(2p)!\langle f\wtimes f,g\wtimes g\rangle\right|\\
&\leq\left(2^{-p}\ol{v}_N^p-1\right)\sum_{r=1}^{p}r!^2\binom{p}{r}^4(2p-2r)!r!\sum_{(i_1,\dots,i_{2p-r})\in\Delta^N_{2p-r}}\left|f\wstar{r}f(i_{1},\dots,i_{2p-r})g\wstar{r}g(i_{1},\dots,i_{2p-r})\right|\cdot2^r\\
&\leq\left(2^{-p}\ol{v}_N^p-1\right)\sqrt{\sum_{r=1}^{p}r!^2\binom{p}{r}^4(2p-2r)!r!\sum_{(i_1,\dots,i_{2p-r})\in\Delta^N_{2p-r}}f\wstar{r}f(i_{1},\dots,i_{2p-r})^2\cdot2^r}\\
&\times\sqrt{\sum_{r=1}^{p}r!^2\binom{p}{r}^4(2p-2r)!r!\sum_{(i_1,\dots,i_{2p-r})\in\Delta^N_{2p-r}}g\wstar{r}g(i_{1},\dots,i_{2p-r})^2\cdot2^r},
\end{align*}
where we used the Schwarz inequality in the third line. 
Now, from \eqref{eq:gaussian2} we also have
\begin{align*}
&\sum_{r=1}^{p}r!^2\binom{p}{r}^4(2p-2r)!r!\sum_{(i_1,\dots,i_{2p-r})\in\Delta^N_{2p-r}}f\wstar{r}f(i_{1},\dots,i_{2p-r})g\wstar{r}g(i_{1},\dots,i_{2p-r})\cdot 2^r\\
&=(2p)!\left\langle f\wtimes f,g\wtimes g1_{(\Delta^N_{2p})^c}\right\rangle.
\end{align*}
Thus we infer that
\begin{align*}
&\left|\expe{J_{2p}(F^2)J_{2p}(G^2)}-(2p)!\langle f\wtimes f,g\wtimes g\rangle\right|\\
&\leq\left(2^{-p}\ol{v}_N^p-1\right)\sqrt{(2p)!\| f\wtimes f1_{(\Delta^N_{2p})^c}\|^2_{\ell_2}}
\sqrt{(2p)!\| g\wtimes g1_{(\Delta^N_{2p})^c}\|^2_{\ell_2}}\\
&\leq\left(2^{-p}\ol{v}_N^p-1\right)
\sqrt{(2p)!\mf{c}_p\sum_{i=1}^N\influence_i(f)^2}
\sqrt{(2p)!\mf{c}_p\sum_{i=1}^N\influence_i(g)^2},
\end{align*}
where we apply Lemma \ref{lemma:dk} to obtain the last inequality. 
This completes the proof. 
\end{proof}

%
%

\begin{lemma}\label{lemma:zheng}
Under the assumptions of Proposition \ref{prop:normal-gamma}, we have
\begin{equation}\label{cov-ineq}
\sum_{k=1}^{p+q-1}\expe{J_k(FG)^2}\leq \covariance[F^2,G^2]-2\expe{FG}^2
\end{equation}
and
\begin{equation}\label{var-ineq}
\sum_{k=1}^{2p-1}\expe{J_k(F^2)^2}+p!^2\sum_{r=1}^{p-1}\binom{p}{r}^2\|f\star_rf\|_{\ell_2}^2\leq \expe{F^4}-3\expe{F^2}^2.
\end{equation}
\end{lemma}

\begin{proof}
The proof is parallel to that of \cite[Lemma 3.1]{Zheng2017}. First, we have
\begin{align*}
\expe{F^2G^2}=\ex{\left(\sum_{k=0}^{p+q}J_k(FG)\right)^2}
=\expe{FG}^2+\sum_{k=1}^{p+q}\expe{J_k(FG)^2}.
\end{align*}
Therefore, Lemma \ref{lemma:key} yields
\begin{equation}\label{eq:lower}
\expe{F^2G^2}
\geq \expe{FG}^2+\sum_{k=1}^{p+q-1}\expe{J_k(FG)^2}+(p+q)!\|f\wtimes g\|_{\ell_2}^2.
\end{equation}
Now, by \cite[Lemma 2.2]{NR2014} we have
\begin{equation}\label{NR2014}
\|f\wtimes g\|_{\ell_2}^2=\frac{p!q!}{(p+q)!}\sum_{r=0}^{p\wedge q}\binom{p}{r}\binom{q}{r}\|f\star_rg\|_{\ell_2}^2.
\end{equation}
Hence it holds that
\[
(p+q)!\|f\wtimes g\|_{\ell_2}^2\geq p!q!\|f\|_{\ell_2}^2\|g\|_{\ell_2}^2+1_{\{p=q\}}p!^2\langle f,g\rangle^2
=\expe{F^2}\expe{G^2}+\expe{FG}^2.
\]
Consequently, we obtain
\begin{align*}
\expe{F^2G^2}
\geq 2\expe{FG}^2+\sum_{k=1}^{p+q-1}\expe{J_k(FG)^2}+\expe{F^2}\expe{G^2}.
\end{align*}
This implies \eqref{cov-ineq}. 

Next we prove \eqref{var-ineq}. By \eqref{NR2014} we have
\[
(2p)!\|f\wtimes f\|_{\ell_2}^2=2\expe{F^2}^2+p!^2\sum_{r=1}^{p-1}\binom{p}{r}\binom{q}{r}\|f\star_rf\|_{\ell_2}^2.
\]
Combining this identity with \eqref{eq:lower}, we infer that
\[
\expe{F^4}
\geq 3\expe{F^2}^2+\sum_{k=1}^{p+q-1}\expe{J_k(F^2)^2}+p!^2\sum_{r=1}^{p-1}\binom{p}{r}\binom{q}{r}\|f\star_rf\|_{\ell_2}^2.
\]
Hence we obtain the desired result. 
\end{proof}

\begin{proof}[Proof of Proposition \ref{gamma-bound}]
The non-negativity of $\kappa_4(F)$ and $\kappa_4(G)$ follows from \eqref{var-ineq}.

The remaining proof is parallel to Step 2 in the proof of \cite[Theorem 1.2]{Zheng2017}. 
By definition we have
\begin{equation}\label{gamma-expr}
\Gamma(F,G)=\frac{1}{2}\left(\opL(FG)+qFG+pGF\right)
=\frac{p+q}{2}\expe{FG}+\sum_{k=1}^{p+q-1}\frac{p+q-k}{2}J_k(FG).
\end{equation}
Hence it holds that
\begin{equation*}
\variance\left[\Gamma(F,G)\right]
=\sum_{k=1}^{p+q-1}\frac{(p+q-k)^2}{4}\ex{J_k(FG)^2}
\leq q^2\sum_{k=1}^{p+q-1}\expe{J_k(FG)^2}.
\end{equation*}
Therefore, according to \eqref{cov-ineq}, the desired result follows once we show that
\begin{multline}\label{aim:normal-gamma}
|\covariance[F^2,G^2]-2\expe{FG}^2|
\leq
1_{\{p<q\}}\sqrt{\expe{F^4}}\sqrt{\kappa_4(G)}\\     
+1_{\{p=q\}}\left\{2\sqrt{\kappa_4(F)}\sqrt{\kappa_4(G)} 
+\left(2^{-p}\ol{v}_N^p-1\right)
\sqrt{\expe{F^2}\mathcal{M}(f)}
\sqrt{\expe{G^2}\mathcal{M}(g)}\right\}.  
\end{multline}
To prove \eqref{aim:normal-gamma}, we consider the following decomposition:
\begin{align*}
\expe{F^2G^2}
&=\ex{F^2\left(\expe{G^2}+\sum_{k=1}^{2q-1}J_k(G^2)+J_{2q}(G^2)\right)}\\
&=\expe{F^2}\expe{G^2}+\ex{F^2\sum_{k=1}^{2q-1}J_k(G^2)}+\expe{J_{2q}(F^2)J_{2q}(G^2)}.
\end{align*}
Now, if $p<q$, $\expe{FG}=\expe{J_{2q}(F^2)J_{2q}(G^2)}=0$ and
\begin{align*}
|\covariance[F^2,G^2]|
&=\left|\ex{F^2\sum_{k=1}^{2q-1}J_k(G^2)}\right|
\leq\sqrt{\expe{F^4}}\sqrt{\sum_{k=1}^{2q-1}\expe{J_k(G^2)^2}}.
\end{align*}
Therefore, by \eqref{var-ineq} we obtain
\begin{align*}
|\covariance[F^2,G^2]-2\expe{FG}^2|
\leq\sqrt{\expe{F^4}}\sqrt{\kappa_4(G)}.
\end{align*}
Meanwhile, if $p=q$, we have
\begin{align*}
\labs\covariance[F^2,G^2]-2\expe{FG}^2\rabs
\leq\sum_{k=1}^{2q-1}\expe{J_k(F^2)J_k(G^2)}
+\labs\expe{J_{2q}(F^2)J_{2q}(G^2)}-2\expe{FG}^2\rabs.
\end{align*}
The Schwarz inequality and \eqref{var-ineq} yield
\[
\sum_{k=1}^{2q-1}\expe{J_k(F^2)J_k(G^2)}
\leq\sqrt{\kappa_4(F)}\sqrt{\kappa_4(G)}.
\]
Also, the triangular inequality and \eqref{key:eq2} yield
\begin{align*}
&\labs\expe{J_{2q}(F^2)J_{2q}(G^2)}-2\expe{FG}^2\rabs\\
&\leq\left(2^{-p}\ol{v}_N^p-1\right)
\sqrt{\expe{F^2}\mathcal{M}(f)}
\sqrt{\expe{G^2}\mathcal{M}(g)}
+\labs(2q)!\langle f\wtimes f,g\wtimes g\rangle-2\expe{FG}^2\rabs.
\end{align*}
Now, by \cite[Lemma 2.2]{NR2014} we have
\begin{align*}
(2q)!\langle f\wtimes f,g\wtimes g\rangle
&=2q!^2\langle f,g\rangle^2
+\sum_{r=1}^{q-1}q!^2\binom{q}{r}^2\langle f\star_r g,g\star_r f\rangle\\
&=2\expe{FG}^2+\sum_{r=1}^{q-1}q!^2\binom{q}{r}^2\langle f\star_{r} g,g\star_{r} f\rangle.
\end{align*}
Therefore, by the Schwarz inequality and \eqref{var-ineq} we deduce that
\begin{align*}
&\left|(2q)!\langle f\wtimes f,g\wtimes g\rangle-2\expe{FG}^2\right|\\
&\leq\sum_{r=1}^{q-1}q!^2\binom{q}{r}^2\|f\star_{r} g\|_{\ell_2}\|g\star_{r} f\|_{\ell_2}
=\sum_{r=1}^{q-1}q!^2\binom{q}{r}^2\langle f\star_{q-r}f, g\star_{q-r} g\rangle\\
&\leq\sum_{r=1}^{q-1}q!^2\binom{q}{r}^2\| f\star_{q-r}f\|_{\ell_2}\|g\star_{q-r} g\|_{\ell_2}
=\sum_{r=1}^{q-1}q!^2\binom{q}{r}^2\| f\star_{r}f\|_{\ell_2}\|g\star_{r} g\|_{\ell_2}\\
&\leq\sqrt{\sum_{r=1}^{q-1}q!^2\binom{q}{r}^2\| f\star_{r}f\|_{\ell_2}^2}\sqrt{\sum_{r=1}^{q-1}q!^2\binom{q}{r}^2\| g\star_{r}g\|_{\ell_2}^2}
\leq\sqrt{\kappa_4(F)}\sqrt{\kappa_4(G)}.
\end{align*}
This complete the proof.
\end{proof}

%
%
%
%

\subsection{Hypercontraction principle}

To derive \eqref{max-normal-gamma} from Proposition \ref{gamma-bound}, we use several properties of the $\psi_\alpha$-norm which are collected in Appendix \ref{sec:psi}. 
For this purpose we need to control the quantity $\|\Gamma(\bs{Q}(f_j;\bs{Y}),\bs{Q}(f_k;\bs{Y}))-\expe{\Gamma(\bs{Q}(f_j;\bs{Y}),\bs{Q}(f_k;\bs{Y}))}\|_{\psi_\alpha}$ for some $\alpha>0$ by $\variance[\Gamma(\bs{Q}(f_j;\bs{Y}),\bs{Q}(f_k;\bs{Y}))]$. 
This is accomplished by using the following form of the hypercontractivity: 
%
\begin{definition}\label{def:hyper-seq}
Let $\bs{\xi}=(\xi_1,\dots,\xi_N)$ be a sequence of independent random variables such that $\|\xi_i\|_m<\infty$ for all $i\in\mathbb{N}$ and some $m\in\mathbb{N}$. For every $i\in\mathbb{N}$, let $(Q_{\xi_i,k})_{k=0}^m$ be an orthogonal polynomial sequence with respect to the law of $\xi_i$. That is, $Q_{\xi_i,k}$ is a polynomial of degree $k$ and $\{Q_{\xi_i,k}(\xi_i):k=0,1,\dots,m\}$ is an orthogonal set in $L^2(P)$.   
We say that $\bs{\xi}$ is \textit{$(p,q,\eta)$-hypercontractive up to degree $m$} for $1\leq p<q<\infty$ and $0<\eta<1$ if 
\begin{equation}\label{eq:hyper}
\left\|\sum_{k_1,\dots,k_N=0}^m\varphi(k_1,\dots,k_N)\prod_{i=1}^N\eta^{k_i}Q_{\xi_i,k_i}(\xi_i)\right\|_q\leq\left\|\sum_{k_1,\dots,k_N=0}^m\varphi(k_1,\dots,k_N)\prod_{i=1}^NQ_{\xi_i,k_i}(\xi_i)\right\|_p
\end{equation}
for every function $\varphi:\{0,1\dots,m\}^N\to\mathbb{R}$.
\end{definition}
Note that the above definition does not depend on the choice of the orthogonal polynomial sequences $(Q_{\xi_i,k})_{k=0}^N$ because each $Q_{\xi_i,k}$ is uniquely determined up to multiplicative constants. 
\begin{rmk}
Definition \ref{def:hyper-seq} is related to various concepts of hypercontractivity used in the literature as follows.
\begin{enumerate}

\item \citet{KS1988PTRF} introduced hypercontractibity for a \textit{single} random variable taking values in a Banach space. For a centered real-valued random variable, this notion is equivalent to the hypercontractivity up to degree 1 as a sequence of length 1 in our sense. 

\item In \citet[Section 5]{Janson1997}, the notion of hypercontractivity is defined for a \textit{set} of random variables. This notion requires \eqref{eq:hyper} type inequalities to hold true for all polynomials in random variables belonging to the set, so it is stronger than ours. 

\item \citet{MOO2010} defined hypercontractivity for a \textit{sequence of ensembles} of random variables. This notion requires that a special class of polynomial satisfies a \eqref{eq:hyper} type inequality with slightly different powers of $\eta$, and it is usually weaker than ours. 

\end{enumerate} 
\end{rmk}

We have the following analog of \cite[Lemma 5.3]{Janson1997} and \cite[Proposition 3.11]{MOO2010}: 
\begin{lemma}\label{hyper-union}
Let $\bs{\xi}=(\xi_1,\dots,\xi_N)$ and $\bs{\theta}=(\theta_1,\dots,\theta_M)$ be two sequences of independent random variables. Suppose that $\bs{\xi}$ and $\bs{\theta}$ are independent and that both $\bs{\xi}$ and $\bs{\theta}$ are $(p,q,\eta)$-hypercontractive up to degree $m$ for some $1\leq p<q<\infty$, $0<\eta<1$ and $m\in\mathbb{N}$. Then the sequence $(\xi_1,\dots,\xi_N,\theta_1,\dots\theta_M)$ is $(p,q,\eta)$-hypercontractive up to degree $m$.  
\end{lemma}

\begin{proof}
The lemma can be shown in the same way as in the analogous proofs of \cite{Janson1997,MOO2010}. 
Take a function $\varphi:\{0,1\dots,m\}^{N+M}\to\mathbb{R}$ arbitrarily. Then, noting that $\bs{\xi}$ and $\bs{\theta}$ are independent, we have
\begin{align*}
&\left\|\sum_{k_1,\dots,k_N,l_1,\dots,l_M=0}^m\varphi(k_1,\dots,k_N,l_1,\dots,l_M)\prod_{i=1}^N\eta^{k_i}Q_{\xi_i,k_i}(\xi_i)\prod_{j=1}^M\eta^{l_j}Q_{\theta_j,l_j}(\theta_j)\right\|_q\\
&=\left\|\left\{\ex{\left|\sum_{k_1,\dots,k_N,l_1,\dots,l_M=0}^m\varphi(k_1,\dots,k_N,l_1,\dots,l_M)\prod_{i=1}^N\eta^{k_i}Q_{\xi_i,k_i}(\xi_i)\prod_{j=1}^M\eta^{l_j}Q_{\theta_j,l_j}(\theta_j)\right|^q\mid\bs{\theta}}\right\}^{1/q}\right\|_q\\
&\leq\left\|\left(\ex{\left|\sum_{k_1,\dots,k_N,l_1,\dots,l_M=0}^m\varphi(k_1,\dots,k_N,l_1,\dots,l_M)\prod_{i=1}^NQ_{\xi_i,k_i}(\xi_i)\prod_{j=1}^M\eta^{l_j}Q_{\theta_j,l_j}(\theta_j)\right|^p\mid\bs{\theta}}\right)^{1/p}\right\|_q\\
&\leq\left\|\left(\ex{\left|\sum_{k_1,\dots,k_N,l_1,\dots,l_M=0}^m\varphi(k_1,\dots,k_N,l_1,\dots,l_M)\prod_{i=1}^NQ_{\xi_i,k_i}(\xi_i)\prod_{j=1}^M\eta^{l_j}Q_{\theta_j,l_j}(\theta_j)\right|^q\mid\bs{\xi}}\right)^{1/q}\right\|_p\\
&\leq\left\|\left(\ex{\left|\sum_{k_1,\dots,k_N,l_1,\dots,l_M=0}^m\varphi(k_1,\dots,k_N,l_1,\dots,l_M)\prod_{i=1}^NQ_{\xi_i,k_i}(\xi_i)\prod_{j=1}^MQ_{\theta_j,l_j}(\theta_j)\right|^p\mid\bs{\xi}}\right)^{1/p}\right\|_p\\
&=\left\|\sum_{k_1,\dots,k_N,l_1,\dots,l_M=0}^m\varphi(k_1,\dots,k_N,l_1,\dots,l_M)\prod_{i=1}^NQ_{\xi_i,k_i}(\xi_i)\prod_{j=1}^MQ_{\theta_j,l_j}(\theta_j)\right\|_p.
\end{align*}
Here, the second and the fourth inequalities follow from the hypercontractivity of $\bs{\xi}$ and $\bs{\theta}$, while the third inequality is a consequence of \cite[Proposition C.4]{Janson1997}. This completes the proof. 
\end{proof}

%

According to Lemma \ref{hyper-union}, we can establish the hypercontractivity of the sequence $\bs{Y}$ once we prove this property for each $Y_i$ as a sequence of length 1. 
For a standard normal variable, this is an immediate consequence of \cite[Theorem 5.1]{Janson1997}:
\begin{lemma}\label{normal-hyper}
Let $X\sim \mathcal{N}(0,1)$. Then, $X$ is $(p,q,\eta)$-hypercontractive up to degree $m$ (as a sequence of length 1) for any $m\in\mathbb{N}$, $1< p< q<\infty$ and $0<\eta\leq\sqrt{(p-1)/(q-1)}$. 
\end{lemma}

For a gamma variable $X\sim\gamma(\nu)$, the situation is slightly complicated. When $\nu\geq\frac{1}{2}$, the desired property follows from the hypercontractivity of the semigroup with generator $\opL_\nu$, but when $\nu<\frac{1}{2}$, we need an additional argument to obtain a hypercontractivity constant appropriate to our purpose. 

For $\nu>0$ and $k\in\mathbb{N}$, we define
\[
\varsigma_{\nu,k}:=\sqrt{\variance[L_k^{(\nu-1)}(X)]},\qquad X\sim\gamma(\nu),
\]
and $\eta_{\nu,k}:=\min\{1,\varsigma_{\nu,k}/\varsigma_{1/2,k}\}$. 
From Eq.(8.980) in \cite{GR2007} we have
\[
\varsigma_{\nu,k}^2=\binom{\nu+k-1}{k}=\frac{(\nu+k-1)(\nu+k-2)\cdots(\nu+1)\nu}{k!}.
\]
In particular, $\varsigma_{\nu,k}$ is increasing in $\nu$. Thus, we have $\eta_{\nu,k}=1$ if and only if $\nu\geq1/2$. 

\begin{lemma}\label{gamma-hyper}
Let $X\sim\gamma(\nu)$ for some $\nu>0$. 
Then, $X$ is $(2,q,\eta(q-1)^{-1})$-hypercontractive up to degree $m$ (as a sequence of length 1) for any $m\in\mathbb{N}$, $q>2$ and $0<\eta\leq\min_{1\leq k\leq m}\eta_{\nu,k}^{1/k}$. 
\end{lemma}

\begin{proof}
We need to verify that
\begin{equation}\label{aim:gamma-hyper}
\left\|\sum_{k=0}^m\eta^k(q-1)^{-k}a_kL_k^{(\nu-1)}(X)\right\|_q\leq\left\|\sum_{k=0}^ma_kL_k^{(\nu-1)}(X)\right\|_2
\end{equation}
for all $a_0,a_1,\dots,a_m\in\mathbb{R}$.

First we consider the case $\nu\geq\frac{1}{2}$. Let $(P_t^\nu)_{t\geq0}$ be the semigroup with generator $\opL_\nu$. By Theorem 2 in \cite{Korzeniowski1987} and Gross' hypercontractivity theorem we obtain 
$
\|P_t^\nu f\|_{L^q(\gamma(\nu))}\leq\|f\|_{L^2(\gamma(\nu))}
$ 
for any $f\in L^2(\gamma(\nu))$ and $t\geq0$ such that $e^t\geq q-1$. Setting $t:=\log(q-1)$, we have $P_t^\nu L_k^{(\nu-1)}=(q-1)^{-k}L_k^{(\nu-1)}$ for every $k\in\mathbb{Z}_+$, so we obtain
\begin{align*}
\left\|\sum_{k=0}^m\eta^k(q-1)^{-k}a_kL_k^{(\nu-1)}(X)\right\|_q
&=\left\|P_t^\nu\lpa\sum_{k=0}^m\eta^ka_kL_k^{(\nu-1)}\rpa\right\|_{L^q(\gamma(\nu))}
\leq\left\|\sum_{k=0}^m\eta^ka_kL_k^{(\nu-1)}\right\|_{L^2(\gamma(\nu))}\\
&=\sqrt{\sum_{k=0}^m\eta^{2k}a_k^2L_k^{(\nu-1)}(X)^2}
\leq\left\|\sum_{k=0}^ma_kL_k^{(\nu-1)}(X)\right\|_2.
\end{align*}
Hence we get \eqref{aim:gamma-hyper}. 

Next we consider the case $\nu<\frac{1}{2}$. Extending the probability space if necessary, we take a random variable $X'\sim\gamma(\frac{1}{2}-\nu)$ independent of $X$. Since $X+X'\sim\gamma(\frac{1}{2})$ and $\eta_{1/2,k}=1$ for every $k$, from the above result we have
\begin{align*}
&\left\|\sum_{k=0}^m\eta^k(q-1)^{-k}a_kL_k^{(-1/2)}(X+X')\right\|_q
\leq\left\|\sum_{k=0}^m\eta^ka_kL_k^{(-1/2)}(X+X')\right\|_2\\
&=a_0^2+\sum_{k=1}^m\eta^{2k}a_k^2\varsigma_{1/2,k}^2
\leq a_0^2+\sum_{k=1}^ma_k^2\varsigma_{\nu,k}^2
=\left\|\sum_{k=0}^ma_kL_k^{(\nu-1)}(X)\right\|_2^2.
\end{align*}
Meanwhile, Eq.(8.977.1) in \cite{GR2007} yields
\[
L_k^{(-1/2)}(X+X')
=\sum_{j=0}^kL_{k-j}^{(\nu-1)}(X)L_{j}^{(-1/2-\nu)}(X')
=L_k^{(\nu-1)}(X)+\sum_{j=1}^kL_{k-j}^{(\nu-1)}(X)L_{j}^{(-1/2-\nu)}(X')
\]
for every $k$. 
Since we have $E[\sum_{j=1}^kL_{k-j}^{(\nu-1)}(X)L_{j}^{(-1/2-\nu)}(X')\mid X]=0$, Lemma 1.1 in \cite{KS1988PTRF} implies that
\begin{align*}
\left\|\sum_{k=0}^m\eta^k(q-1)^{-k}a_kL_k^{(\nu-1)}(X)\right\|_q
&\leq\left\|\sum_{k=0}^m\eta^k(q-1)^{-k}a_kL_k^{(-1/2)}(X+X')\right\|_q.
\end{align*}
Hence we obtain \eqref{aim:gamma-hyper}.
\end{proof}

Now we are ready to prove the following result. 
\begin{lemma}\label{gamma-psi}
Under the assumptions of Proposition \ref{prop:normal-gamma}, we have
\[
\|\Gamma(F,G)-\expe{\Gamma(F,G)}\|_{\psi_{(w_*(p+q)-1)^{-1}}}\leq C_{p,q}\ul{\eta}_N^{-(w_*(p+q)-1)}\sqrt{\variance[\Gamma(F,G)]},
\]
where $C_{p,q}>0$ depends only on $p,q$. 
\end{lemma}

\begin{proof}
First we consider the case $w_*=1$, i.e.~the general case. 
Since $\eta_{\nu,1}=1\wedge\sqrt{2\nu}\geq1\wedge\sqrt{\nu}$ and $\eta_{\nu,2}=\min\{1,\sqrt{4\nu(\nu+1)/3}\}\geq1\wedge\nu$, Lemmas \ref{normal-hyper}--\ref{gamma-hyper} imply that $Y_i$ is $(2,r,\ul{\eta}_N(r-1)^{-1})$-hypercontractive up to degree 2 for all $r>2$ and $i\in[N]$. Therefore, the sequence $\bs{Y}$ is $(2,r,\ul{\eta}_N(r-1)^{-1})$-hypercontractive up to degree 2 for all $r>2$ by Lemma \ref{hyper-union}. 
Now, combining this fact with \eqref{eq:spectral}, \eqref{FG-expr} and \eqref{gamma-expr}, we obtain
\begin{align*}
\|\Gamma(F,G)-\expe{\Gamma(F,G)}\|_r
&\leq\left\|\sum_{k=1}^{p+q-1}\frac{p+q-k}{2}\ul{\eta}_N^{-k}(r-1)^{k}J_k(FG)\right\|_2\\
&\leq\ul{\eta}_N^{-(p+q-1)}(r-1)^{p+q-1}\sqrt{\sum_{k=1}^{p+q-1}\frac{(p+q-k)^2}{4}J_k(FG)^2}\\
&=\ul{\eta}_N^{-(p+q-1)}(r-1)^{p+q-1}\sqrt{\variance[\Gamma(F,G)]}
\end{align*}
for all $r>2$. Consequently, Lemma \ref{moment-psi} yields the desired result.

Next we consider the case $w_*=1/2$, i.e.~we assume $Y_i\sim \mathcal{N}(0,1)$ for every $i\in[N]$. In this case we have $\ul{\eta}_N=1$. Moreover, product formulae for multiple Wiener-It\^o integrals with respect to an isonormal Gaussian process (see e.g.~\cite[Theorem 2.7.10]{NP2012}) yield $J_{p+q-1}(FG)=0$. Therefore, \eqref{gamma-expr} implies that the variable $\Gamma(F,G)-\expe{\Gamma(F,G)}$ is a polynomial of degree $\leq p+q-2$ in the variables $Y_1,\dots,Y_N$. Hence, by Theorem 5.11 and Remark 5.11 in \cite{Janson1997} we obtain 
$
\|\Gamma(F,G)-\expe{\Gamma(F,G)}\|_r
\leq(r-1)^{(p+q-2)/2}\sqrt{\variance[\Gamma(F,G)]}
$ 
for all $r>2$. Consequently, Lemma \ref{moment-psi} again yields the desired result.
\end{proof}

\subsection{Proof of Proposition \ref{prop:normal-gamma}}

We have already established the non-negativity of $\kappa_4(Q(f_j;\bs{Y}))$'s in Proposition \ref{gamma-bound}. 
The remaining claim of the proposition follows once we prove \eqref{max-normal-gamma}. 
Noting that $\expe{\Gamma(Q(f_j;\bs{Y}),Q(f_k;\bs{Y}))}=q_j\expe{Q(f_j;\bs{Y})Q(f_k;\bs{Y})}$ for every $(j,k)\in[d]^2$ by \eqref{eq:IBP}, we obtain \eqref{max-normal-gamma} by Proposition \ref{gamma-bound}, Lemmas \ref{gamma-psi} and \ref{max-psi}. \hfill\qed

\section{Randomized Lindeberg method}\label{sec:lindeberg}

For any $\varpi\geq0$ and $x\geq0$, we set 
\[
\chi_\varpi(x)=
\left\{\begin{array}{ll}
\exp(-x^{1/\varpi}) & \text{if }\varpi>0,\\
1_{[0,1)}(x) & \text{if }\varpi=0.
\end{array}\right.
\]
The aim of this section is to prove the following result. 
\begin{proposition}\label{prop:lindeberg}
Set $\Lambda_i:=(\log d)^{(\ol{q}_d-1)/\alpha}\max_{1\leq k\leq d}M_N^{q_k-1}\sqrt{\influence_{i}(f_k)}$ for $i\in[N]$. 
Let $\bs{X}=(X_i)_{i=1}^N$ and $\bs{Y}=(Y_i)_{i=1}^N$ be two sequences of independent centered random variables with unit variance. 
Suppose that $M_N:=\max_{1\leq i\leq N}(\|X_i\|_{\psi_\alpha}\vee\|Y_i\|_{\psi_\alpha})<\infty$ for some $\alpha\in(0,2]$. Suppose also that there is an integer $m\geq3$ such that $\expe{X_i^r}=\expe{Y_i^r}$ for all $i\in[N]$ and $r\in[m-1]$. Then, for any $h\in C^m_b(\mathbb{R})$, $\beta>0$ and $\tau,\rho\geq0$ with $\tau\rho M_N\max_{1\leq i\leq N}\Lambda_i\leq\beta^{-1}$, we have 
\begin{align}
&\sup_{y\in\mathbb{R}^d}\left|\ex{h\left(\Phi_\beta(\bs{Q}(\bs{X})-y)\right)}-\ex{h\left(\Phi_\beta(\bs{Q}(\bs{Y})-y)\right)}\right|\nonumber\\
&\leq C\lpa\max_{1\leq j\leq m}\beta^{m-j}\|h^{(j)}\|_\infty\rpa\left\{(\log d)^{m(\ol{q}_d-1)/\alpha}\max_{1\leq k\leq d}M_N^{mq_k}\sum_{i=1}^N\influence_i(f_k)^{m/2}\right.\nonumber\\
&\qquad\left.+\lpa 
e^{-\lpa\frac{\tau}{K_1}\rpa^{\alpha}}
+\chi_{(\ol{q}_d-1)/\alpha}\lpa\frac{\rho}{K_2}\rpa(\rho\vee1)^m
+\exp\lpa-\lpa\frac{\tau\rho}{K_3}\rpa^{\alpha/\ol{q}_d}\rpa(\tau\rho\vee1)^m
\rpa
M_N^m\sum_{i=1}^N\Lambda_{i}^m\right\},\label{bound:lindeberg}
\end{align}
where $C>0$ depends only on $m,\alpha,\ol{q}_d$, $K_1$ depends only on $\alpha$, and $K_2,K_3>0$ depend only on $\alpha,\ol{q}_d$. 
\end{proposition}

\begin{rmk}
Proposition \ref{prop:lindeberg} can be viewed as a version of \cite[Theorem 7.1]{NPR2010aop}. Apart from that we take account of higher moment matching, there are important differences between these two results. 
On the one hand, the latter takes all $C^3$ functions with bounded third-order partial derivatives as test functions, while the former focus only on test functions of the form $x\mapsto h(\Phi_\beta(x-y))$ for some $h\in C^m_b(\mathbb{R})$ and $y\in\mathbb{R}^d$. 
On the other hand, in the bound of \eqref{bound:lindeberg}, terms like $\sum_{i=1}^N\max_{1\leq k\leq d}\influence_i(f_k)$ always appear with exponential factors, so we can remove such terms by appropriately selecting the parameters $\tau,\rho$. In contrast, such a quantity appears (as the constant $\mathsf{C}$) in the dominant term of the bound given by \cite[Theorem 7.1]{NPR2010aop}. As pointed out in Remark \ref{rmk:npr}(b), this point can be crucial in a high-dimensional setting, and this phenomenon originates from a (na\"ive) application of the Lindeberg method. To avoid this difficulty, we use a randomized version of the Lindeberg method, which was originally introduced in \citet{DZ2017} for sums of independent random variables. 

Here, we briefly describe how we shall randomize the standard Lindeberg method. To fix the idea, we focus on the case $d=q_1=1$. Roughly speaking, using a certain interpolation technique, the standard Lindeberg method replaces the variables $X_1,\dots,X_N$ by $Y_1,\dots,Y_N$ one by one \textit{with this order} (see e.g.~the proof of \cite[Proposition 11.1.3]{NP2012} for more details). However, given that $\sum_{i=1}^Nf_1(i)x_i$ is invariant under permutation of $((1,x_1),\dots,(N,x_N))$, there is no reason to keep the order of replacement to apply the Lindeberg method. Namely, for any $\sigma\in\mf{S}_N$, we can apply the Lindeberg method to replace $X_{\sigma(1)},\dots,X_{\sigma(N)}$ by $Y_{\sigma(1)},\dots,Y_{\sigma(N)}$, then we obtain another bound for the difference between $\sum_{i=1}^Nf_1(i)X_i=\sum_{i=1}^Nf_1(\sigma(i))X_{\sigma(i)}$ and $\sum_{i=1}^Nf_1(i)Y_i=\sum_{i=1}^Nf_1(\sigma(i))Y_{\sigma(i)}$. Now we can deduce a new bound by averaging all such bounds over $\sigma\in\mf{S}_N$. Namely, we randomly choose $\sigma\in\mf{S}_N$ and construct a Lindeberg type bound associated with $\sigma$, then we take the expectation with respect to $\sigma$. 
\end{rmk}

For the proof we need three auxiliary results. The first one is a generalization of \cite[Lemma S.5.1]{KC2018}:
\begin{lemma}\label{lemma:kc}
Let $\xi$ be a non-negative random variable such that $P(\xi>x)\leq Ae^{-(x/B)^\alpha}$ for all $x\geq0$ and some constants $A,B,\alpha>0$. Then we have 
\[
\ex{\xi^p1_{\{\xi>t\}}}\leq A\left(1+\frac{2p-\alpha}{p-\alpha}\right)\lpa t\vee\{(2(p/\alpha-1))^{1/\alpha}B\}\rpa^pe^{-(t/B)^\alpha}
\]
for any $p>\alpha$ and $t>0$.
\end{lemma}

\begin{proof}
By \cite[Theorem 8.16]{Rudin1987} and a change of variables we obtain
\begin{align*}
\ex{\xi^p1_{\{\xi>t\}}}
&\leq At^pe^{-(t/B)^\alpha}+\frac{pAB^p}{\alpha}\int_{(t/B)^\alpha}^\infty y^{p/\alpha-1}e^{-y}dy.
\end{align*}
Now if $(t/B)^\alpha\geq2(p/\alpha-1)$, Eq.(3.2) of \cite{NP2000} yields
\[
\int_{(t/B)^\alpha}^\infty y^{p/\alpha-1}e^{-y}dy\leq2(t/B)^{p-\alpha}e^{-(t/B)^\alpha}.
\]
Hence we obtain
\begin{align*}
\expe{\xi^p1_{\{\xi>t\}}}
&\leq A\left(1+\frac{p}{2(p-\alpha)}\right)t^pe^{-(t/B)^\alpha}
\leq A\left(1+\frac{2p-\alpha}{p-\alpha}\right)t^pe^{-(t/B)^\alpha},
\end{align*}
where we use $p<2(2p-\alpha)$ to obtain the last inequality. 
Meanwhile, if $(t/B)^\alpha<2(p/\alpha-1)$, we have
\begin{align*}
\int_{(t/B)^\alpha}^\infty y^{p/\alpha-1}e^{-y}dy
&=\int_{(t/B)^\alpha}^{2(p/\alpha-1)} y^{p/\alpha-1}e^{-y}dy+\int_{2(p/\alpha-1)}^\infty y^{p/\alpha-1}e^{-y}dy\\
&\leq\frac{2^{p/\alpha}(p/\alpha-1)^{p/\alpha}}{p/\alpha}e^{-(t/B)^\alpha}+2(2(p/\alpha-1))^{p/\alpha-1}e^{-(2(p/\alpha-1))^\alpha}\\
&\leq\frac{\alpha(2p-\alpha)}{p(p-\alpha)}(2(p/\alpha-1))^{p/\alpha}e^{-(t/B)^\alpha},
\end{align*}
where we again use Eq.(3.2) of \cite{NP2000} to estimate the second term in the right side of the first line. Consequently, we obtain 
\begin{align*}
\ex{\xi^p1_{\{\xi>t\}}}
&\leq A\left(1+\frac{2p-\alpha}{p-\alpha}\right)(2(p/\alpha-1))^{p/\alpha}B^pe^{-(t/B)^\alpha}.
\end{align*}
This completes the proof.
\end{proof}

The second one is a moment inequality for homogeneous sums with a sharp constant:
\begin{lemma}\label{lem:hom-mom}
Let $\bs{X}=(X_i)_{i=1}^N$ be a sequence of independent centered random variables. Suppose that $M:=\max_{1\leq i\leq N}\|X_i\|_{\psi_\alpha}<\infty$ for some $\alpha\in(0,2]$. 
Also, let $q\in\mathbb{N}$ and $f:[N]^q\to\mathbb{R}$ be a symmetric function vanishing on diagonals. Then we have
\[
\left\|Q(f;\bs{X})\right\|_p\leq K_{\alpha,q}p^{q/\alpha}M^q\|f\|_{\ell_2}
\]
for any $p\geq2$, where $K_{\alpha,q}>0$ depends only on $\alpha,q$. 
\end{lemma}
Since we need additional lemmas to prove Lemma \ref{lem:hom-mom}, we postpone its proof to Appendix \ref{sec:hom-mom}. 

The third one is a well-known elementary fact and immediately follows from the commutativity of addition, but it will deserve to be explicitly stated for later reference. 
\begin{lemma}\label{lemma:commute}
Let $S$ be a finite set and $\varphi$ be a real-valued function on $S$. Also, let $b:S\to S$ be a bijection. Then we have $\sum_{x\in A}\varphi(b(x))=\sum_{x\in b(A)}\varphi(x)$ for any $A\subset S$. 
\end{lemma}

Now we turn to the main body of the proof. Throughout the proof, we will use the standard multi-index notation. 
For a multi-index $\lambda=(\lambda_1,\dots,\lambda_d)\in\mathbb{Z}_+^d$, we set $|\lambda|:=\lambda_1+\cdots+\lambda_d$, $\lambda!:=\lambda_1!\cdots\lambda_d!$ and $\partial^\lambda:=\partial_1^{\lambda_1}\cdots\partial_d^{\lambda_d}$ as usual. Also, given a vector $x=(x_1,\dots,x_d)\in\mathbb{R}^d$, we write $x^\lambda=x_1^{\lambda_1}\cdots x_d^{\lambda_d}$. 
\begin{proof}[Proof of Proposition \ref{prop:lindeberg}]
Without loss of generality, we may assume that $\bs{X}$ and $\bs{Y}$ are independent. 
Throughout the proof, for two real numbers $a$ and $b$, the notation $a\lesssim b$ means that $a\leq cb$ for some constant $c>0$ which depends only on $m,\alpha,\ol{q}_d$. 

Take a vector $y\in\mathbb{R}^d$ and define the function $\Psi:\mathbb{R}^d\to\mathbb{R}$ by $\Psi(x)=h(\Phi_\beta(x-y))$ for $x\in\mathbb{R}^d$. 
For any $i\in[N]$, $\sigma\in\mathfrak{S}_N$ and $k\in[d]$, we define
\[
\bs{W}^\sigma_i=(W_{i,1}^\sigma,\dots,W_{i,N}^\sigma):=(X_{\sigma(1)},\dots,X_{\sigma(i)},Y_{\sigma(i+1)},\dots,Y_{\sigma(N)})
\]
and
\begin{align*}
U_{k,i}^\sigma&:=\sum_{\begin{subarray}{c}
i_1,\dots,i_{q_k}=1\\
i_1\neq i,\dots,i_{q_k}\neq i
\end{subarray}}^N
f_k(\sigma(i_1),\dots,\sigma(i_{q_k}))W^\sigma_{i,i_1}\cdots W^\sigma_{i,i_{q_k}},\\
V_{k,i}^\sigma&:=\sum_{\begin{subarray}{c}
i_1,\dots,i_{q_k}=1\\
\exists j:i_j= i
\end{subarray}}^N
f_k(\sigma(i_1),\dots,\sigma(i_{q_k}))\prod_{l:i_l\neq i}W_{i,i_l}^\sigma.
\end{align*}
Then we set $\bs{U}^\sigma_{i}=(U_{k,i}^\sigma)_{k=1}^d$ and $\bs{V}^\sigma_{i}=(V^\sigma_{k,i})_{k=1}^d$. 
By construction $\bs{U}^\sigma_{i}$ and $\bs{V}^\sigma_{i}$ are independent of $X_{\sigma(i)}$ and $Y_{\sigma(i)}$. 
Moreover, we have
$Q(f_k;\bs{W}^\sigma_{i-1})=U^\sigma_{k,i}+Y_{\sigma(i)}V^\sigma_{k,i}$ and $Q(f_k;\bs{W}^\sigma_{i})=U^\sigma_{k,i}+X_{\sigma(i)}V^\sigma_{k,i}$ (with $\bs{W}^\sigma_0:=(Y_{\sigma(1)},\dots,Y_{\sigma(N)})$). 
In particular, by Lemma \ref{lemma:commute} it holds that
\begin{align*}
Q(f_k;\bs{W}^\sigma_{0})&=\sum_{i_1,\dots,i_{q_k}=1}^Nf_k(\sigma(i_1),\dots,\sigma(i_{q_k}))Y_{\sigma(i_1)}\cdots Y_{\sigma(i_{q_k})}
=Q(f_k;\bs{Y}),\\
Q(f_k;\bs{W}^\sigma_{N})&=\sum_{i_1,\dots,i_{q_k}=1}^Nf_k(\sigma(i_1),\dots,\sigma(i_{q_k}))X_{\sigma(i_1)}\cdots X_{\sigma(i_{q_k})}
=Q(f_k;\bs{X}).
\end{align*}
Therefore, we obtain
\begin{align}
\left|\ex{\Psi(\bs{Q}(\bs{X}))}-\ex{\Psi(\bs{Q}(\bs{Y}))}\right|
&=\frac{1}{N!}\sum_{\sigma\in\mathcal{S}_N}\left|\ex{\Psi(\bs{Q}(\bs{W}^\sigma_N))}-\ex{\Psi(\bs{Q}(\bs{W}^\sigma_0))}\right|\nonumber\\
&\leq\frac{1}{N!}\sum_{\sigma\in\mathcal{S}_N}\sum_{i=1}^N\left|\ex{\Psi(\bs{Q}(\bs{W}^\sigma_i))}-\ex{\Psi(\bs{Q}(\bs{W}^\sigma_{i-1}))}\right|.\label{lindeberg-eq1}
\end{align}

Now, Taylor's theorem and the independence of $X_{\sigma(i)}$ and $Y_{\sigma(i)}$ from $\bs{U}_{i}^\sigma$ and $\bs{V}_{i}^\sigma$ yield
\begin{multline*}
\ex{\Psi(\bs{U}^\sigma_i+\xi\bs{V}^\sigma_{i})}
=\sum_{\lambda\in\mathbb{Z}_+^d:|\lambda|\leq m-1}\frac{1}{\lambda!}\ex{\partial^\lambda\Psi(\bs{U}^\sigma_i)\left(\bs{V}_i^\sigma\right)^\lambda}\ex{\xi^{|\lambda|}}\\
+\sum_{\lambda\in\mathbb{Z}_+^d:|\lambda|=m}\frac{m}{\lambda!}\int_0^1(1-t)^{m-1}\ex{\partial^\lambda\Psi(\bs{U}^\sigma_i+t\xi\bs{V}^\sigma_i)\xi^{m}(\bs{V}^\sigma_{i})^\lambda}dt
\end{multline*}
when $\xi=X_{\sigma(i)}$ or $\xi=Y_{\sigma(i)}$. 
Since we have $\expe{X_i^r}=\expe{Y_i^r}$ for all $i\in[N]$ and $r\in[m-1]$ by assumption, we obtain
\begin{align}
&\labs\ex{\Psi\lpa\bs{Q}^\sigma_{i})}-\ex{\Psi(\bs{Q}^\sigma_{i-1}\rpa}\rabs\nonumber\\
&\leq\sum_{\lambda\in\mathbb{Z}_+^d:|\lambda|=m}\frac{m}{\lambda!}\int_0^1(1-t)^{m-1}\ex{\labs\partial^\lambda\Psi\lpa\bs{U}^\sigma_i+tX_{\sigma(i)}\bs{V}^\sigma_i\rpa\rabs\labs X_{\sigma(i)}\rabs^{m}\labs\lpa\bs{V}^\sigma_{i}\rpa^\lambda\rabs}dt\nonumber\\
&\qquad+\sum_{\lambda\in\mathbb{Z}_+^d:|\lambda|=m}\frac{m}{\lambda!}\int_0^1(1-t)^{m-1}\ex{\labs\partial^\lambda\Psi\lpa\bs{U}^\sigma_i+tY_{\sigma(i)}\bs{V}^\sigma_i\rpa\rabs\labs Y_{\sigma(i)}\rabs^{m}\labs\lpa\bs{V}^\sigma_{i}\rpa^\lambda\rabs}dt\nonumber\\
&=:\mathbf{I}_i^\sigma+\mathbf{II}_i^\sigma,\label{lindeberg-eq2}
\end{align}
where $\mathbf{I}_i^\sigma:=\mathbf{I}_i^\sigma[X_{\sigma(i)}]+\mathbf{I}_i^\sigma[Y_{\sigma(i)}]$, $\mathbf{II}_i^\sigma:=\mathbf{II}_i^\sigma[X_{\sigma(i)}]+\mathbf{II}_i^\sigma[Y_{\sigma(i)}]$ and
\begin{align*}
\mathbf{I}_i^\sigma[\xi]&:=\sum_{\lambda\in\mathbb{Z}_+^d:|\lambda|=m}\frac{m}{\lambda!}\int_0^1(1-t)^{m-1}\ex{|\partial^\lambda\Psi(\bs{U}^\sigma_i+t\xi\bs{V}^\sigma_i)||\xi|^{m}|(\bs{V}^\sigma_{i})^\lambda|;\mathcal{E}_{\sigma,i}}dt,\\
\mathbf{II}_i^\sigma[\xi]&:=\sum_{\lambda\in\mathbb{Z}_+^d:|\lambda|=m}\frac{m}{\lambda!}\int_0^1(1-t)^{m-1}\ex{|\partial^\lambda\Psi(\bs{U}^\sigma_i+t\xi\bs{V}^\sigma_i)||\xi|^{m}|(\bs{V}^\sigma_{i})^\lambda|;\mathcal{E}_{\sigma,i}^c}dt
\end{align*}
for $\xi\in\{X_{\sigma(i)},Y_{\sigma(i)}\}$ and $\mathcal{E}_{\sigma,i}:=\{(|X_{\sigma(i)}|+|Y_{\sigma(i)}|)\|\bs{V}^\sigma_i\|_{\ell_\infty}\leq\tau\rho M_N\Lambda_{\sigma(i)}\}$.

First we consider $\mathbf{I}_i^\sigma$. 
Since $\tau\rho M_N\max_{1\leq i\leq N}\Lambda_i\leq\beta^{-1}$ by assumption, Lemma \ref{cck-derivative} yields
\begin{align*}
\mathbf{I}_i^\sigma
&\leq\frac{e^8}{m!}\sum_{j_1,\dots,j_m=1}^d\ex{\Upsilon_\beta^{j_1,\dots, j_m}(\bs{U}^\sigma_i-y)(|X_{\sigma(i)}|^{m}+|Y_{\sigma(i)}|^{m})|V^\sigma_{j_1,i}|^m}.
\end{align*}
Since $X_{\sigma(i)}$ and $Y_{\sigma(i)}$ are independent of $\bs{U}^\sigma_i$ and $\bs{V}^\sigma_i$, we obtain
\begin{align}
\mathbf{I}_i^\sigma
&\leq\frac{e^8}{m!}\sum_{j_1,\dots,j_m=1}^d\ex{\Upsilon_\beta^{j_1,\dots, j_m}(\bs{U}^\sigma_i-y)|V^\sigma_{j_1,i}|^m}\ex{|X_{\sigma(i)}|^{m}+|Y_{\sigma(i)}|^{m}}\nonumber\\
&\leq\frac{e^8}{m!}\sup_{1\leq i\leq N}\expe{|X_{i}|^{m}+|Y_{i}|^{m}}\left\{\mathbf{I}(1)_i^\sigma+\mathbf{I}(2)_i^\sigma+\mathbf{I}(3)_i^\sigma\right\},\label{eq:I-decompose}
\end{align}
where
\begin{align*}
\mathbf{I}(1)_i^\sigma&:=\sum_{j_1,\dots,j_m=1}^d\ex{\Upsilon_\beta^{j_1,\dots, j_m}(\bs{U}^\sigma_i-y)|V^\sigma_{j_1,i}|^m;\mathcal{C}_{\sigma,i}\cap\mathcal{D}_{\sigma,i}},\\
\mathbf{I}(2)_i^\sigma&:=\sum_{j_1,\dots,j_m=1}^d\ex{\Upsilon_\beta^{j_1,\dots, j_m}(\bs{U}^\sigma_i-y)|V^\sigma_{j_1,i}|^m;\mathcal{C}_{\sigma,i}^c},\\
\mathbf{I}(3)_i^\sigma&:=\sum_{j_1,\dots,j_m=1}^d\ex{\Upsilon_\beta^{j_1,\dots, j_m}(\bs{U}^\sigma_i-y)|V^\sigma_{j_1,i}|^m;\mathcal{D}_{\sigma,i}^c}
\end{align*}
and $\mathcal{C}_{\sigma,i}:=\{|X_{\sigma(i)}|+|Y_{\sigma(i)}|\leq\tau M_N\},\mathcal{D}_{\sigma,i}:=\{\|\bs{V}^\sigma_i\|_{\ell_\infty}\leq\rho\Lambda_{\sigma(i)}\}$. 

We begin by estimating $\mathbf{I}(1)_i^\sigma$. Let $(\delta_i)_{i=1}^N$ be a sequence of i.i.d.~Bernoulli variables independent of $\bs{X}$ and $\bs{Y}$ with $P(\delta_i=1)=1-P(\delta_i=0)=i/(N+1)$. We set $\zeta_{i,a}:=\delta_i X_a+(1-\delta_i)Y_a$ for all $i,a\in[N]$.
Then, since $\|\zeta_{i,\sigma(i)}\bs{V}^\sigma_i\|_{\ell_\infty}\leq\tau\rho M_N\max_{1\leq i\leq N}\Lambda_i\leq\beta^{-1}$ on the set $\mathcal{C}_{\sigma,i}\cap\mathcal{D}_{\sigma,i}$, by Lemma \ref{cck-derivative} we obtain
\begin{align*}
\mathbf{I}(1)_i^\sigma
\leq e^8\sum_{j_1,\dots,j_m=1}^d\ex{\Upsilon_\beta^{j_1,\dots, j_m}(\bs{U}^\sigma_i+\zeta_{i,\sigma(i)}\bs{V}^\sigma_i-y)|V^\sigma_{j_1,i}|^m}.
\end{align*}
The subsequent discussions are inspired by the proof of \cite[Lemma 2]{DZ2017} and we introduce some notation analogous to theirs. For any $i,a\in[N]$, we set
\[
\mathcal{A}_{i,a}=\{(A,B):A\subset[N],B\subset[N],A\cup B=[N]\setminus\{a\},\#A=i-1,\#B=N-i\},
\]
where $\#S$ denotes the number of elements in a set $S$. 
We also set
\[
\mathcal{A}_{i}=\{(A,B):A\subset[N],B\subset[N],A\cup B=[N],\#A=i,\#B=N-i\}
\]
for every $i\in\{0,1\dots,N\}$. 
Moreover, for any $A,B\subset[N]$ with $A\cap B=\emptyset$ and $i\in A\cup B$, we define the random variable $W^{(A,B)}_i$ by
\[
W^{(A,B)}_i=
\left\{
\begin{array}{ll}
X_i  & \text{if }i\in A,     \\
Y_i  & \text{if }i\in B.  
\end{array}
\right.
\]
Then we define
\[
Q_k^{(A,B)}:=\sum_{i_1,\dots,i_{q_k}=1}^Nf_k(i_1,\dots,i_{q_k})W^{(A,B)}_{i_1}\cdots W^{(A,B)}_{i_{q_k}}
\]
for any $k\in[d]$ and $(A,B)\in\bigcup_{i=0}^\infty\mc{A}_i$, and set $\bs{Q}^{(A,B)}:=(Q_k^{(A,B)})_{k=1}^d$. 
We also define
\begin{align*}
U_{k,a}^{(A,B)}&:=\sum_{\begin{subarray}{c}
i_1,\dots,i_{q_k}=1\\
i_1\neq a,\dots,i_{q_k}\neq a
\end{subarray}}^N
f_k(i_1,\dots,i_{q_k})W^{(A,B)}_{i_1}\cdots W^{(A,B)}_{i_{q_k}},\\
V_{k,a}^{(A,B)}&:=\sum_{\begin{subarray}{c}
i_1,\dots,i_{q_k}=1\\
\exists j:i_j= a
\end{subarray}}^N
f_k(i_1,\dots,i_{q_k})\prod_{l:i_l\neq a}W^{(A,B)}_{i_l}
\end{align*}
for any $k\in[d]$, $i,a\in[N]$ and $(A,B)\in\bigcup_{j=0}^N\mathcal{A}_{j,a}$ and set $\bs{U}_a^{(A,B)}:=(U_{k,a}^{(A,B)})_{k=1}^d$ and $\bs{V}_a^{(A,B)}:=(V_{k,a}^{(A,B)})_{k=1}^d$. 
Finally, for any $\sigma\in\mathfrak{S}_N$ and $i\in[N]$ we set $A^\sigma_i:=\{\sigma(1),\dots,\sigma(i-1)\}$ and $B^\sigma_i:=\{\sigma(i+1),\dots,\sigma(N)\}$. 

Now, since we have $W^\sigma_{i,j}=W^{(A^\sigma_i,B^\sigma_i)}_{\sigma(j)}$ for $j\in[N]\setminus\{i\}$, it holds that
$\bs{U}^\sigma_{i}=\bs{U}^{(A^\sigma_i,B^\sigma_i)}_{\sigma(i)}$ and $\bs{V}^\sigma_{i}=\bs{V}^{(A^\sigma_i,B^\sigma_i)}_{\sigma(i)}$ by Lemma \ref{lemma:commute}.
Therefore, we obtain
\begin{align*}
&\frac{1}{N!}\sum_{\sigma\in\mathfrak{S}_N}\sum_{i=1}^N\mathbf{I}(1)_i^\sigma\\
&\leq\frac{e^8}{N!}\sum_{\sigma\in\mathfrak{S}_N}\sum_{i=1}^N\sum_{j_1,\dots,j_m=1}^d\ex{\Upsilon_\beta^{j_1,\dots, j_m}\lpa\bs{U}^{(A^\sigma_i,B^\sigma_i)}_{\sigma(i)}+\zeta_{i,\sigma(i)}\bs{V}^{(A^\sigma_i,B^\sigma_i)}_{\sigma(i)}-y\rpa\labs V^{(A^\sigma_i,B^\sigma_i)}_{j_1,\sigma(i)}\rabs^m}\\
&=\frac{e^8}{N!}\sum_{i=1}^N\sum_{a=1}^N\sum_{\sigma\in\mathfrak{S}_N:\sigma(i)=a}\sum_{j_1,\dots,j_m=1}^d\ex{\Upsilon_\beta^{j_1,\dots, j_m}\lpa\bs{U}^{(A^\sigma_i,B^\sigma_i)}_{a}+\zeta_{i,a}\bs{V}^{(A^\sigma_i,B^\sigma_i)}_{a}-y\rpa\labs V^{(A^\sigma_i,B^\sigma_i)}_{j_1,a}\rabs^m}\\
&=\frac{e^8}{N!}\sum_{i=1}^N\sum_{a=1}^N\sum_{(A,B)\in\mathcal{A}_{i,a}}\#\{\sigma\in\mathfrak{S}_N:A^\sigma_i=A,\sigma(i)=a\}\sum_{j_1,\dots,j_m=1}^d\ex{\Upsilon_\beta^{j_1,\dots, j_m}\lpa\bs{U}^{(A,B)}_{a}+\zeta_{i,a}\bs{V}^{(A,B)}_{a}-y\rpa\labs V^{(A,B)}_{j_1,a}\rabs^m}\\
&=e^8\sum_{i=1}^N\frac{(i-1)!(N-i)!}{N!}\sum_{a=1}^N\sum_{(A,B)\in\mathcal{A}_{i,a}}\sum_{j_1,\dots,j_m=1}^d\ex{\Upsilon_\beta^{j_1,\dots, j_m}\lpa\bs{U}^{(A,B)}_{a}+\zeta_{i,a}\bs{V}^{(A,B)}_{a}-y\rpa\labs V^{(A,B)}_{j_1,a}\rabs^m}.
\end{align*}
Now, for $(A,B)\in\mathcal{A}_{i,a}$ we have $\bs{U}^{(A,B)}_{a}+X_a\bs{V}^{(A,B)}_{a}=\bs{Q}^{(A\cup\{a\},B)}$ and  $\bs{U}^{(A,B)}_{a}+Y_a\bs{V}^{(A,B)}_{a}=\bs{Q}^{(A,B\cup\{a\})}$, so we obtain
\begin{align*}
&\frac{1}{N!}\sum_{\sigma\in\mathfrak{S}_N}\sum_{i=1}^N\mathbf{I}(1)_i^\sigma\\
&\leq e^8\sum_{i=1}^N\frac{(i-1)!(N-i)!}{N!}\sum_{a=1}^N\sum_{(A,B)\in\mathcal{A}_{i,a}}\sum_{j_1,\dots,j_m=1}^d\left\{\frac{i}{N+1}\ex{\Upsilon_\beta^{j_1,\dots, j_m}\lpa\bs{Q}^{(A\cup\{a\},B)}-y\rpa\labs V^{(A,B)}_{j_1,a}\rabs^m}\right.\\
&\left.\hphantom{\frac{e^8}{m!}\sum_{i=1}^N\frac{(i-1)!(N-i)!}{N!}}+\frac{N+1-i}{N+1}\ex{\Upsilon_\beta^{j_1,\dots, j_m}\lpa\bs{Q}^{(A,B\cup\{a\})}-y\rpa\labs V^{(A,B)}_{j_1,a}\rabs^m}\right\}\\
&=e^8\sum_{i=1}^N\frac{i!(N-i)!}{(N+1)!}\sum_{a=1}^N\sum_{(A,B)\in\mathcal{A}_{i,a}}\sum_{j_1,\dots,j_m=1}^d\ex{\Upsilon_\beta^{j_1,\dots, j_m}\lpa\bs{Q}^{(A\cup\{a\},B)}-y\rpa\labs V^{(A,B)}_{j_1,a}\rabs^m}\\
&\quad+e^8\sum_{i=1}^N\frac{(i-1)!(N+1-i)!}{(N+1)!}\sum_{a=1}^N\sum_{(A,B)\in\mathcal{A}_{i,a}}\sum_{j_1,\dots,j_m=1}^d\ex{\Upsilon_\beta^{j_1,\dots, j_m}\lpa\bs{Q}^{(A,B\cup\{a\})}-y\rpa\labs V^{(A,B)}_{j_1,a}\rabs^m}\\
&=e^8\sum_{i=1}^N\frac{i!(N-i)!}{(N+1)!}\sum_{a=1}^N\sum_{(A,B)\in\mathcal{A}_{i,a}}\sum_{j_1,\dots,j_m=1}^d\ex{\Upsilon_\beta^{j_1,\dots, j_m}\lpa\bs{Q}^{(A\cup\{a\},B)}-y\rpa\labs V^{(A,B)}_{j_1,a}\rabs^m}\\
&\quad+e^8\sum_{i=0}^{N-1}\frac{i!(N-i)!}{(N+1)!}\sum_{a=1}^N\sum_{(A,B)\in\mathcal{A}_{i+1,a}}\sum_{j_1,\dots,j_m=1}^d\ex{\Upsilon_\beta^{j_1,\dots, j_m}\lpa\bs{Q}^{(A,B\cup\{a\})}-y\rpa\labs V^{(A,B)}_{j_1,a}\rabs^m}\\
&=e^8\sum_{i=1}^N\frac{i!(N-i)!}{(N+1)!}\sum_{a=1}^N\sum_{(A,B)\in\mathcal{A}_{i}:a\in A}\sum_{j_1,\dots,j_m=1}^d\ex{\Upsilon_\beta^{j_1,\dots, j_m}\lpa\bs{Q}^{(A,B)}-y\rpa\labs V^{(A\setminus\{a\},B)}_{j_1,a}\rabs^m}\\
&\quad+e^8\sum_{i=0}^{N-1}\frac{i!(N-i)!}{(N+1)!}\sum_{a=1}^N\sum_{(A,B)\in\mathcal{A}_{i}:a\in B}\sum_{j_1,\dots,j_m=1}^d\ex{\Upsilon_\beta^{j_1,\dots, j_m}\lpa\bs{Q}^{(A,B)}-y\rpa\labs V^{(A,B\setminus\{a\})}_{j_1,a}\rabs^m}\\
&=e^8\sum_{i=0}^N\frac{i!(N-i)!}{(N+1)!}\sum_{(A,B)\in\mathcal{A}_{i}}\sum_{j_1,\dots,j_m=1}^d\ex{\Upsilon_\beta^{j_1,\dots, j_m}\lpa\bs{Q}^{(A,B)}-y\rpa\sum_{a=1}^N\labs V^{(A\setminus\{a\},B\setminus\{a\})}_{j_1,a}\rabs^m}\\
&\leq e^8\sum_{i=0}^N\frac{i!(N-i)!}{(N+1)!}\sum_{(A,B)\in\mathcal{A}_{i}}\sum_{j_1,\dots,j_m=1}^d\ex{\Upsilon_\beta^{j_1,\dots, j_m}\lpa\bs{Q}^{(A,B)}-y\rpa\max_{1\leq k\leq d}\sum_{a=1}^N\labs V^{(A\setminus\{a\},B\setminus\{a\})}_{k,a}\rabs^m}.
\end{align*}
Hence, Lemma \ref{cck-derivative} yields
\begin{align*}
&\frac{1}{N!}\sum_{\sigma\in\mathfrak{S}_N}\sum_{i=1}^N\mathbf{I}(1)_i^\sigma\\
&\lesssim \max_{1\leq j\leq m}\beta^{m-j}\|h^{(j)}\|_\infty\sum_{i=0}^N\frac{i!(N-i)!}{(N+1)!}\sum_{(A,B)\in\mathcal{A}_{i}}\ex{\max_{1\leq k\leq d}\sum_{a=1}^N\labs V^{(A\setminus\{a\},B\setminus\{a\})}_{k,a}\rabs^m}.
\end{align*}
Now, by Lemma \ref{lem:hom-mom} we have
\begin{equation}\label{eq:V-hyper}
\left\|V^{(A\setminus\{a\},B\setminus\{a\})}_{k,a}\right\|_{r}\leq C_{\alpha,\ol{q}_d}r^{(q_k-1)/\alpha}M_N^{q_k-1}\sqrt{\influence_a(f_k)}
\end{equation}
for any $r\geq1$, where $C_{\alpha,\ol{q}_d}>0$ depends only on $\alpha,\ol{q}_d$. Hence, the Minkowski inequality yields
\begin{equation}\label{eq:V-sum}
\left\|\sum_{a=1}^N\labs V^{(A\setminus\{a\},B\setminus\{a\})}_{k,a}\rabs^m\right\|_r
\leq C_{\alpha,\ol{q}_d}(mr)^{m(q_k-1)/\alpha}M_N^{m(q_k-1)}\sum_{a=1}^N\influence_a(f_k)^{m/2}.
\end{equation}
Thus, if $\ol{q}_d>1$, Lemma \ref{moment-psi} yields
\[
\left\|\sum_{a=1}^N\labs V^{(A\setminus\{a\},B\setminus\{a\})}_{k,a}\rabs^m\right\|_{\psi_{\alpha/\{m(\ol{q}_d-1)\}}}
\lesssim M_N^{m(q_k-1)}\sum_{a=1}^N\influence_a(f_k)^{m/2}.
\]
Therefore, by Lemmas \ref{max-psi} and \ref{psi-moment} we conclude that
\[
\ex{\max_{1\leq k\leq d}\sum_{a=1}^N\labs V^{(A\setminus\{a\},B\setminus\{a\})}_{k,a}\rabs^m}
\lesssim (\log d)^{m(\ol{q}_d-1)/\alpha}\max_{1\leq k\leq d}M_N^{m(q_k-1)}\sum_{a=1}^N\influence_a(f_k)^{m/2}.
\]
This inequality also holds true when $\ol{q}_d=1$ because in this case $V^{(A\setminus\{a\},B\setminus\{a\})}_{k,a}$'s are non-random and thus it is a direct consequence of \eqref{eq:V-sum}. 
As a result, we obtain
\begin{equation}\label{est:I1}
\frac{1}{N!}\sum_{\sigma\in\mathfrak{S}_N}\sum_{i=1}^N\mathbf{I}(1)_i^\sigma
\lesssim \lpa\max_{1\leq j\leq m}\beta^{m-j}\|h^{(j)}\|_\infty\rpa\lpa(\log d)^{m(\ol{q}_d-1)/\alpha}\max_{1\leq k\leq d}M_N^{m(q_k-1)}\sum_{a=1}^N\influence_a(f_k)^{m/2}\rpa.
\end{equation} 

Next we estimate $\mathbf{I}(2)^\sigma_i$. 
Since $X_{\sigma(i)}$ and $Y_{\sigma(i)}$ are independent of $\bs{U}^\sigma_i$ and $\bs{V}^\sigma_i$, we have
\begin{align*}
\mathbf{I}(2)_i^\sigma
&\leq\sum_{j_1,\dots,j_m=1}^d\ex{\Upsilon_\beta^{j_1,\dots, j_m}(\bs{U}^\sigma_i-y)\|V^\sigma_{i}\|_{\ell_\infty}^m}P\lpa\mathcal{C}_{\sigma,i}^c\rpa.
\end{align*}
Hence, Lemma \ref{cck-derivative} yields
\begin{align*}
\mathbf{I}(2)_i^\sigma
\lesssim \lpa\max_{1\leq j\leq m}\beta^{m-j}\|h^{(j)}\|_\infty\rpa\ex{\|V^\sigma_{i}\|_{\ell_\infty}^m}P\lpa\mathcal{C}_{\sigma,i}^c\rpa.
\end{align*}
Now, if $\ol{q}_d>1$, \eqref{eq:V-hyper} and Lemma \ref{moment-psi} yield
\[
\|V^\sigma_{k,i}\|_{\psi_{\alpha/(\ol{q}_d-1)}}\leq c_{\alpha,\ol{q}_d}M_N^{q_k-1}\sqrt{\influence_{\sigma(i)}(f_k)},
\] 
where $c_{\alpha,\ol{q}_d}>0$ depends only on $\alpha,\ol{q}_d$. 
Hence, Lemmas \ref{max-psi} and \ref{psi-moment} yield
\begin{equation}\label{eq:Vmax-moment}
\|\|\bs{V}^\sigma_{i}\|_{\ell_\infty}\|_{r}\leq c'_{\alpha,\ol{q}_d}r^{(\ol{q}_d-1)/\alpha}\Lambda_{\sigma(i)}
\end{equation}
for every $r\geq1$ with $c'_{\alpha,\ol{q}_d}>0$ depending only on $\alpha,\ol{q}_d$. This inequality also holds true when $\ol{q}_d=1$ because in this case $\bs{V}^\sigma_{i}$ is non-random and thus it is a direct consequence of \eqref{eq:V-hyper}. 
Meanwhile, \eqref{sum-psi} and Lemma \ref{psi-tail} yield 
$
P(\mathcal{C}_{\sigma,i}^c)\leq 2e^{-(\tau/2^{1\vee\alpha^{-1}})^{\alpha}}.
$ 
Consequently, we obtain
\begin{equation}\label{est:I2}
\frac{1}{N!}\sum_{\sigma\in\mathfrak{S}_N}\sum_{i=1}^N\mathbf{I}(2)_i^\sigma\\
\lesssim \lpa\max_{1\leq j\leq m}\beta^{m-j}\|h^{(j)}\|_\infty\rpa e^{-(\tau/2^{1\vee\alpha^{-1}})^{\alpha}}\sum_{i=1}^N\Lambda_i^m.
\end{equation}

Third, we estimate $\mathbf{I}(3)_i^\sigma$. Lemma \ref{cck-derivative} yields
\[
\mathbf{I}(3)_i^\sigma
\lesssim \lpa\max_{1\leq j\leq m}\beta^{m-j}\|h^{(j)}\|_\infty\rpa\ex{\|\bs{V}^\sigma_{i}\|_{\ell_\infty}^m;\mathcal{D}_{\sigma,i}^c}.
\]
If $\ol{q}_d>1$, \eqref{eq:Vmax-moment} and Lemma \ref{moment-tail} yield
\[
P\lpa\|\bs{V}^\sigma_{i}\|_{\ell_\infty}\geq x\rpa\leq e^{(\ol{q}_d-1)/\alpha}\exp\lpa-\lpa\frac{x}{K_{\alpha,\ol{q}_d}\Lambda_{\sigma(i)}}\rpa^{\alpha/(\ol{q}_d-1)}\rpa
\]
for every $x>0$ with $K_{\alpha,\ol{q}_d}>0$ depending only on $\alpha,\ol{q}_d$. Hence Lemma \ref{lemma:kc} implies that
\[
\ex{\|\bs{V}^\sigma_{i}\|_{\ell_\infty}^m;\mathcal{D}_{\sigma,i}^c}
\lesssim(\rho\vee1)^m\Lambda_{\sigma(i)}^m\exp\lpa-\lpa\frac{\rho}{K_{\alpha,\ol{q}_d}}\rpa^{\alpha/(\ol{q}_d-1)}\rpa.
\]
Meanwhile, if $\ol{q}_d=1$, $\bs{V}^\sigma_{i}$ is non-random, so \eqref{eq:Vmax-moment} yields
\[
\ex{\|\bs{V}^\sigma_{i}\|_{\ell_\infty}^m;\mathcal{D}_{\sigma,i}^c}\lesssim\Lambda_{\sigma(i)}^m1_{\{c'_{\alpha,\ol{q}_d}>\rho\}}.
\]
Consequently, setting $K_{\alpha,\ol{q}_d}':=K_{\alpha,\ol{q}_d}\vee c'_{\alpha,\ol{q}_d}$, we obtain
\begin{equation}\label{est:I3}
\frac{1}{N!}\sum_{\sigma\in\mathfrak{S}_N}\sum_{i=1}^N\mathbf{I}(3)_i^\sigma
\lesssim\lpa\max_{1\leq j\leq m}\beta^{m-j}\|h^{(j)}\|_\infty\rpa\chi_{(\ol{q}_d-1)/\alpha}\lpa\frac{\rho}{K'_{\alpha,\ol{q}_d}}\rpa(\rho\vee1)^m\sum_{i=1}^N\Lambda_{i}^m.
\end{equation}
Now, combining \eqref{eq:I-decompose}, \eqref{est:I1}, \eqref{est:I2}, \eqref{est:I3} with Lemma \ref{psi-moment}, we obtain
\begin{multline}\label{lindeberg-eq3}
\frac{1}{N!}\sum_{\sigma\in\mathfrak{S}_N}\sum_{i=1}^N\mathbf{I}_i^\sigma
\lesssim \lpa\max_{1\leq j\leq m}\beta^{m-j}\|h^{(j)}\|_\infty\rpa\left\{(\log d)^{m(\ol{q}_d-1)/\alpha}\max_{1\leq k\leq d}M_N^{mq_k}\sum_{i=1}^N\influence_i(f_k)^{m/2}\right.\\
\left.+e^{-(\tau/2^{1\vee\alpha^{-1}})^{\alpha}}M_N^{m}\sum_{i=1}^N\Lambda_i^m
+\chi_{(\ol{q}_d-1)/\alpha}\lpa\frac{\rho}{K'_{\alpha,\ol{q}_d}}\rpa(\rho\vee1)^mM_N^{m}\sum_{i=1}^N\Lambda_{i}^m\right\}.
\end{multline} 

Next we consider $\mathbf{II}_i^\sigma$. Lemma \ref{cck-derivative} yields
\begin{align*}
\frac{1}{N!}\sum_{\sigma\in\mathfrak{S}_N}\sum_{i=1}^N\mathbf{II}_i^\sigma
\lesssim \frac{1}{N!}\lpa\max_{1\leq j\leq m}\beta^{m-j}\|h^{(j)}\|_\infty\rpa\sum_{\sigma\in\mathfrak{S}_N}\sum_{i=1}^N\ex{(|X_{\sigma(i)}|^{m}+|Y_{\sigma(i)}|^{m})\max_{1\leq k\leq d}\labs V^\sigma_{k,i}\rabs^m;\mathcal{E}_{\sigma,i}^c}.
\end{align*}
Since $X_{\sigma(i)}$ and $Y_{\sigma(i)}$ are independent of $\bs{V}^\sigma_{i}$, Lemma \ref{psi-moment} and \eqref{eq:Vmax-moment} imply that
\[
\|(|X_{\sigma(i)}|+|Y_{\sigma(i)}|)\|\bs{V}^\sigma_{i}\|_{\ell_\infty}\|_{r}\leq L_{\alpha,\ol{q}_d}r^{\ol{q}_d/\alpha}M_N\Lambda_{\sigma(i)}
\]
for every $r\geq1$ with $L_{\alpha,\ol{q}_d}>0$ depending only on $\alpha,\ol{q}_d$. 
Thus, by Lemma \ref{moment-tail} we obtain
\[
P\lpa(|X_{\sigma(i)}|+|Y_{\sigma(i)}|)\|\bs{V}^\sigma_{i}\|_{\ell_\infty}\geq x\rpa\leq e^{\ol{q}_d/\alpha}\cdot \exp\lpa-\left(\frac{x}{L'_{\alpha,\ol{q}_d}M_N\Lambda_{\sigma(i)}}\right)^{\alpha/\ol{q}_d}\rpa
\]
for every $x>0$ with $L'_{\alpha,\ol{q}_d}>0$ depending only on $\alpha,\ol{q}_d$. Therefore, Lemma \ref{lemma:kc} yields
\begin{align*}
&\frac{1}{N!}\sum_{\sigma\in\mathfrak{S}_N}\sum_{i=1}^N\mathbf{II}_i^\sigma\\
&\lesssim \frac{1}{N!}\lpa\max_{1\leq j\leq m}\beta^{m-j}\|h^{(j)}\|_\infty\rpa\sum_{\sigma\in\mathfrak{S}_N}\sum_{i=1}^N\lpa\tau\rho\vee 1\rpa^mM_N^m\Lambda_{\sigma(i)}^m
\exp\lpa-\left(\frac{\tau\rho}{L'_{\alpha,\ol{q}_d}}\right)^{\alpha/\ol{q}_d}\rpa\\
&=\lpa\max_{1\leq j\leq m}\beta^{m-j}\|h^{(j)}\|_\infty\rpa\sum_{i=1}^N\lpa\tau\rho\vee 1\rpa^mM_N^m\Lambda_{i}^m
\exp\lpa-\left(\frac{\tau\rho}{L'_{\alpha,\ol{q}_d}}\right)^{\alpha/\ol{q}_d}\rpa.
\end{align*}
Combining this inequality with \eqref{lindeberg-eq1}, \eqref{lindeberg-eq2} and \eqref{lindeberg-eq3}, we obtain the desired result. 
\end{proof}

\section{Proof of the main results}\label{sec:proof}

\subsection{Proof of Theorem \ref{thm:main}}

The following result is a version of \cite[Lemma 4.3]{NPR2010aop}. The proof is a minor modification of the latter's, so we omit it.  
\begin{lemma}\label{lemma:npr4.3}
Let $q\in\mathbb{N}$ and $f:[N]^q\to\mathbb{R}$ be a symmetric function vanishing on diagonals. 
Also, let $\bs{X}=(X_i)_{i=1}^N$ and $\bs{Y}=(Y_i)_{i=1}^N$ be two sequences of independent centered random variables with unit variance. 
Suppose that there are integers $3\leq m\leq l$ such that $M_N:=\max_{1\leq i\leq N}(\|X_i\|_{l}\vee\|Y_i\|_{l})<\infty$ and $\expe{X_i^r}=\expe{Y_i^r}$ for all $i\in[N]$ and $r\in[m-1]$. 
Then we have $Q(f;\bs{X}),Q(f;\bs{Y})\in L^l(P)$ and 
\[
|\expe{Q(f;\bs{X})^l}-\expe{Q(f;\bs{Y})^l}|\leq CM_N^{ql}(1\vee\|f\|_{\ell_2})^{l-m}\sum_{i=1}^N\max\{\influence_i(f)^{\frac{m}{2}},\influence_i(f)^{\frac{l}{2}}\},
\]
where $C>0$ depends only on $q,l$.
\end{lemma}

\begin{proof}[Proof of Theorem \ref{thm:main}]
Throughout the proof, for two real numbers $a$ and $b$, the notation $a\lesssim b$ means that $a\leq cb$ for some constant $c>0$ which depends only on $\alpha,\ol{q}_d$. Moreover, if $(\log d)^{\mu+\frac{1}{2}}\delta_1[\bs{Q}(\bs{X})]^{\frac{1}{3}}\geq1$, then the claim evidently holds true with $C=1$, so we may assume $(\log d)^{\mu+\frac{1}{2}}\delta_1[\bs{Q}(\bs{X})]^{\frac{1}{3}}<1$. 

Set $s_i:=\expe{X_i^3}$ for every $i$. We take a sequence $\bs{Y}=(Y_i)_{i=1}^N$ of independent random variables such that
\[
Y_i\sim\left\{
\begin{array}{cl}
\mathcal{N}(0,1)  & \text{if }s_i=0, \\
\gamma_+(4/s_i^2)  &  \text{if }s_i>0,   \\
\gamma_-(4/s_i^2)   &  \text{if }s_i<0.
\end{array}
\right.
\]
By construction we have $\expe{X_i^r}=\expe{Y_i^r}$ for any $i\in[N]$ and $r\in[3]$. 
Moreover, Lemmas \ref{normal-hyper}--\ref{gamma-hyper} imply that $\|Y_i\|_r\leq \ol{B}_N(r-1)^w$ for any $i\in[N]$ and $r\geq2$. 
Hence, by Lemma \ref{moment-psi} we have $\max_{1\leq i\leq N}\|Y_i\|_{\psi_\alpha}\leq c_\alpha \ol{B}_N$ with $c_\alpha\geq1$ depending only on $\alpha$. 
Therefore, applying Proposition \ref{prop:lindeberg} with $m=4$, we obtain
\begin{align*}
&\Delta_\varepsilon(\bs{Q}(\bs{X}),\bs{Q}(\bs{Y}))\\
&\leq C_1\varepsilon^{-4}(\log d)^3
\left\{(\log d)^{4(\ol{q}_d-1)/\alpha}\max_{1\leq k\leq d}\ol{B}_N^{4q_k}\sum_{i=1}^N\influence_i(f_k)^{2}
\right.\\
&\qquad\left.+\lpa
e^{-\lpa\frac{\tau}{K_1}\rpa^{\alpha}}
+\chi_{(\ol{q}_d-1)/\alpha}\lpa\frac{\rho}{K_2}\rpa(\rho\vee1)^4
+\exp\lpa-\lpa\frac{\tau\rho}{K_3\ol{B}_N}\rpa^{\alpha/\ol{q}_d}\rpa(\tau\rho\vee1)^4
\rpa
\ol{B}_N^4\sum_{i=1}^N\Lambda_{i}^4
\right\}\\
&=:C_1\varepsilon^{-4}(\log d)^3\lpa\mathbb{I}+\mathbb{II}\rpa
\end{align*}
for any $\varepsilon>0$ and $\tau,\rho\geq0$ with $\tau\rho c_\alpha\ol{B}_N\max_{1\leq i\leq N}\Lambda_i\leq\varepsilon/\log d$, where $C_1,K_1,K_2,K_3>0$ depend only on $\alpha,\ol{q}_d$ and $\Lambda_i:=(\log d)^{(\ol{q}_d-1)/\alpha}\max_{1\leq k\leq d}\ol{B}_N^{q_k-1}\sqrt{\influence_{i}(f_k)}$.  
We apply this inequality with $\tau:=(\log d^2)^{1/\alpha}\{K_1\vee(K_3/K_2)\}$, 
$\rho:=(\log d^2)^{(\ol{q}_d-1)/\alpha}K_2$ and
\[
\varepsilon:=(\log d)^{1/6}\delta_0[\bs{Q}(\bs{X})]^{1/3}+(\log d)^{\mu}\delta_1[\bs{Q}(\bs{X})]^{1/3}+(\log d)\tau\rho c_\alpha\ol{B}_N\max_{1\leq i\leq N}\Lambda_i.
\] 
By construction we have 
\begin{align*}
\mathbb{II}&\lesssim d^{-1}\sum_{i=1}^N\sum_{k=1}^d\ol{B}_N^{4q_k}\influence_{i}(f_k)^2
\leq \max_{1\leq k\leq d}\ol{B}_N^{4q_k}\sum_{i=1}^N\influence_{i}(f_k)^2.
\end{align*}
Therefore, we obtain
\begin{align*}
\Delta_\varepsilon(\bs{Q}(\bs{X}),\bs{Q}(\bs{Y}))
&\lesssim \varepsilon^{-4}(\log d)^{3+4(\ol{q}_d-1)/\alpha}\max_{1\leq k\leq d}\ol{B}_N^{4q_k}\sum_{i=1}^N\influence_i(f_k)^2\\
&\lesssim\varepsilon^{-4}(\log d)^{3+4(\ol{q}_d-1)/\alpha}\delta_1[\bs{Q}(\bs{X})]^2
\leq (\log d)^{3+4\{(\ol{q}_d-1)/\alpha-\mu\}}\delta_1[\bs{Q}(\bs{X})]^{2/3}.
\end{align*}
Since $3+4\{(\ol{q}_d-1)/\alpha-\mu\}\leq\frac{4}{3\alpha}(\ol{q}_d-1)+\frac{5}{3}\leq2\mu+1$ and $(\log d)^{\mu+\frac{1}{2}}\delta_1[\bs{Q}(\bs{X})]^{\frac{1}{3}}<1$, we conclude that
\begin{equation}\label{est-delta1}
\Delta_\varepsilon(\bs{Q}(\bs{X}),\bs{Q}(\bs{Y}))\lesssim (\log d)^{2\mu+1}\delta_1[\bs{Q}(\bs{X})]^{2/3}
\leq (\log d)^{\mu+\frac{1}{2}}\delta_1[\bs{Q}(\bs{X})]^{\frac{1}{3}}.
\end{equation}

Meanwhile, Proposition \ref{prop:normal-gamma} yields
\begin{align*}
\Delta_{\varepsilon}(\bs{Q}(\bs{Y}),Z)
&\lesssim \varepsilon^{-2}(\log d)\left(\delta_0[\bs{Q}(\bs{Y})]+\delta_2[\bs{Q}(\bs{Y})]\right).
\end{align*}
Now, in the present situation, the constants $w_*$, $\ol{v}_N$ and $\ol{\eta}_N$ appearing in Proposition \ref{prop:normal-gamma} satisfy $w_*=w$, $\ol{v}_N\leq2+\ol{A}_N^2/2$ and $\ol{\eta}_N^{-1}\leq\ol{A}_N/2$, so we have
\begin{multline*}
\delta_2[\bs{Q}(\bs{Y})]\leq(\ol{A}_N/2)^{2w\ol{q}_d-1}(\log d)^{2w\ol{q}_d-1}
\max_{1\leq j,k\leq d}\left\{1_{\{q_j< q_k\}}\|Q(f_j;\bs{Y})\|_4\kappa_4(Q(f_k;\bs{Y}))^{1/4}\right.\\
\left.+1_{\{q_j=q_k\}}\sqrt{2\kappa_4(Q(f_k;\bs{Y}))+\left((1+\ol{A}_N^2/4)^{q_k}-1\right)(2q_k)!\mf{c}_{q_k}
\sum_{i=1}^N\influence_i(f_k)^2}\right\}.
\end{multline*}
Moreover, since the sequence $\bs{Y}$ is $(2,r,\ul{\eta}_N(r-1)^{-w})$-hypercontractive up to degree 1 for any $r>2$ by Lemmas \ref{hyper-union}--\ref{gamma-hyper}, we have $\|Q(f_j;\bs{Y})\|_4\lesssim\ol{A}_N^{q_j}\|Q(f_j;\bs{Y})\|_2$ for every $j$. 
Also, since Lemma \ref{moment-psi} yields $\max_{1\leq i\leq N}\|Y_i\|_{\psi_\alpha}\lesssim\ul{\eta}_N^{-1}$, by Lemma \ref{lemma:npr4.3} (with $l=m=4$) we obtain
\[
|\expe{Q(f_k;\bs{X})^4}-\expe{Q(f_k;\bs{Y})^4}|
\lesssim\ol{A}_N^{4q_k}\sum_{i=1}^N\influence_i(f_k)^2
\]
for every $k$. Since we have $\|Q(f_k;\bs{Y})\|_2=\sqrt{q_k!}\|f_k\|_{\ell_2}=\|Q(f_k;\bs{X})\|_2$ for every $k$, it holds that 
$
\delta_2[\bs{Q}(\bs{Y})]\lesssim (\log d)^{2w\ol{q}_d-1}\delta_1[\bs{Q}(\bs{X})].
$ 
Consequently, we obtain
\begin{align*}
\Delta_{\varepsilon}(\bs{Q}(\bs{Y}),Z)
&\lesssim (\log d)^{2/3}\delta_0[\bs{Q}(\bs{X})]^{1/3}+(\log d)^{2(w\ol{q}_d-\mu)}\delta_1[\bs{Q}(\bs{X})]^{1/3}.
\end{align*}
Since $2(w\ol{q}_d-\mu)\leq\frac{2}{3}w\ol{q}_d+\frac{1}{3}\leq\mu+\frac{1}{2}$, we conclude that
\begin{equation}\label{est-delta2}
\Delta_{\varepsilon}(\bs{Q}(\bs{Y}),Z)\lesssim (\log d)^{2/3}\delta_0[\bs{Q}(\bs{X})]^{1/3}+(\log d)^{\mu+\frac{1}{2}}\delta_1[\bs{Q}(\bs{X})]^{1/3}.
\end{equation}

Now, \eqref{est-delta1}--\eqref{est-delta2} imply that
\[
\Delta_{\varepsilon}(\bs{Q}(\bs{X}),Z)
\leq\Delta_{\varepsilon}(\bs{Q}(\bs{X}),\bs{Q}(\bs{Y}))+\Delta_{\varepsilon}(\bs{Q}(\bs{Y}),Z)
\lesssim (\log d)^{2/3}\delta_0[\bs{Q}(\bs{X})]^{1/3}+(\log d)^{\mu+\frac{1}{2}}\delta_1[\bs{Q}(\bs{X})]^{1/3}.
\]
Therefore, Proposition \ref{smoothing} yields
\begin{align*}
&\sup_{x\in\mathbb{R}^d}\left|P(\bs{Q}(\bs{X})\leq x)-P(Z\leq x)\right|\\
&\lesssim (\log d)^{2/3}\delta_0[\bs{Q}(\bs{X})]^{1/3}+(\log d)^{\mu+\frac{1}{2}}\delta_1[\bs{Q}(\bs{X})]^{1/3}+\ul{\sigma}^{-1}\varepsilon\sqrt{\log d}\\
&\lesssim(1+\ul{\sigma}^{-1})\{(\log d)^{2/3}\delta_0[\bs{Q}(\bs{X})]^{1/3}+(\log d)^{\mu+\frac{1}{2}}\delta_1[\bs{Q}(\bs{X})]^{1/3}\}
+\ul{\sigma}^{-1}(\log d)^{(2\ol{q}_d-1)/\alpha+\frac{3}{2}}\max_{1\leq k\leq d}\ol{B}_N^{q_k}\sqrt{\mathcal{M}(f_k)}.
\end{align*}
This completes the proof. 
\end{proof}

\subsection{Proof of Corollaries \ref{coro:rect} and \ref{coro:main}}

Corollary \ref{coro:rect} can be shown in an analogous manner to the proof of \cite[Corollary 5.1]{CCK2017} with applying Theorem \ref{thm:main} instead of \cite[Lemma 5.1]{CCK2017}. 
Corollary \ref{coro:main} is an immediate consequence of Corollary \ref{coro:rect}. \hfill\qed

\subsection{Proof of Theorem \ref{thm:universality}}

\begin{lemma}\label{lemma:transfer}
Let $q\geq2$ and $f:[N]^q\to\mathbb{R}$ be a symmetric function vanishing on diagonals. 
Suppose that the sequence $\bs{X}$ satisfies one of conditions \ref{univ:kurtosis}--\ref{univ:gamma}. Then we have $\kappa_4(Q(f;\bs{X}))\geq0$ and
\begin{equation}\label{aim:transfer}
\mathcal{M}(f)\leq\max_{1\leq r\leq q-1}\|f\star_rf\|_{\ell_2}\leq \frac{1}{q\cdot q!}\sqrt{\kappa_4(Q(f;\bs{X}))}.
\end{equation}
\end{lemma}

\begin{proof}
The first inequality in \eqref{aim:transfer} is a consequence of Eq.(1.9) in \cite{NPR2010aop} (note that they define $\influence_i(f)$ with dividing ours by $(q-1)!$). 
To prove the second inequality in \eqref{aim:transfer}, first we suppose that $\bs{X}$ satisfies condition \ref{univ:kurtosis}. Let $\bs{G}=(G_i)_{i\in\mathbb{N}}$ be a sequence of independent standard normal variables. Then, by \cite[Proposition 3.1]{NPPS2016} we have $\kappa_4(Q(f;\bs{X}))\geq\kappa_4(Q(f;\bs{G}))$. Therefore, \eqref{var-ineq} yields the desired result. 
Next, when $\bs{X}$ satisfies condition \ref{univ:poisson}, the desired result follows from Eq.(5.3) in \cite{DVZ2018}. 
Finally, when $\bs{X}$ satisfies condition \ref{univ:gamma}, the desired result follows from \eqref{var-ineq}. 
Hence we complete the proof. 
\end{proof}

\begin{lemma}\label{kol-moment}
Let $F,G$ be two random variables such that $\|F\|_{\psi_\alpha}\vee\|G\|_{\psi_\alpha}<\infty$ for some $\alpha>0$. Then we have
\[
|\expe{|F|^r}-\expe{|G|^r}|
\leq\frac{2r(\|F\|_{\psi_\alpha}^r+\|G\|_{\psi_\alpha}^r)}{\alpha b}\bs{\Gamma}\lpa\frac{r}{\alpha b}\rpa\sup_{x\in\mathbb{R}}|P(F\leq x)-P(G\leq x)|^b
\]
for any $r\geq1$ and $b\in(0,1)$, where $\bs{\Gamma}$ denotes the gamma function. 
\end{lemma}

\begin{proof}
Set $\rho:=\sup_{x\in\mathbb{R}}|P(F\leq x)-P(G\leq x)|$. By \cite[Theorem 8.16]{Rudin1987} we obtain
\begin{align*}
|\expe{|F|^r}-\expe{|G|^r}|
&\leq r\int_0^\infty|P(|F|>x)-P(|G|>x)|x^{r-1}dx\\
&\leq r\rho^{1-b}\int_0^\infty(\{P(|F|>x)^b+P(|G|>x)^b\}x^{r-1}dx.
\end{align*}
Now, Lemma \ref{psi-tail} and a change of variables yield
\begin{align*}
\int_0^\infty P(|H|>x)^bx^{r-1}dx
\leq2\int_0^\infty x^{r-1}e^{-(x/\|H\|_{\psi_\alpha})^{\alpha b}} dx
=\frac{2\|H\|_{\psi_\alpha}^r}{\alpha b}\bs{\Gamma}\lpa\frac{r}{\alpha b}\rpa
\end{align*}
for $H\in\{F,G\}$. Hence we obtain the desired result. 
\end{proof}

\begin{proof}[Proof of Theorem \ref{thm:universality}]
The inequality $\kappa_4(Q(f;\bs{X}))\geq0$ is proved in Lemma \ref{lemma:transfer}. 
The implication \ref{univ:general} $\Rightarrow$ \ref{univ:joint} $\Rightarrow$ \ref{univ:marginal} is obvious. 
The implication \ref{univ:fourth} $\Rightarrow$ \ref{univ:general} follows from Corollary \ref{coro:main} and Lemma \ref{lemma:transfer}. 

It remains to prove \ref{univ:marginal} $\Rightarrow$ \ref{univ:fourth}. In view of Lemma \ref{kol-moment}, it is enough to prove $\sup_{n\in\mathbb{N}}\max_{1\leq j\leq d_n}(\|Q(f_{n,j};\bs{X})\|_{\psi_{\beta}}+\|Z_{n,j}\|_{\psi_{\beta}})<\infty$ for some $\beta>0$. Set $\alpha_*:=\alpha\wedge1$, where $\alpha$ is the constant appearing in condition \ref{univ:kurtosis}. By assumptions we have $M:=\sup_{i\in\mathbb{N}}\|X_i\|_{\psi_{\alpha_*}}<\infty$, so Lemmas \ref{lem:hom-mom} and \ref{moment-psi} imply that 
\[
\sup_{n\in\mathbb{N}}\max_{1\leq j\leq d_n}\|Q(f_{n,j};\bs{X})\|_{\psi_{\alpha_*/\ol{q}_\infty}}
\leq CM^{\ol{q}_\infty}\sup_{n\in\mathbb{N}}\|f_{n,j}\|_{\ell_2}
\leq CM^{\ol{q}_\infty}\sup_{n\in\mathbb{N}}\|Q(f_{n,j};\bs{X})\|_2<\infty,
\]
where $C>0$ depends only on $\alpha,\ol{q}_\infty$. Finally, since $Z^{(n)}$ is Gaussian, we have
\[
\sup_{n\in\mathbb{N}}\max_{1\leq j\leq d_n}\|Z_{n,j}\|_{\psi_{2}}=\sup_{n\in\mathbb{N}}\max_{1\leq j\leq d_n}\sqrt{\frac{8}{3}\mf{C}_{n,jj}}<\infty.
\] 
Hence we complete the proof. 
\end{proof}

\subsection{Proof of Lemma \ref{lemma:contraction}}\label{sec:proof-contraction}

Let us define the sequence of random variables $(Y_i)_{i=1}^N$ in the same way as in the proof of Theorem \ref{thm:main}. Then, since $\expe{Y_i^4}=3+\frac{3}{2}|\expe{X_i^3}|\leq\frac{9}{2}M$, Lemma \ref{lemma:npr4.3} yields
\[
|\expe{Q(f;\bs{X})^4}-\expe{Q(f;\bs{Y})^4}|\leq C_1M^{q}\mc{M}(f)\|f\|_{\ell_2}^2,
\]
where $C_1>0$ depends only on $q$. Now, since $\expe{Q(f;\bs{X})^2}=q!\|f\|_{\ell_2}^2=\expe{Q(f;\bs{Y})^2}$ and $\sqrt{\kappa_4(Q(f;\bs{Y}))}\geq q\cdot q!\max_{1\leq r\leq q-1}\|f\star_rf\|_{\ell_2}$ by Lemma \ref{lemma:transfer}, we obtain the desired result. \hfill\qed

\subsection{Proof of Proposition \ref{prop:cck-main}}\label{proof:comp-cck}

Extending the probability space if necessary, we take a sequence $(X_i)_{i=1}^\infty$ of i.i.d.~bounded random variables independent of $(\xi_i)_{i=1}^n$ and such that $\expe{X_i}=0$ and $\expe{X_i^2}=\expe{X_i^3}=1$; for example, we may assume
\[
P\lpa X_i=\frac{1\pm\sqrt{5}}{2}\rpa=\frac{\sqrt{5}\mp1}{2\sqrt{5}}.
\]
Then we set $\xi_i^*:=X_i\xi_i$ for $i\in[n]$ and $S_n^*=(S^*_{n,1},\dots,S^*_{n,d}):=n^{-1/2}\sum_{i=1}^n\xi_i^*$. 
The following lemma is a special case of \cite[Theorem 2]{DZ2017}:
\begin{lemma}\label{lemma:dz2017}
Assume $\ul{\sigma}^2:=\min_{1\leq j\leq d}n^{-1}\sum_{i=1}^n\expe{\xi_{ij}^2}>0$.  
Then we have
\[
P\lpa\sup_{t\in\mathbb{R}}\labs P\lpa\max_{1\leq j\leq d}S_{n,j}\leq t\rpa-P\lpa\max_{1\leq j\leq d}S^*_{n,j}\leq t\mid(\bs{\xi}_i)_{i=1}^n\rpa\rabs\geq C_1\varepsilon_n^*\rpa\leq C_1\varepsilon_n^*,
\]
where $C_1>0$ depends only on $\ul{\sigma}$ and
\[
\varepsilon_n^*:=\lpa\frac{\log^5(dn)}{n}\ex{\max_{1\leq j\leq d}\frac{1}{n}\sum_{i=1}^n\xi_{ij}^4}\rpa^{1/6}.
\]
\end{lemma}
\begin{rmk}
Strictly speaking, a direct application of \cite[Theorem 2]{DZ2017} requires us to define the variables $\xi_i^*$ as $\xi_i^*:=X_i(\xi_i-n^{-1/2}S_n)$, but it does not matter to define $\xi_i^*:=X_i\xi_i$ instead. In fact, this theorem is obtained as an application of \cite[Theorem 7]{DZ2017}, which is indeed shown with the latter definition.  
\end{rmk}

\begin{lemma}\label{lemma:mammen}
Assume $\ul{\sigma}^2:=\min_{1\leq j\leq d}n^{-1}\sum_{i=1}^n\expe{\xi_{ij}^2}>0$.  
Then we have
\begin{align*}
&\ex{\sup_{A\in\mathcal{A}^\mathrm{re}(d)}\left|P\lpa S_n^*\in A\mid(\xi_i)_{i=1}^n\rpa-P(Z\in A)\right|}\\
&\leq C_2\left\{(\log d)^{\frac{2}{3}}\ex{\Xi_0^{1/3}}
+(\log d)\ex{\Xi_1^{1/3}}
+\sqrt{\frac{\log^5 d}{n}}\max_{1\leq k\leq d}\ex{\max_{1\leq i\leq n}|\xi_{ij}|}\right\},
\end{align*}
where $C_2>0$ depends only on $\ul{\sigma}$ and 
\[
\Xi_0:=\max_{1\leq j,k\leq d}\labs\frac{1}{n}\sum_{i=1}^n\xi_{ij}\xi_{ik}-\mf{C}_{jk}\rabs,\quad
\Xi_1:=\frac{1}{n}\max_{1\leq j\leq d}\sqrt{\sum_{i=1}^n\xi_{ij}^4}.
\]
\end{lemma}

\begin{proof}
We apply Corollary \ref{coro:rect} with $f_j(i):=n^{-1/2}\xi_{ij}$ conditionally on $(\xi_i)_{i=1}^n$. Then, we can easily check that $\delta_0[\bs{Q}(\bs{X})]=\Xi_0$ and $\delta_1[\bs{Q}(\bs{X})]=(1+\kappa_4(X_1))\Xi_1$. Hence we obtain the desired result. 
\end{proof}

\begin{proof}[Proof of Proposition \ref{prop:cck-main}]
We may assume $(B_n^2\log^6(dn)/n)^{1/6}\leq1$ (otherwise we can take $C=1$). 
Throughout the proof, for two real numbers $a$ and $b$, the notation $a\lesssim b$ means that $a\leq cb$ for some universal constant $c>0$. 

Lemmas \ref{lemma:dz2017}--\ref{lemma:mammen} imply that
\begin{align*}
&\sup_{t\in\mathbb{R}}\labs P(\lpa\max_{1\leq j\leq d}S_{n,j}\leq t\rpa-P\lpa\max_{1\leq j\leq d}Z_j\leq t\rpa\rabs\\
&\leq C_0\left\{(\log d)^{\frac{2}{3}}\ex{\Xi_0^{1/3}}
+(\log dn)\ex{\Xi_1^2}^{1/6}
+\sqrt{\frac{\log^5 d}{n}}\max_{1\leq k\leq d}\ex{\max_{1\leq i\leq n}|\xi_{ij}|}\right\},
\end{align*}
where $C_0>0$ depends only on $\ul{\sigma}$. 
By \cite[Lemma 8]{CCK2015} we have
\[
\ex{\Xi_0}\lesssim \frac{1}{n}\sqrt{(\log d)\max_{1\leq j,k\leq d}\sum_{i=1}^n\ex{\lpa\xi_{ij}\xi_{ik}-\expe{\xi_{ij}\xi_{ik}}\rpa^2}}
+\frac{\log d}{n}\sqrt{\ex{\max_{1\leq i\leq n}\max_{1\leq j\leq d}\lpa\xi_{ij}\xi_{ik}-\expe{\xi_{ij}\xi_{ik}}\rpa^2}}.
\]
The Schwarz inequality and Lemma \ref{psi-moment} yield
\[
\max_{1\leq j,k\leq d}\sum_{i=1}^n\ex{\lpa\xi_{ij}\xi_{ik}-\expe{\xi_{ij}\xi_{ik}}\rpa^2}
\leq\max_{1\leq j\leq d}\sum_{i=1}^n\ex{\xi_{ij}^2}
\lesssim nB_n^2.
\]
\eqref{sum-psi} and Lemmas \ref{max-psi}, \ref{psi-moment}, \ref{product-psi} yield
\[
\left\|\max_{1\leq i\leq n}\max_{1\leq j,k\leq d}\labs\xi_{ij}\xi_{ik}-\expe{\xi_{ij}\xi_{ik}}\rabs\right\|_2\lesssim (\log nd)^2B_n^2.
\]
Hence we obtain
\[
\ex{\Xi_0}\lesssim \sqrt{\frac{\log d}{n}}B_n
+\frac{(\log dn)^2}{n}B_n
\leq2\sqrt{\frac{B_n^2\log dn}{n}}.
\]
Meanwhile, by \cite[Lemma 9]{CCK2015} we have
\[
\ex{\max_{1\leq j\leq d}\sum_{i=1}^n\xi_{ij}^4}
\lesssim \max_{1\leq j\leq d}\sum_{i=1}^n\ex{\xi_{ij}^4}
+(\log d)\ex{\max_{1\leq i\leq n}\max_{1\leq j\leq d}\xi_{ij}^4}.
\]
Therefore, the assumptions of the proposition and Lemmas \ref{max-psi}, \ref{psi-moment}, \ref{product-psi} imply that
\[
\ex{\Xi_1^2}\lesssim\frac{B_n^2}{n}+\frac{B_n^4(\log dn)^5}{n^2}
\leq2\frac{B_n^2}{n}.
\]
Finally, Lemmas \ref{max-psi} and \ref{psi-moment} yield
\[
\sqrt{\frac{\log^5 d}{n}}\max_{1\leq k\leq d}\ex{\max_{1\leq i\leq n}|\xi_{ij}|}
\lesssim \sqrt{\frac{\log^5 d}{n}}B_n\log n
=\sqrt{\frac{B_n^2(\log d)^5(\log n)^2}{n}}.
\]
Combining these estimates, we obtain the desired result. 
\end{proof}

\subsection{Proof of Proposition \ref{prop:koike-main}}

\begin{lemma}\label{fourth-trace}
Let $\bs{X}=(X_i)_{i=1}^N$ be a sequence of independent centered random variables with unit variance and such that $M:=\max_{1\leq i\leq N}\expe{X_i^4}<\infty$. 
Also, let $f:[N]^2\to\mathbb{R}$ be a symmetric function vanishing on diagonals. Then we have 
\[
|\kappa_4(Q(f;\bs{X}))|\leq C\left\{(1+M)\|f\|_{\ell_2}^2\mc{M}(f)+\trace([f]^4)\right\},
\] 
where $C>0$ is a universal constant. 
\end{lemma}

\begin{proof}
By Proposition 3.1 and Eq.(3.1) in \cite{deJong1987} we have 
\[
|\expe{Q(f;\bs{X})^4}-6G_{\text{V}}|\leq G_{\text{I}}+18G_{\text{II}}+24|G_{\text{IV}}|,
\]
where
\begin{align*}
G_{\text{I}}&=
2^3\sum_{(i,j)\in\Delta_2^N}f(i,j)^4\expe{X_i^4}\expe{X_j^4}
,&
G_{\text{II}}&=
2^3\sum_{(i,j,k)\in\Delta_3^N}f(i,j)^2f(i,k)^2\expe{X_i^4}
,\\
G_{\text{IV}}&=
2\sum_{(i,j,k,l)\in\Delta_4^N}f(i,j)f(i,k)f(l,j)f(l,k)
,&
G_{\text{V}}&=
2\sum_{(i,j,k,l)\in\Delta_4^N}f(i,j)^2f(k,l)^2.
\end{align*}
Since we have $2\|f\|_{\ell_2}^4-G_{\text{V}}\leq8\|f\|_{\ell_2}^2\mc{M}(f)$, it holds that $|\kappa_4(Q(f;\bs{X}))|\leq|\expe{Q(f;\bs{X})^4}-6G_{\text{V}}|+48\|f\|_{\ell_2}^2\mc{M}(f)$. 
Meanwhile, a straightforward computation yields
\begin{align*}
\trace([f]^4)
&=\sum_{(i,j)\in\Delta_2^N}f(i,j)^4
+2\sum_{(i,j,k)\in\Delta_3^N}f(i,k)^2f(j,k)^2
+\sum_{(i,j,k,l)\in\Delta_4^N}f(i,k)f(j,k)f(i,l)f(j,l).
\end{align*}
Hence we obtain
\[
|\expe{Q(f;\bs{X})^4}-6G_{\text{V}}|
\leq C_1\left\{(1+M)\left(\sum_{i,k=1}^Nf(i,k)^4+\sum_{(i,j,k)\in\Delta_3^N}f(i,k)^2f(j,k)^2\right)+\trace([f]^4)\right\},
\]
where $C_1>0$ is a universal constant. Since it holds that
\[
\max\left\{\sum_{(i,j)\in\Delta_2^N}f(i,j)^4,\sum_{(i,j,k)\in\Delta_3^N}f(i,k)^2f(j,k)^2\right\}\leq\|f\|_{\ell_2}^2\mc{M}(f),
\]
we obtain the desired result.
\end{proof}

\begin{proof}[Proof of Proposition \ref{prop:koike-main}]
The desired result immediately follows from Corollary \ref{coro:main} and Lemma \ref{fourth-trace}. 
\end{proof}


%
%
%

\subsection{Proof of Proposition \ref{boot-dist}}\label{sec:boot-dist}

Define the $n_1\times n_2$ matrix $\Xi_n(\theta)$ by 
$\Xi_n(\theta)=
(\frac{1}{2}\Delta_i^nX^1K^{ij}_\theta\Delta_j^nX^2)_{i,j},$ 
and set
\[
\widetilde{\Xi}_n(\theta)=
\left(
\begin{array}{cc}
O  &  \Xi_n(\theta) \\
\Xi_n(\theta)^\top  &  O   
\end{array}
\right).
\]
Note that $U_n^*(\theta)=\boldsymbol{w}^\top\widetilde{\Xi}_n(\theta)\boldsymbol{w}$ with $\boldsymbol{w}=((w^1_i)_{i=1}^{n_1},(w^2_j)_{j=1}^{n_2})^\top$. Hence, by Proposition \ref{prop:koike-main}, it suffices to prove 
\begin{align}
&\log^2(\#\mathcal G_n)\cdot\sqrt{n}\max_{\theta,\theta'\in\mathcal{G}_n}\left|\expe{U_n(\theta)U_n(\theta')}-\expe{U^*_n(\theta)U^*_n(\theta')\mid X}\right|\to^p0,
\label{hry:cov}\\
&\log^5(\#\mathcal G_n)\max_{\theta\in\mathcal G_n}\sqrt{\trace\left(\widetilde{\Xi}_n(\theta)^4\right)}\to^p0,
\label{hry:fourth}\\
&\log^5(\#\mathcal G_n)\cdot\max_{\theta\in\mathcal G_n}\sqrt{\max_{1\leq i\leq n_1}n\sum_{j=1}^{n_2}(\Delta^n_iX^1)^2(\Delta^n_jX^2)^2K^{ij}_\theta+\max_{1\leq j\leq n_2}n\sum_{i=1}^{n_1}(\Delta^n_iX^1)^2(\Delta^n_jX^2)^2K^{ij}_\theta}\to^p0.\label{hry:influence}
\end{align}
\eqref{hry:cov} and \eqref{hry:fourth} are established in the proof of \cite[Proposition B.8]{Koike2017stein} under the current assumptions. Moreover, as in bounding the quantity $\expe{R_{n,1}^*}$ in the proof of \cite[Proposition B.8]{Koike2017stein}, we deduce for any $p\geq1$
\begin{align*}
&\ex{\left|\max_{\theta\in\mathcal G_n}\max_{i}n\sum_{j}(\Delta_i^n X^1)^2(\Delta_j^n X^2)^2K_\theta^{ij}\right|^p}
\leq n^p\sum_i\ex{(\Delta_i^n X^1)^{2p}\max_{\theta\in\mathcal G_n}\left|\sum_{j}(\Delta_j^n X^2)^2K_\theta^{ij}\right|^p}\\
&\leq n^p\sum_i\sqrt{\ex{(\Delta_i^n X^1)^{4p}}\ex{\max_{\theta\in\mathcal G_n}\left|\sum_{j}(\Delta_j^n X^2)^2K_\theta^{ij}\right|^{2p}}}
=O\left(n^pr_n^{p-1}(r_n\log(\#\mathcal G_n))^p\right).
\end{align*}
Exchanging $X^1$ and $X^2$, we obtain a similar estimate. Hence \eqref{hry:influence} holds by assumption.\qed

\appendix

\addcontentsline{toc}{section}{Appendix}
\section*{Appendix}

\section{Properties of the $\psi_\alpha$-norm}\label{sec:psi}

In this appendix we collect several properties of the $\psi_\alpha$-norm used in this paper. Let $\alpha$ be a positive number. Recall that the $\psi_\alpha$-norm of a random variable $X$ is defined by
\begin{equation}\label{def:psi-norm}
\|X\|_{\psi_\alpha}:=\inf\{C>0:\expe{\psi_\alpha(|X|/C)}\leq1\},
\end{equation}
where $\psi_\alpha(x):=\exp(x^\alpha)-1$. 
From the definition we can easily deduce the following useful identity:
\begin{equation}\label{eq:psi1}
\|X\|_{\psi_\alpha}=\||X|^\alpha\|_{\psi_1}^{1/\alpha}.
\end{equation}
Using this relation, we can derive the properties of the $\psi_\alpha$-norm from those of the $\psi_1$-norm. This is very convenient because the latter ones are well-studied in the literature. 
For example, since $\|\cdot\|_{\psi_1}$ satisfies the triangle inequality, we have
\begin{equation}\label{sum-psi}
\|X+Y\|_{\psi_\alpha}\leq 2^{1\vee\alpha^{-1}-1}(\|X\|_{\psi_\alpha}+\|Y\|_{\psi_\alpha})
\end{equation}
for any random variables $X,Y$. 
Also, using Young's inequality for products and H\"older's inequality, one can prove $\|X\|_{\psi_1}\leq (\log 2)^{1/p-1}\|X\|_{\psi_p}$ for any random variable $X$ and $p>1$. Consequently, we obtain
\[
\|X\|_{\psi_\alpha}\leq(\log 2)^{1/\beta-1/\alpha}\|X\|_{\psi_\beta}
\]
for any $0<\alpha\leq\beta<\infty$. 
Other useful results can be obtained from \cite[Lemmas 2.2.1--2.2.2]{VW1996}:
\begin{lemma}\label{tail-psi}
Suppose that there are constants $C,K>0$ such that $P(|X|>x)\leq Ke^{-Cx^\alpha}$ for all $x>0$. Then we have $\|X\|_{\psi_\alpha}\leq\left((1+K)/C\right)^{1/\alpha}$. 
\end{lemma}
\begin{lemma}\label{max-psi}
There is a universal constant $K>0$ such that
\begin{equation*}
\left\|\max_{1\leq j\leq d}|X_j|\right\|_{\psi_\alpha}\leq K^{1/\alpha}(\log (d+1))^{1/\alpha}\max_{1\leq j\leq d}\|X_j\|_{\psi_\alpha}
\end{equation*}
for any $\alpha>0$ and random variables $X_1,\dots,X_d$. 
\end{lemma}

It is also easy to check that $\|X\|_{\psi_\alpha}$ attains the infimum in \eqref{def:psi-norm} if $\|X\|_{\psi_\alpha}<\infty$. That is, $\expe{\psi_\alpha(|X|/\|X\|_{\psi_\alpha})}\leq1$. 
Therefore, the Markov inequality yields the following converse of Lemma \ref{tail-psi}:
\begin{lemma}\label{psi-tail}
If $\|X\|_{\psi_\alpha}<\infty$, we have $P(|X|\geq x)\leq 2e^{-(x/\|X\|_{\psi_\alpha})^\alpha}$ for every $x>0$. 
\end{lemma}

Next we investigate the relation between the $\psi_\alpha$-norm and moment growth. First, \cite[Lemma A.1]{Dirksen2015} yields the following result: 
\begin{lemma}\label{moment-tail}
If there is a constant $A>0$ such that $\|X\|_p\leq Ap^{1/\alpha}$ for all $p\geq1$, then 
$
P(|X|\geq x)\leq e^{1/\alpha}e^{-(\alpha e)^{-1}(x/A)^\alpha}
$ 
for every $x>0$. 
\end{lemma}
Combining Lemma \ref{moment-tail} with Lemma \ref{tail-psi}, we obtain the following result:
\begin{lemma}\label{moment-psi}
Suppose that there is a constant $A>0$ such that $\|X\|_p\leq Ap^{1/\alpha}$ for all $p\geq1$. Then we have $\|X\|_{\psi_\alpha}\leq\left((1+e^{1/\alpha})\alpha e\right)^{1/\alpha}A$. 
\end{lemma}
Lemma \ref{psi-tail} and \cite[Lemma A.2]{Dirksen2015} yield the following converse of Lemma \ref{moment-psi}:
%
\begin{lemma}\label{psi-moment}
For all $p\geq1$, it holds that $\|X\|_p\leq c_\alpha\|X\|_{\psi_\alpha}p^{1/\alpha}$ with 
$
c_\alpha:=e^{1/2e-1/\alpha}\alpha^{-1/\alpha}\max\left\{1,2\sqrt{\frac{2\pi}{\alpha}}e^{\alpha/12}\right\}.
$
\end{lemma}
Finally, we have the following H\"older type inequality for the $\psi_\alpha$-norm:
\begin{lemma}[\cite{KC2018}, Proposition S.3.2]\label{product-psi}
Let $X_1,X_2$ be two random variables such that $\|X_1\|_{\psi_{\alpha_1}}+\|X_2\|_{\psi_{\alpha_2}}<\infty$ for some $\alpha_1,\alpha_2>0$. Then we have 
$
\|X_1X_2\|_{\psi_\alpha}\leq\|X_1\|_{\psi_{\alpha_1}}\|X_2\|_{\psi_{\alpha_2}},
$ 
where $\alpha>0$ is defined by the equation $1/\alpha=1/\alpha_1+1/\alpha_2$. 
\end{lemma}

\section{Proof of Lemma \ref{lem:hom-mom}}\label{sec:hom-mom}

\begin{lemma}[Symmetrization]\label{lemma:symmetrization}
Let $(\xi_i)_{i=1}^N$ be a sequence of independent centered random variables. 
Also, let $(\epsilon_i)_{i=1}^N$ be a sequence of i.i.d.~Rademacher variables independent of $(\xi_i)_{i=1}^N$. 
Then, for any $p\geq1$ we have
\[
\left\|\sum_{i=1}^N\xi_i\right\|_p
\leq 2\left\|\sum_{i=1}^N\epsilon_i\xi_i\right\|_p.
\]
\end{lemma}

\begin{proof}
See \cite[Lemma 2.3.1]{VW1996}.
\end{proof}

\begin{lemma}[Strong domination]\label{lemma:domination}
Let $(\xi_i)_{i=1}^N$ and $(\theta_i)_{i=1}^N$ be two sequences of independent symmetric random variables. 
Suppose that there is an integer $k>0$ such that $P(|\xi_i|>t)\leq k P(|\theta_i|>t)$ for all $i\in[N]$ and $t>0$. Then, for any $p\geq1$ and $a_1,\dots,a_N\in\mathbb{R}$ we have
\[
\left\|\sum_{i=1}^Na_i\xi_i\right\|_p
\leq (2k)^{1/p}k\left\|\sum_{i=1}^Na_i\theta_i\right\|_p.
\]
\end{lemma}

\begin{proof}
This lemma is a consequence of Theorem 3.2.1 and Corollary 3.2.1 in \cite{KW1992}.
\end{proof}

\begin{lemma}\label{lemma:weibull}
Let $(\xi_i)_{i\in\mathbb{N}}$ be a sequence of independent copies of a symmetric random variable $\xi$ satisfying $P(|\xi|\geq t)=e^{-|t|^\alpha}$ for every $t\geq0$ and some $0<\alpha\leq2$. Then, there is a constant $C_\alpha>0$ which depends only on $\alpha$ such that 
\[
\left\|\sum_{i=1}^Na_i\xi_i\right\|_p\leq C_\alpha p^{1/\alpha}\sqrt{\sum_{i=1}^Na_i^2}
\]
for any $p\geq2$, $N\in\mathbb{N}$ and $a_1,\dots,a_N\in\mathbb{R}$. 
\end{lemma}

\begin{proof}
We separately consider the following two cases. 

\noindent\ul{Case 1}: $\alpha\leq1$. In this case, the function $t\mapsto \log P(|\xi|\geq t)$ is convex, so Theorem 1.1 in \cite{HMSO1997} yields
\[
\left\|\sum_{i=1}^Na_i\xi_i\right\|_p\leq K\left\{\lpa\sum_{i=1}^N|a_i|^p\|\xi_i\|_p^p\rpa^{1/p}+\sqrt{p\sum_{i=1}^N|a_i|^2\|\xi_i\|_2^2}\right\},
\]
where $K>0$ is a universal constant. Now, by Lemmas \ref{tail-psi} and \ref{psi-moment} there is a constant $K_\alpha>0$ which depends only on $\alpha$ such that $\|\xi\|_r\leq K_\alpha r^\alpha$ for all $r\geq1$. Hence we obtain
\begin{align*}
\left\|\sum_{i=1}^Na_i\xi_i\right\|_p\leq K\cdot K_\alpha\left\{p^{1/\alpha}\lpa\sum_{i=1}^N|a_i|^p\rpa^{1/p}+2^\alpha\sqrt{p\sum_{i=1}^N|a_i|^2}\right\}.
\end{align*}
Since $\alpha\leq2$ and $p\geq2$, we have $\lpa\sum_{i=1}^N|a_i|^p\rpa^{1/p}\leq\sqrt{\sum_{i=1}^N|a_i|^2}$ and $\sqrt{p}\leq p^{1/\alpha}$, where the former follows from the inequality $(x+y)^{2/p}\leq x^{2/p}+y^{2/p}$ holding for any $x,y\geq0$. Thus we obtain the desired result. 

\noindent\ul{Case 2}: $1<\alpha\leq2$. In this case, the function $t\mapsto \log P(|\xi|\geq t)$ is concave, so an application of the Gluskin-Kwapie\'n inequality yields
\[
\left\|\sum_{i=1}^Na_i\xi_i\right\|_p\leq K'\left\{\|\xi\|_p\lpa\sum_{i=1}^N|a_i|^\beta\rpa^{1/\beta}+\|\xi_i\|_2\sqrt{p\sum_{i=1}^N|a_i|^2}\right\},
\]
where $\beta>1$ is a constant such that $\alpha^{-1}+\beta^{-1}=1$ and $K'>0$ is a universal constant (see page 17 of \cite{HMSO1997}). Thus we obtain
\[
\left\|\sum_{i=1}^Na_i\xi_i\right\|_p\leq K'\cdot K_\alpha\left\{p^{1/\alpha}\lpa\sum_{i=1}^N|a_i|^\beta\rpa^{1/\beta}+2^\alpha p^{1/\alpha}\sqrt{\sum_{i=1}^N|a_i|^2}\right\}.
\]
Now, since $1<\alpha\leq2$, we have $\beta\geq2$. Therefore, we obtain $\lpa\sum_{i=1}^N|a_i|^\beta\rpa^{1/\beta}\leq\sqrt{\sum_{i=1}^N|a_i|^2}$ analogously to the above. This yields the desired result. 
\end{proof}

\begin{lemma}\label{lemma:sum-psi}
Let $(\zeta_i)_{i=1}^N$ be a sequence of independent centered random variables such that $M:=\max_{1\leq i\leq N}\|\zeta_i\|_{\psi_\alpha}<\infty$ for some $0<\alpha\leq2$. Then we have
\[
\left\|\sum_{i=1}^Na_i\zeta_i\right\|_p\leq K_\alpha Mp^{1/\alpha}\sqrt{\sum_{i=1}^Na_i^2}
\]
for any $p\geq1$ and $a_1,\dots,a_N\in\mathbb{R}$, where $K_\alpha>0$ depends only on $\alpha$.
\end{lemma}

\begin{proof}
Thanks to Lemma \ref{lemma:symmetrization}, it suffices to consider the case that $\zeta_i$ is symmetric for all $i$. 
Let $(\xi_i)_{i\in\mathbb{N}}$ be a sequence of independent copies of a symmetric random variable $\xi$ satisfying $P(|\xi|\geq t)=e^{-|t|^\alpha}$ for every $t\geq0$. Then, by Lemma \ref{psi-tail} we have $P(|\zeta_i|>t)\leq2e^{-(t/\|\zeta_i\|_{\psi_\alpha})^\alpha}=2P(\|\zeta_i\|_{\psi_\alpha}|\xi_i|>t)$ for any $t>0$ and $i\in\mathbb{N}$, so Lemma \ref{lemma:domination} yields
\[
\left\|\sum_{i=1}^Na_i\zeta_i\right\|_p
\leq2^{2/p+1}\left\|\sum_{i=1}^Na_i\|\zeta_i\|_{\psi_\alpha}\xi_i\right\|_p
\leq8\left\|\sum_{i=1}^Na_i\|\zeta_i\|_{\psi_\alpha}\xi_i\right\|_p.
\]
Now, Lemma \ref{lemma:weibull} implies that
\[
\left\|\sum_{i=1}^Na_i\|\zeta_i\|_{\psi_\alpha}\xi_i\right\|_p
\leq C_\alpha p^{1/\alpha}\sqrt{\sum_{i=1}^Na_i^2\|\zeta_i\|_{\psi_\alpha}^2}
\leq C_\alpha p^{1/\alpha}M\sqrt{\sum_{i=1}^Na_i^2},
\]
where $C_\alpha>0$ depends only on $\alpha$. This completes the proof.
\end{proof}

\begin{lemma}\label{decoupled-mom}
Let $(X_{i,j})_{1\leq i\leq N,1\leq j\leq q}$ be an array of independent centered random variables. Suppose that $M:=\max_{1\leq i\leq N,1\leq j\leq q}\|X_{i,j}\|_{\psi_\alpha}<\infty$ for some $0<\alpha\leq2$. Then we have
\[
\left\|\sum_{i_1,\dots,i_q=1}^Nf(i_1,\dots,i_q)X_{i_1,1}\cdots X_{i_q,q}\right\|_p
\leq K_{\alpha}^qp^{q/\alpha}M^q\|f\|_{\ell_2}
\]
for any $p\geq2$ and function $f:[N]^q\to\mathbb{R}$, where $K_{\alpha}>0$ is the constant appearing in Lemma \ref{lemma:sum-psi}. 
\end{lemma}

\begin{proof}
We prove the claim by induction on $q$. When $q=1$, it is a direct consequence of Lemma \ref{lemma:sum-psi}. 
Next, assume $q\geq2$ and suppose that the claim holds true for $q-1$. Then, by the assumption of induction we have 
\begin{align*}
&\left\|\sum_{i_1,\dots,i_p=1}^Nf(i_1,\dots,i_q)X_{i_1,1}\cdots X_{i_q,q}\right\|_p\\
&=\left\|\left\{\ex{\labs\sum_{i_1,\dots,i_{q-1}=1}^N\lpa\sum_{i_q=1}^Nf(i_1,\dots,i_q)X_{i_q,q}\rpa X_{i_1,1}\cdots X_{i_{q-1},q-1}\rabs^p\mid X_{1,q},\dots,X_{N,q}}\right\}^{1/p}\right\|_p\\
&\leq K_{\alpha}^{q-1}p^{(q-1)/\alpha}M^{q-1}\left\|\sqrt{\sum_{i_1,\dots,i_{q-1}=1}^N\lpa\sum_{i_q=1}^Nf(i_1,\dots,i_q)X_{i_q,q}\rpa^2}\right\|_p.
\end{align*}
Moreover, it holds that
\begin{align*}
&\left\|\sqrt{\sum_{i_1,\dots,i_{q-1}=1}^N\lpa\sum_{i_q=1}^Nf(i_1,\dots,i_q)X_{i_q,q}\rpa^2}\right\|_p
=\left\|\sum_{i_1,\dots,i_{q-1}=1}^N\lpa\sum_{i_q=1}^Nf(i_1,\dots,i_q)X_{i_q,q}\rpa^2\right\|_{p/2}^{1/2}\\
&\leq \sqrt{\sum_{i_1,\dots,i_{q-1}=1}^N\left\|\sum_{i_q=1}^Nf(i_1,\dots,i_q)X_{i_q,q}\right\|_{p}^2}\quad(\text{Minkowski's inequality})\\
&\leq K_\alpha p^{1/\alpha}M\|f\|_{\ell_2}\quad(\text{Lemma \ref{lemma:sum-psi}}).
\end{align*}
Hence we obtain the claim of the lemma. 
\end{proof}

\begin{proof}[Proof of Lemma \ref{lem:hom-mom}]
The claim is an immediate consequence of \cite[Theorem 1]{dlPMS1995}, \cite[Theorem 8.16]{Rudin1987} and Lemma \ref{decoupled-mom}.
\end{proof}

\section*{Acknowledgements}

The author is grateful to the participants at the Osaka Probability Seminar on November 28, 2017 for insightful comments which motivated the author to write this paper. 
The author also thanks Professor Giovanni Peccati for having indicated that the same type bound as Corollary \ref{coro:wass} has already appeared in \cite[Theorem 3.1]{APP2016}.  
This work was supported by JST CREST and JSPS KAKENHI Grant Numbers JP16K17105, JP17H01100, JP18H00836.

{\small
\renewcommand*{\baselinestretch}{1}\selectfont
\addcontentsline{toc}{section}{References}

}

\end{document}